\documentclass{article}

\usepackage{latexsym, fullpage, hyperref}
\usepackage[english]{babel}
\usepackage[latin1]{inputenc}
\usepackage[usenames,dvipsnames]{color}
\usepackage{amsmath, amsfonts,amssymb,mathtools,amsthm,mathrsfs}
\usepackage{bbm}
\usepackage[round]{natbib}
\usepackage{soul}
\usepackage[draft]{todonotes}

\newtheorem{proposition}{Proposition}[section]
\newtheorem{theorem}[proposition]{Theorem}
\newtheorem{lemma}[proposition]{Lemma}

\newtheorem{assumption}[proposition]{Assumption}
\theoremstyle{definition}

\newtheoremstyle{step}{3pt}{0pt}{}{}{\bf}{}{.5em}{}
\theoremstyle{step} 

\DeclareMathAlphabet{\mathpzc}{OT1}{pzc}{m}{it}

\providecommand{\N}{{\ensuremath{\mathbb{N}}}}
\providecommand{\R}{{\ensuremath{\mathbb{R}}}}
\newcommand{\E}{{\ensuremath{\mathbb{E}}}}
\renewcommand{\P}{{\ensuremath{\mathbb{P}}}}
\providecommand{\1}{{\ensuremath{\mathbbm{1}}}}

\providecommand{\Var}{\operatorname{Var}}

\newcommand{\cadlag}{c\`adl\`ag}

\providecommand{\eps}     {{\ensuremath{\varepsilon}}}

\providecommand{\D}{{\ensuremath{\mathcal{D}}}}
\newcommand{\BIGOP}[1]{\mathop{\mathchoice%
    {\raise-0.22em\hbox{\huge $#1$}}%
    {\raise-0.05em\hbox{\Large $#1$}}{\hbox{\large $#1$}}{#1}}}
\newcommand{\bigtimes}{\BIGOP{\times}}

\providecommand{\Var}{\operatorname{Var}}

\providecommand{\betab}{{\ensuremath{\bar{\beta}}}}

\newcommand{\sgn}{\operatorname{sgn}}

\begin{document}
\title{Costly defense traits in structured populations}
\author{
  Martin Hutzenthaler\thanks{Research supported by the DFG in the Priority Program
    ``Probabilistic Structures in Evolution'' (SPP 1590)} ,
  Felix Jordan\footnotemark[1] ,
  and Dirk Metzler\footnotemark[1]}

\maketitle
\makeatletter
\let\@makefnmark\relax
\let\@thefnmark\relax
\@footnotetext{\emph{AMS 2010 subject classification:} 60K35, 92D25}
\@footnotetext{\emph{Key words and phrases:} costly defense, altruism, group selection, kin selection,
  interacting Wright--Fisher diffusions, local competition, extinction, survival,
  Lotka--Volterra equations, McKean--Vlasov limit, many-demes limit, host--parasite, predator--prey,
  parasite defense, slave rebellion, slavemaker ants}
\makeatother

\begin{abstract}
  We propose a model for the dynamics of frequencies
  of a costly
  defense trait.
  More precisely, we consider Lotka--Volterra-type models involving a
  prey (or host) population consisting of two types and a predator (or parasite)
  population, where one type of prey individuals -- modeling carriers of
  a defense trait -- is more effective in defending against the predators
  but has a weak reproductive disadvantage.
  Under certain assumptions we prove that the relative frequency of
  these  defenders in the total
prey
population converges
to spatially structured Wright--Fisher diffusions with
frequency-dependent migration rates.
For the many-demes limit (mean-field approximation) hereof,
we show that the defense trait goes to fixation/extinction 
if and only if the selective disadvantage is smaller/larger than
an explicit function of the ecological model parameters.
\end{abstract}
%
%
%
%
\section{Introduction} \label{sec:introduction}
If individuals of a prey population defend their conspecifics
when predators arrive, this may be dangerous for the acting individual
and thus be costly in the evolutionary sense of decreasing
the expected number of their surviving offspring.
For example, alarm calling has been described as altruistic behavior
\citep{Tamachi87}.
Altruism is defined as a behavior
that decreases
the reproductive success of the actor while increasing the reproductive
success of one or more recipients \citep{Hamilton1964a}.
In most natural systems, non-altruistic individuals benefit from
altruistic individuals without suffering from the fitness 
disadvantage and, thus, have a direct reproductive advantage.
So how can genetically inherited selfless behavior be explained
by natural selection?
This problem has bothered biologists since Charles Darwin, who
reflected the puzzle of sterile social insects such as the worker castes of ants
in his famous book ``The Origin of Species''~\citep{Darwin1859}.

In behavioral biology and game theory there exist several explanations
for the emergence of altruism and cooperation.
The central idea of inclusive fitness theory is that helping direct relatives
benefits the reproductive success of the altruists' genes.
This idea is formalized in Hamilton's rule, which states that traits increase in
frequency if $C < B\cdot R $, where $C$ is the reproductive cost to the actor, 
$B$ is the additional reproductive benefit gained by the recipient, and $R$ is the
relatedness of the recipient, that is the probability of sharing the same allele
by descent, e.g., $1/2$ for two sisters and $1/8$ for two cousins \citep{Hamilton1964a}.
In other words, genes can spread in a population by kin selection if the
inclusive fitness $B\cdot R-C$ is positive.
However, general applicability of Hamilton's rule is controversial;
e.g., the fundamental criticism of inclusive fitness theory in \cite{NowakEtAl2010}
provoked a strong response including a rebuttal from 137 researchers \citep{AbbotEtAl2011}.

Another explanation for the emergence of altruistic behavior
is the intensively debated theory of group selection; see, e.g., \cite{Wade1978}
and \cite{Queller1992}. The central idea is that groups of cooperators grow faster
and, therefore, split earlier or into more groups than groups of defectors 
\citep{TraulsenNowak2006}.
The importance of group selection (or more generally
multilevel selection) in evolution remains controversial
\citep[see, e.g.][]{MaynardSmith1976, GoodnightStevens1997, GoodnightWade2000,
Nowak2006,WestEtAl2007,Traulsen2010, Gardner2015}.
Many scientists consider group selection (or multi-level selection) and inclusive
fitness theory as equivalent \citep[e.g.][]{Marshall2011}, while others 
\cite[e.g.][]{vVG+12} argue that group selection is a more general concept.
This conflict may be due to disagreement on the precise definition of the
variables in Hamilton's rule \citep{GWW11,BO15}.

%
If groups are formed by related individuals, it may seem obvious that group/kin selection
can lead to the evolution of altruistic traits. 
If, however, individuals also compete with their neighbors -- which may include relatives --
for space, food or other resources, altruistic behavior can be very costly
regarding reproductive success.
\citep{WPD92,VD10}.
In this case, kin selection can still be effective if competition works on
a larger spatial scale than altruism or if competition is reduced as the population is growing
or sends out migrants to conquer empty demes \citep{WPD92,Tay92b,AT08,VD10,VDW12b}.
The latter can occur in a meta-population -- that is, a population that is substructured
into demes that are affected by frequent local extinction and recolonization events.
Already \cite{MaynardSmith1964} proposed that a meta-population dynamic can generate between-deme
variation that is required for group selection, and indeed this was subsequently demonstrated
for several theoretical models \citep{Lev70,Esh72,LK74,Wil73,SW78,Tay92b,AT08,MJP+16}.
The mathematical analyses of \cite{Uye79} however showed for a range of possible selection
pressures that group selection can also work in island population models without local
extinctions, assuming however an extreme migration model in which all surviving offspring
are randomly distributed among all islands.

For the eco-evolutionary dynamics in group or kin selection models it is crucial how
the benefit of altruism depends on the frequency of altruists and on other factors.
\cite{SC11} show for some cases of \textit{diminishing returns}, which means that
the benefit per altruist decreases with the total number of altruists,
that kin selection can maintain the co-existence of altruists and cheaters.
Mechanisms that lead to diminishing returns may include feed-back interactions with the
abundance of other species,
for example predators or parasites if the altruistic trait consist in defending other
individuals against these enemies \citep{BBMW09,DuncanEtAl2011,DLvB+12,BLG13,VB19}.

%

In this article we focus on the evolution of a trait of defense against parasites or predators.
Examples of altruistic or at least costly
defense traits include self-sacrificial colony defense in social
insects \citep{HW10,RHR10,ShorterRueppell2012},
costly chemical alarm signaling in aphids \citep{MR04, Mondor2007, Wu2010},
suicidal defense of bacteria against pathogen infection \citep{FukuyoEtAl2012},
transmission-blocking immunity in vertebrates against \textit{Plasmodium} \citep{MMdS+87},
and slave rebellion in ants \citep{PammingerEtAl2014, MJP+16}.
A number of recent studies propose spatially distributed predator--prey
and host-parasite models and study these models via computer simulations 
\citep[see e.g.][]{CominsHassellMay1992,RandKeelingWilson1995,
HaraguchiSasaki2000,RauchSayamaBar-Yam2002, RauchSayamaBar-Yam2003,
GoodnightEtal2008,BestEtAl2011,DLvB+12,BLG13,LG15,BSL19}.
We note that
our models and results may also be applicable to constitutive resistence or defense
traits in  host--parasite systems if alternative strategies such as induced
resistance \citep{BSL19} and parasite tolerance \citep{VB19} are negligible.


We begin with a structured predator--prey model with a prey population consisting of
two types, which we denote in the following as 
defenders
and 
non-defenders.
Regarding population structure, we assume that the habitat of prey and predators
is subdivided into demes.
Prey as well as predators are panmictic within demes and migrate between the demes.
Our model for the interaction of predators and prey in each deme is based on a
Lotka--Volterra model \citep{Lotka1920,Volterra1926} allowing for competition among prey
individuals of the same deme.
Predators in the same deme compete for prey as well as other resources.
The 
defenders
in the prey population have a smaller reproduction rate than
non-defenders
and reduce the growth of the predator population in the same deme.
Thus, 
non-defenders
profit from the abundance of 
defenders
in their deme and surpass them
in fitness.
We note that the behavior of defenders is not exactly altruistic
if
offspring of defenders benefit from fewer predators
and might produce more grandchildren compared to non-defenders.
In our view, however, costly defense traits are closely related
to altruism and -- similarly to altruism -- it is a priori not clear
whether the behavior of costly defense ``pays off''.



We approximate our first model by an asymptotic model of infinitely large deme population sizes.
As common in population genetics \citep{Ewens2004,durrett2008pm4dna} we measure time in units
of $N$ generations, where $N$ is proportional to effective deme population sizes,
and thus obtain Wright--Fisher diffusions maintaining
random fluctuations in trait frequencies -- so-called genetic drift \citep{Kimura68} --
even in the limit of large demes.
%
Moreover we scale migration rates such that the expected number of migrants per generation
and deme is constant;
which is also common in population genetics \citep{durrett2008pm4dna}.
%
It turns out that, as we let $N$ tend to infinity,
the population genetic
(or ``evolutionary'') time scale of $N$ generations separates
from the ecological time scale of predator--prey interactions.
In other words, on the population genetic time scale, the deme population sizes
of prey and predators instantaneously reach their limits according to the Lotka--Volterra
interactions, where these limits depend on the frequency of 
defenders
in the respective deme.
Our main result is that the diffusion approximation of the frequencies of 
defenders
is the well-known Kimura's stepping stone model with negative selection
\citep[e.g.][Chapter 6]{Etheridge2011}
only with the local population sizes replaced by
a function of the local 
defender
frequencies.

Our diffusion approximation of the frequencies of 
defenders
is mathematically hard to analyze
due to the lack of a suitable dual process,
which is due to variation of local effective population sizes.
For this reason we investigate a meta-population setting at which we arrive by
considering uniform migration on $D\in\N$ demes and then letting the number $D$ of demes
tend to infinity.
In this meta-population setting we prove that the defense trait will become fixed
in the entire prey population if the
`cost of the defense trait'
is smaller than the so-called
`benefit of defense'.
This shows that predator--prey dynamics can indeed induce group selection
that
maintains a defense trait that is under negative selection in each deme,
without the need of extinction (and re-colonization) of demes.
Group selection on a trait requires that the frequency of the trait in a deme must
be correlated with the number of migrants produced by the deme \citep{Queller1992}.
This correlation can only be non-zero if the frequency of the trait varies between
the demes.
While migration reduces this variation, the only factor in our model that can increase it
is genetic drift.

The rest of this article is structured as follows:
In section \ref{sec:introduction_mainresults} we introduce our model and specify
our model assumptions. Moreover we state our main results: first, weak convergence
of the frequencies of 
defenders
as the local population sizes converge to infinity
(Theorem \ref{thm:conv.freq.altruist}) and, second, 
long-term fixation/extinction of 
defenders
in a meta-population setting depending
on whether the
`cost of defense' $\alpha$ is smaller than or bigger than the parameter
$\beta$ which we denote as
`benefit of defense'
(Theorem \ref{thm:longtermbehavior}).
Sections \ref{sec:conv_altr_freq}
and \ref{sec:longtermbehavior} are devoted to the proofs
of 
Theorem \ref{thm:conv.freq.altruist} and of
Theorem \ref{thm:longtermbehavior}, respectively.
In section \ref{sec:mckeanvlasov} we prove that diffusion approximation
of the frequencies of 
defenders
with uniform migration on $D\in\N$ demes
converges as $D\to\infty$ to a McKean--Vlasov equation.
Finally, in section \ref{sec:invasion} we consider the many-demes limit
of the $D$-demes equation when initially only a few demes are populated
and prove that the total mass process of this many-demes limit
converges to infinity or to zero in probability depending
on whether the
`cost
of defense' $\alpha$ is smaller than or bigger than the parameter
$\beta$.

\subsection{Notation} \label{sec:introduction_notation}
Throughout this article, we will use the following notation.
We define $[0,\infty]:=[0,\infty)\cup \{\infty\}$.
We will use the conventions that $0^0=1$, $0\cdot\infty=0$, and that for any
$x\in(0,\infty)$ we have that $\tfrac{x}{\infty}=0$ and $\tfrac{x}{0}=\infty$.
For all $x,y\in\R$ we define $x^+:=\max\{x, 0\}$, $\sgn (x):=\1_{x>0}-\1_{x<0}$,
and $x\land y:=\min\{x,y\}$.
We define $\sup(\emptyset):=-\infty$ and $\inf(\emptyset):=\infty$.
For a topological space $(E,\mathcal{E})$ we denote by $\mathcal{B}(E)$ the Borel
sigma-algebra of $(E,\mathcal{E})$.
Moreover we agree on the convention that zero times an undefined expression is set
to zero.
For every countable set $\D$ and every $\sigma=(\sigma_i)_{i\in\D}\in(0,\infty)^\D$
define a function $\|\cdot\|_{\sigma}\colon\R^\D\to[0,\infty]$
by $\R^\D\ni z=(z_i)_{i\in\D} \mapsto \|z\|_\sigma:=\sum_{i\in\D}\sigma_i|z_i|$ and
define $l_{\sigma}^1:=\{z\in\R^{\D}\colon \|z\|_\sigma<\infty\}$.

\section{Main results} \label{sec:introduction_mainresults}
\subsection{Model}\label{sec:model}
We assume that predator (or parasite) and prey (or host) individuals populate
demes given by a countable set $\mathcal{D}$,
 and the prey population consists of 
 defenders and non-defenders.
For all $i\in \mathcal{D}$
let $A^N_t(i)$, $C^N_t(i)$, and $P^N_t(i)$
be
the total numbers of 
defenders, non-defenders,
and predators
in deme $i$ at time $t$ measured in units of $N\in\N$ individuals.
We will consider the large population limit $N\to\infty$.
The total number 
of host/prey individuals in deme $i\in\mathcal{D}$ is denoted as
$H^N_i=A^N_i+C^N_i$.
Let $\lambda,K,\delta,\nu,\gamma,\eta,\rho\in(0,\infty)$.
For every $N\in\N$, let
 $\kappa_H^N,\kappa_P^N,\alpha^N,\beta_H^N,\beta_P^N,\iota_H^N,\iota_P^N\in[0,\infty)$
satisfy $\alpha^N<\lambda$.
We assume that the prey and predator populations interact in each deme according
to a Lotka--Volterra model with growth rate $\lambda$, carrying capacity $K$, 
per-predator death rate $\delta$ for the prey, per-prey growth rate $\eta$,
competition rate $\gamma$, and death rate $\nu$ for the predator.
Furthermore,
we assume that being a defender increases its death rate by
$\alpha^N$,
and -- as effect of the defense behavior --
decreases the birth rate of predators in the same deme by $\rho$.
We further assume that prey (resp.\ predator) individuals migrate at rate
$\kappa_H^N m(i,j)$ 
(resp.\ $\kappa_P^N m(i,j)$)
from deme $i$ to deme $j$
where $m\in[0,\infty)^{\mathcal{D}\times\mathcal{D}}$
is a symmetric stochastic matrix.
In addition we assume that additional births happen at rate
$g_h$ (resp.\ $g_P$)
per prey (resp.\ predator) individual
and that additional deaths happen at the same rate
$g_H$ (resp.\ $g_P$)
per prey (resp.\ predator) individual.
Moreover, in order
to avoid extinction of the prey populations on the ecological time scale,
we assume immigration of prey (resp.\ predator) individuals
 at rate $\iota_H^N$ (resp.\ $\iota_P^N$) into each deme.
This immigration,
however, is only assumed for technical reasons
and does not appear in the diffusion approximation of the 
defender
frequencies.
To keep the analysis simple, we assume that the probability of an immigrating prey to be
defender
is equal to the current frequency of 
defenders.
Summarizing, for every $N\in\N$ the process 
$(N\cdot A^N,N\cdot C^N,N\cdot P^N)$ is a Markov process
with state space $(\N_0^3)^{\mathcal{D}}$ and transition rates
(where $i\in\mathcal{D}$ and
$\forall k\in\mathcal{D}\colon a_k,c_k,p_k\in\N_0$):
\begin{equation}  \begin{split}
  &(a_k,c_k,p_k)_{k\in\mathcal{D}}\to(a_k+1_{k=i},c_k,p_k)_{k\in\mathcal{D}} \colon\quad
  a_i(g_H+\lambda)
  \\
  &(a_k,c_k,p_k)_{k\in\mathcal{D}}\to(a_k-1_{k=i},c_k,p_k)_{k\in\mathcal{D}} \colon\quad
  a_i(g_H+\tfrac{\lambda}{K}\tfrac{a_i+c_i}{N}+\delta \tfrac{p_i}{N}+\alpha^N)
  \\
  &(a_k,c_k,p_k)_{k\in\mathcal{D}}\to(a_k,c_k+1_{k=i},p_k)_{k\in\mathcal{D}} \colon\quad
  c_i(g_H+\lambda)
  \\
  &(a_k,c_k,p_k)_{k\in\mathcal{D}}\to(a_k,c_k-1_{k=i},p_k)_{k\in\mathcal{D}} \colon\quad
  c_i(g_H+\tfrac{\lambda}{K}\tfrac{a_i+c_i}{N}+\delta \tfrac{p_i}{N})
  \\
  &(a_k,c_k,p_k)_{k\in\mathcal{D}}\to(a_k,c_k,p_k+1_{k=i})_{k\in\mathcal{D}} \colon\quad
  p_i(g_P+\eta \tfrac{c_i}{N}+(\eta-\rho)\tfrac{a_i}{N})
  \\
  &(a_k,c_k,p_k)_{k\in\mathcal{D}}\to(a_k,c_k,p_k-1_{k=i})_{k\in\mathcal{D}} \colon\quad
  p_i(g_P+\nu+\gamma \tfrac{p_i}{N})
  \\
  &(a_k,c_k,p_k)_{k\in\mathcal{D}}\to(a_k-1_{k=i}+1_{k=j},c_k,p_k)_{k\in\mathcal{D}} \colon\quad
  a_i\kappa_H^N m(i,j)
  \\
  &(a_k,c_k,p_k)_{k\in\mathcal{D}}\to(a_k,c_k-1_{k=i}+1_{k=j},p_k)_{k\in\mathcal{D}} \colon\quad
  c_i\kappa_H^N m(i,j)
  \\
  &(a_k,c_k,p_k)_{k\in\mathcal{D}}\to(a_k,c_k,p_k-1_{k=i}+1_{k=j})_{k\in\mathcal{D}} \colon\quad
  p_i\kappa_P^N m(i,j)
\end{split}     \end{equation}
According to a Lotka-Volterra modeling approach and according to the usual diffusion
approximation with SDEs (e.g.\ \cite{ShigaShimizu1980}),
 $A^N$, $C^N$ and $P^N$
satisfy approximatively the stochastic differential equations (SDEs)
%
%
%
\begin{equation} \begin{split} \label{eq:ACP}
  A_t^N(i)=
  &
  A_0^N(i)+
  \int_0^t
  \kappa_H^N\sum_{j\in\D} m(i,j)\left(A_s^N(j)-A_s^N(i)\right)
  +A_s^N(i)\left[\lambda\left(1- \tfrac{A_s^N(i)+C_s^N(i)}{K}\right)
  -\delta P_s^N(i)-\alpha^N\right]\,ds
  \\
  &\quad
  +\int_0^t\iota_H^N \tfrac{A_s^N(i)}{A_s^N(i)+C_s^N(i)}\,ds
  +\int_0^t\sqrt{\beta_H^N A_s^N(i)}\,dW_s^{A,N}(i),
  \\
  C_t^N(i)=
  &C_0^N(i)
  +\int_0^t\kappa_H^N\sum_{j\in\D} m(i,j)\left(C_s^N(j)-C_s^N(i)\right)
  +C_s^N(i)\left[\lambda\left(1-\tfrac{A_s^N(i)+C_s^N(i)}{K}\right)
  -\delta P_s^N(i)\right]\,ds
  \\
  &\quad
  +\int_0^t\iota_H^N \tfrac{C_s^N(i)}{A_s^N(i)+C_s^N(i)}\,ds
  +\int_0^t\sqrt{\beta_H^N C_s^N(i)}\,dW_s^{C,N}(i),
  \\
  P_t^N(i)=
  &
  P_0^N(i)
  +\int_0^t\kappa_P^N\sum_{j\in\D} m(i,j)\left(P_s^N(j)-P_s^N(i)\right)
  \,ds
  \\
  &\quad
  +\int_0^tP_s^N(i)\left[-\nu-\gamma P_s^N(i)
  +\eta C_s^N(i)+(\eta-\rho )A_s^N(i)\right]+\iota_P^N\,ds
  +\int_0^t\sqrt{\beta_P^N P_s^N(i)}\,dW_s^{P,N}(i),
\end{split} \end{equation}
where $W^{A,N}(i), W^{C,N}(i), W^{P,N}(i)\colon[0,\infty)\times\Omega\to\R$, $i\in\mathcal{D}$,
are independent standard Brownian motions,
$\beta_H^N:=2\tfrac{g_H+\lambda}{N}$ and 
$\beta_P^N:=\left.2\left(g_P + \left(\nu+\tfrac{\lambda\gamma\cdot(K\eta-\nu)}{
        \lambda\gamma+\delta K\nu}\right)\right)\right/N$.
Note that these settings for $\beta_H^N$ and $\beta_P^N$ are based on the heuristic
assumption that the system is close to the equilibrium of predator--prey
interactions and extend the approximations of \cite{HutzMetz21} to parameter ranges
in which $\lambda$, $\eta$ and $\nu$ are not negligible compared to $g_H$
and $g_P$.

Existence of solutions to \eqref{eq:ACP}, which we assume here,
can be established
in suitable Liggett-Spitzer spaces
if $\mathcal{D}$ is an Abelian group and if $m$ is translation invariant and irreducible;
cf.~Proposition 2.1 in~\cite{HutzenthalerWakolbinger2007}.
Finally, for our analysis we assume that
$\rho<\eta$
and we set
 $a=\tfrac{\lambda\gamma+\delta K\eta}{\delta K\rho}$
and
$b=\tfrac{\delta\rho}{\delta\nu+\lambda\gamma}$
(which appear in the equilibrium state \eqref{eq:def.h_infty.p_infty}
for prey and predators).
\subsection{Computer simulations}
\label{sec:sim}
To illustrate the model and the roles of certain model parameters, we present here
the results of computer simulations for a setting with finite-size populations in
1000 demes.
When nothing else is stated, we used the parameter settings
$\lambda=2$, $K=1000$, $\delta=0.02$, $\alpha:=\alpha^N=0.01$, $g_P=0$, $\eta=0.005$,
$\rho=0.004$, $\nu=1$, $\gamma=0.01$, $\kappa^N_H=\kappa^N_P=0.01$ and
$\iota_{H, \mathrm{defend.}}^N=\iota_{H, \mathrm{non-defend.}}^N=\iota_P^N=10^{-6}$
    without population-size scaling,
that is with $N=1$.
For the parameter $g_H$ we used the values $0$, $0.5$ and $5$, resulting in the
values of $0.008$, $0.01$ and $0.028$ of
$\beta:=\beta_H^NN \frac{\delta\rho}{\delta\nu + \lambda\gamma}$,
whose analogue in the asymptotic model, the $\beta$ defined in
Theorem \ref{thm:conv.freq.altruist}, will turn out to be the
benefit-of-defense parameter that is crucial
for the fixation probability of the defense allele
(Theorem \ref{thm:longtermbehavior}).
We initialized each deme with 1000 prey individuals,
of which a uniformly drawn random fraction $x$ were defenders.
The number of predators in the deme was then initialized with
$\frac{\lambda}{\delta}\cdot\left(1-1/(Kb\cdot(a-x))\right)$, which is
inspired by the equilibrium frequency of the Lotka--Volterra model.

For the simulations we applied a $\tau$-leaping approach \citep{Gillespie2001}.
For this, we chose a time span $\tau$ and iterated simulation steps in which we
simulated for each deme, each deme sub-population (defenders, non-defenders or predator)
of current size $n$ and each type of event (birth, death or emigration)
of per-individual rate $r$ a binomial number with parameters $(n, p=\tau\cdot r\cdot n)$
of the corresponding events to take place in the next time span of length $\tau$.
We chose $\tau$ small enough to make the binomial-distribution
parameter $p=\tau\cdot r\cdot n$ smaller than 0.01 under most conditions.
In most simulations we performed 50,000,000 $\tau$-leap iterations and read out the
total numbers of defenders, non-defenders and predators once every 200,000 iterations.
\begin{figure}
  \centering
  \includegraphics[width=8cm]{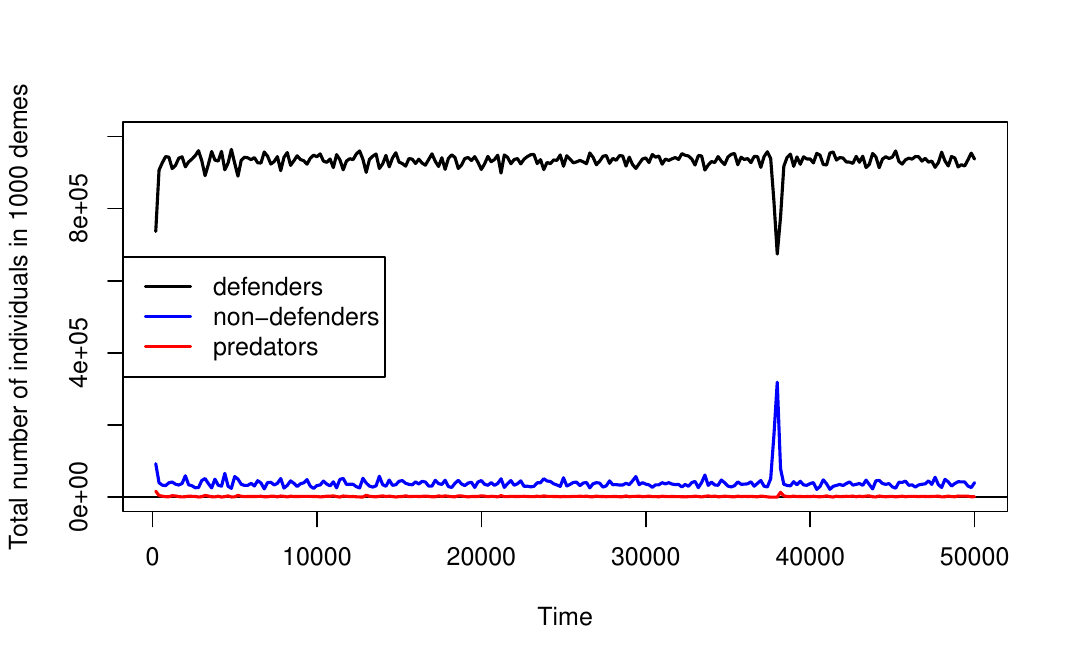}\\
  \includegraphics[width=8cm]{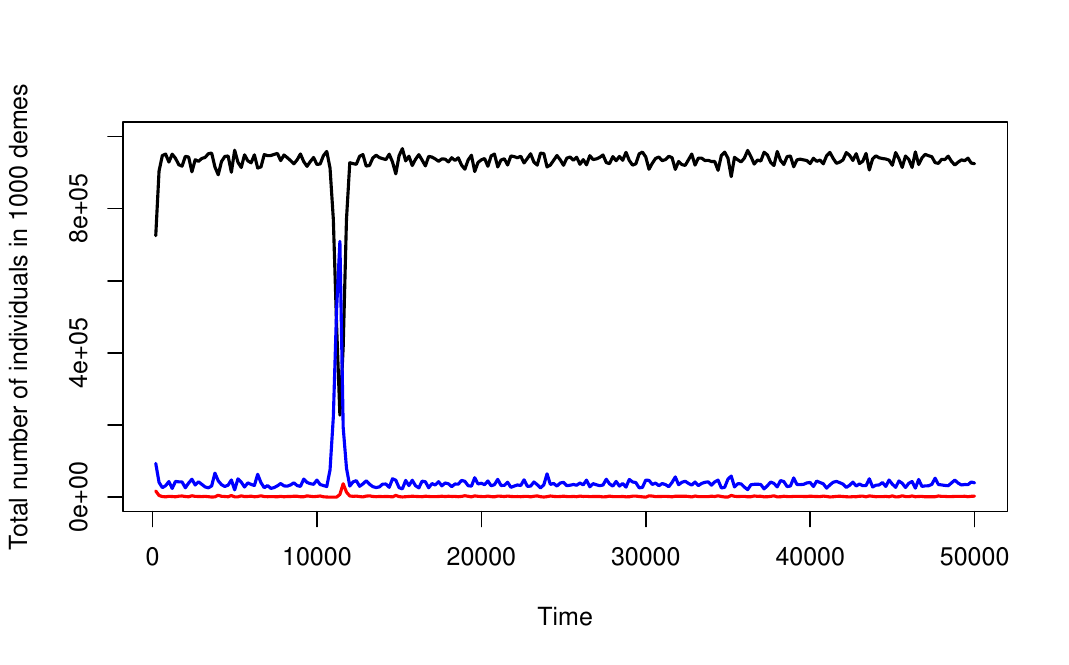}
  \includegraphics[width=8cm]{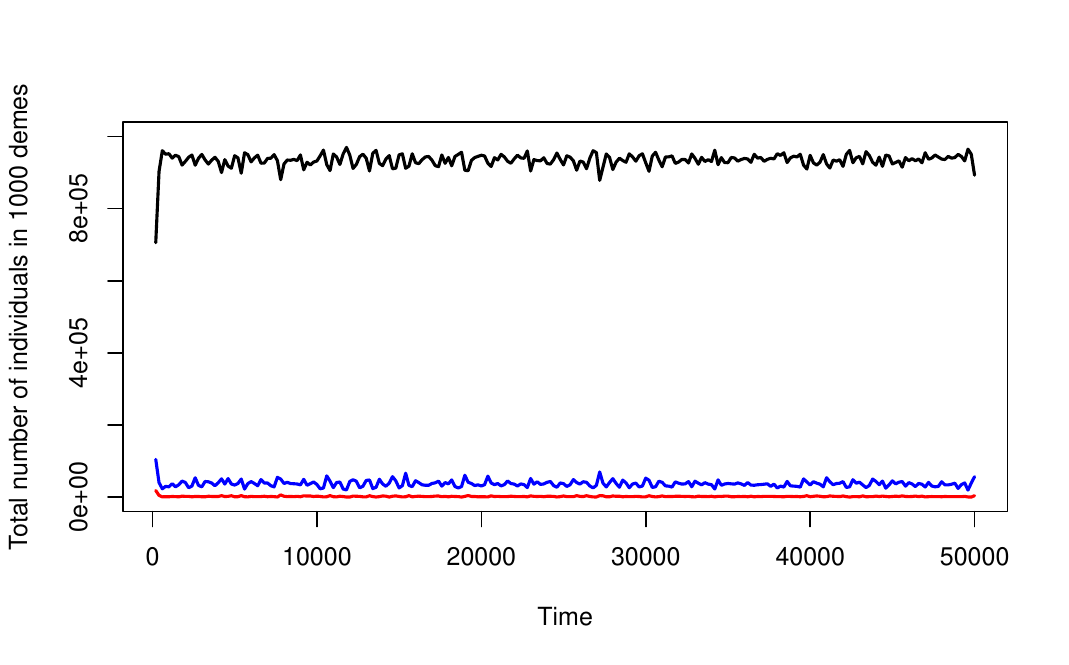}\\
  \includegraphics[width=8cm]{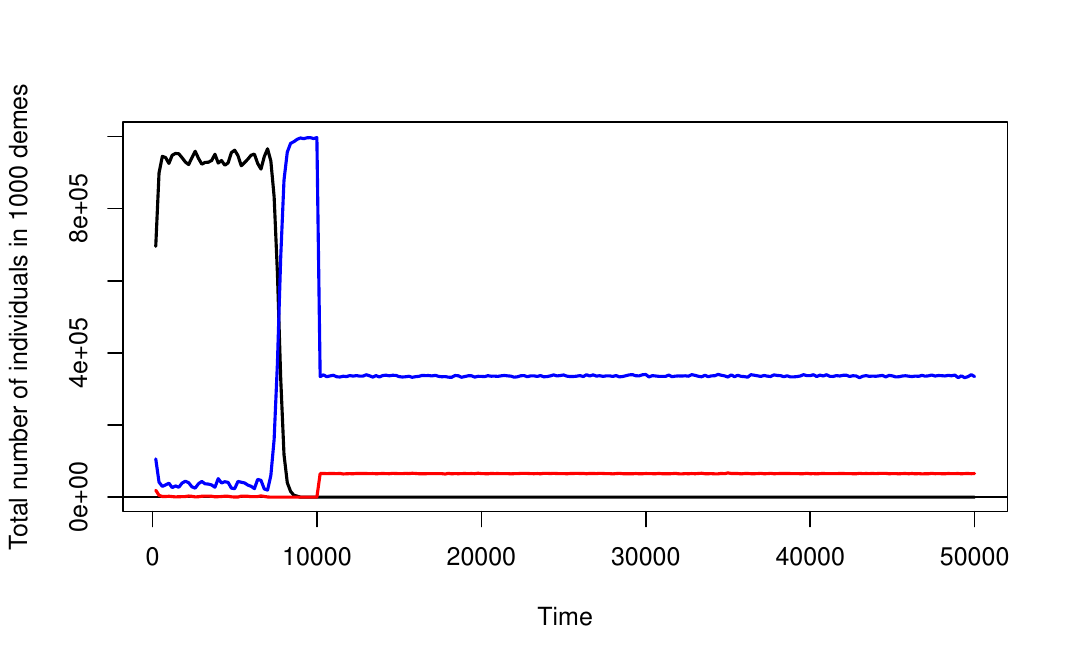}
  \includegraphics[width=8cm]{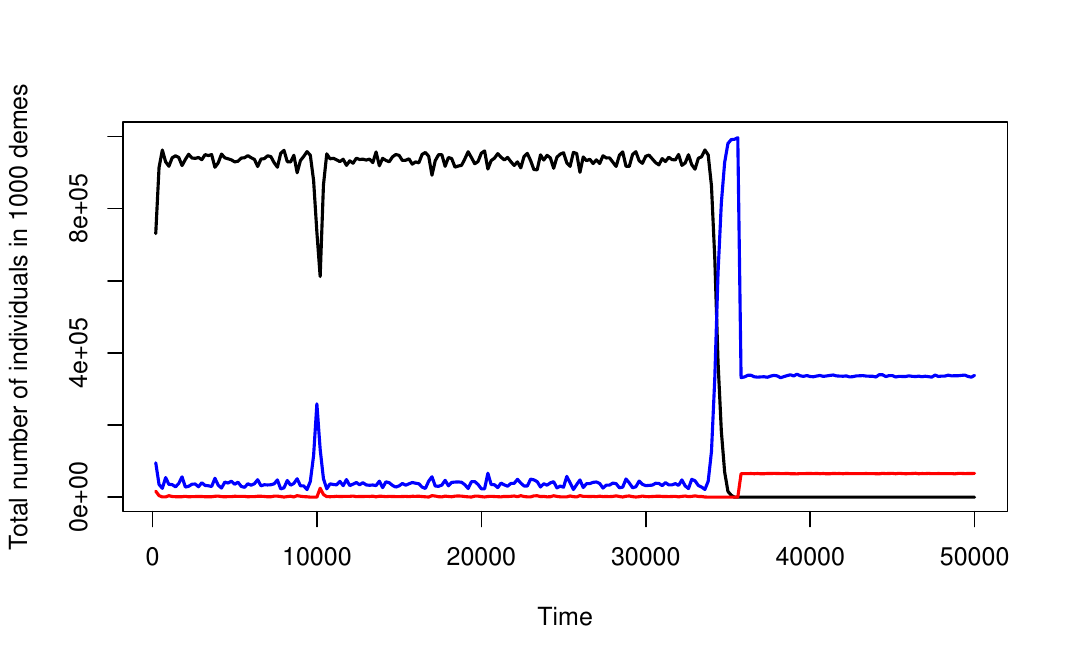}
  \caption{Results from five computer simulation runs with $\alpha=0.01$ and $\beta=0.028$.}
  \label{fig:a01b02}
\end{figure}

Figure \ref{fig:a01b02} shows results from five simulation runs with $\alpha=0.01$ and $\beta=0.028$.
In three of these runs, the defenders became way more frequent than the non-defenders,
which is in accordance with our asymptotic results as $\alpha<\beta$.
At some time points it happened, however, that the predators became rare,
which led to a decrease of the number of defenders and an increase of the
number of non-defenders.
In some cases (Fig.~\ref{fig:a01b02} top and Fig.~\ref{fig:a01b02} middle row left)
this resulted in an increase of the number of predators, which entailed that the
defenders became much more frequent again.
In two simulation runs, however, the defenders went extinct before the predators
returned (Fig.~\ref{fig:a01b02} bottom).
At the end of the simulated time span, no defenders were present in these simulation
runs and non-defenders as well as predators were present in all demes.
The final mean numbers of non-defenders per deme were 334.5 and 337.7 in the two
simulations and the mean number of predators were 65.63 and 66.04.
Note that these values are not far from the limits for the number
$\frac{K\cdot(\delta\nu+\gamma\lambda)}{\lambda\gamma+K\delta\eta}\approx 333.3$ of prey
individuals and the number
$\frac{\lambda K\eta-\lambda\nu}{\lambda\gamma+\delta K\eta}\approx 66.6$ of predators
according to the classical Lotka--Volterra model with within-species competition
(and without migration).

\begin{figure}
  \centering
  \includegraphics[width=10cm]{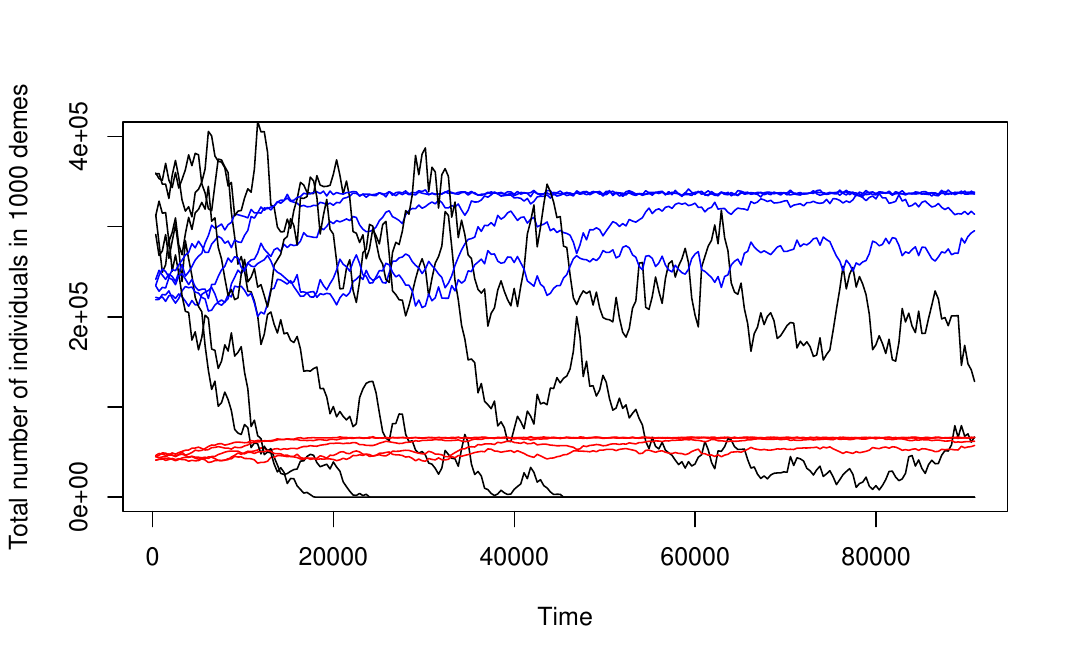}
  \caption{Total frequencies of defenders (black), non-defenders (blue) and predators (red)
    in five computer simulation runs with $\alpha=\beta=0.01$.}
  \label{fig:a01b002}
\end{figure}
Figure \ref{fig:a01b002} shows simulation results with $g_H=0.5$ and thus
$\beta=0.01=\alpha$.
In accordance with the asymptotic result for $\alpha=\beta$, the non-defenders
co-existed with the defenders over a long time span.
In three of the five simulation runs, however, the defenders
went extinct and were not able to re-immigrate and spread during the
simulated time span.

\begin{figure}
  \centering
  \includegraphics[width=10cm]{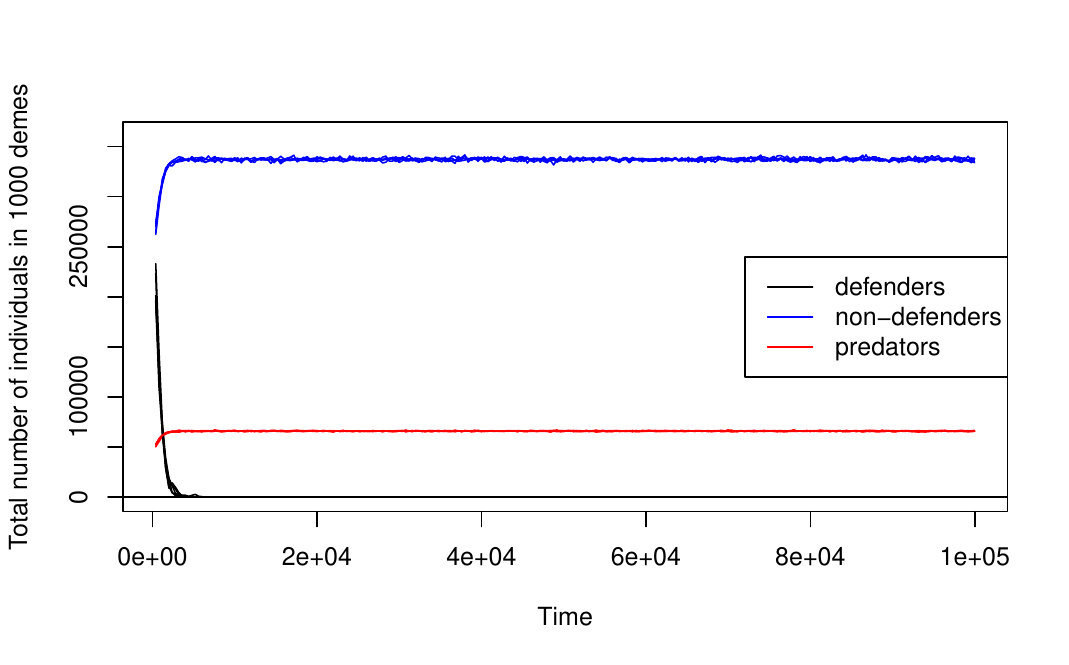}
  \caption{Five computer simulation runs with $\alpha=0.01$ and $\beta=0.008$.}
  \label{fig:a01b0}
\end{figure}
In five simulations with  $\beta=0.008$ and $\alpha=0.01$
the defenders quickly went extinct and were not able to re-immigrate and spread
in the population (Fig.~\ref{fig:a01b0}).
\begin{figure}
  \centering
  \includegraphics[width=10cm]{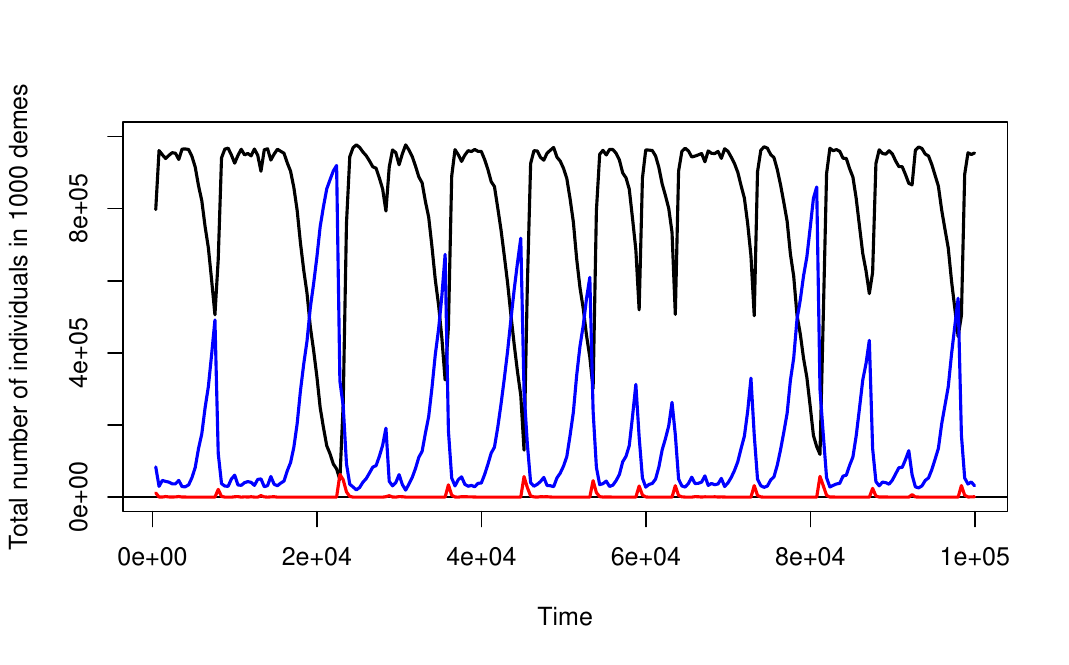}
  \caption{Total frequencies of defenders (black), non-defenders (blue) and predators (red)
    in a computer simulation with $\alpha=0.001$ and $\beta=0.008$.
    For visual clarity only one simulation run is shown but the other four gave
    similar results.
  }
  \label{fig:a001b0}
\end{figure}
This changed, however, when we reduced the cost of defense to $\alpha=0.001<\beta$
(Figure \ref{fig:a001b0}). In these simulations the defenders
became more frequent than the non-defenders and it happened many times that
the predators became rare,
which resulted in an increase in the frequencies of non-defenders and a decrease
in the frequencies of defenders.
But the latter happened more slowly than in simulations with a higher defense cost
$\alpha$, such that the defenders were still abundant when the predators
returned and were able to take over again.

To study the role of between-deme migration we launched simulation runs without it, that is,
with $\kappa_H^N=\kappa_P^N=0$.
In the simulation runs with $\alpha=0.01$ and $\beta=0.008$ or $\beta=0.01$
we observe that demes were in the end of the simulation runs either populated with
defenders and neither non-defenders nor predators or with non-defenders and
predators and no defenders.
The fraction of demes with defenders was 26.14 \% in the simulations with
$\beta=0.008$ and it was 30.3 \% in the simulations with $\beta=0.01$. The average
numbers of individuals in these demes were 992.59 and 994.15, respectively.
In the demes that were populated by non-defenders, the average final prey
population sizes were 394.58 and 385.6 per deme, and the average final predator
per-deme population sizes were 60.28 and 60.96.
Thus, averaged over the five simulation runs with $\beta=0.008$ the total number of defenders
was slightly lower than that of the non-defenders ($2.6\cdot10^5$ vs.\ $2.9\cdot10^5$), whereas
it was higher than the total number of non-defenders for $\beta=0.01$
($3\cdot10^5$ vs.\ $2.7\cdot10^5$).

\begin{figure}
  \centering
  \includegraphics[width=8cm]{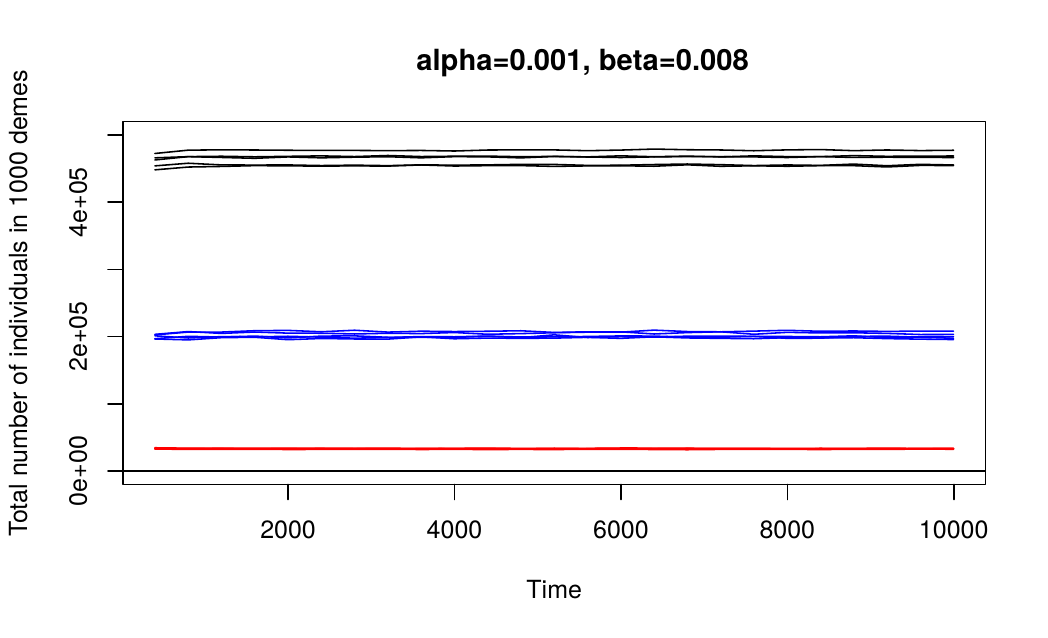}
  \includegraphics[width=8cm]{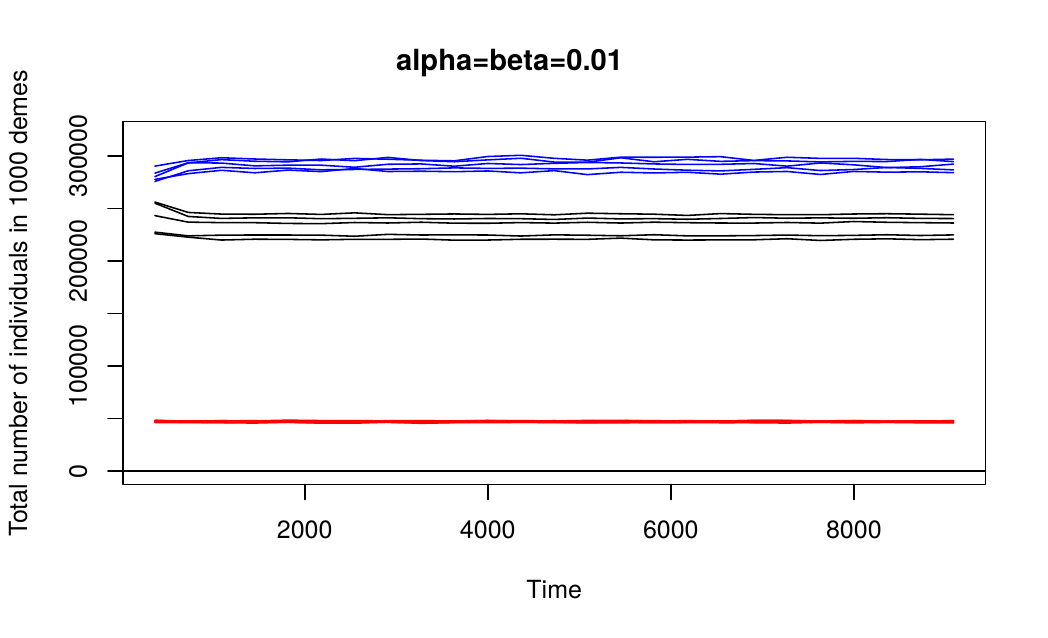}\\
  \includegraphics[width=8cm]{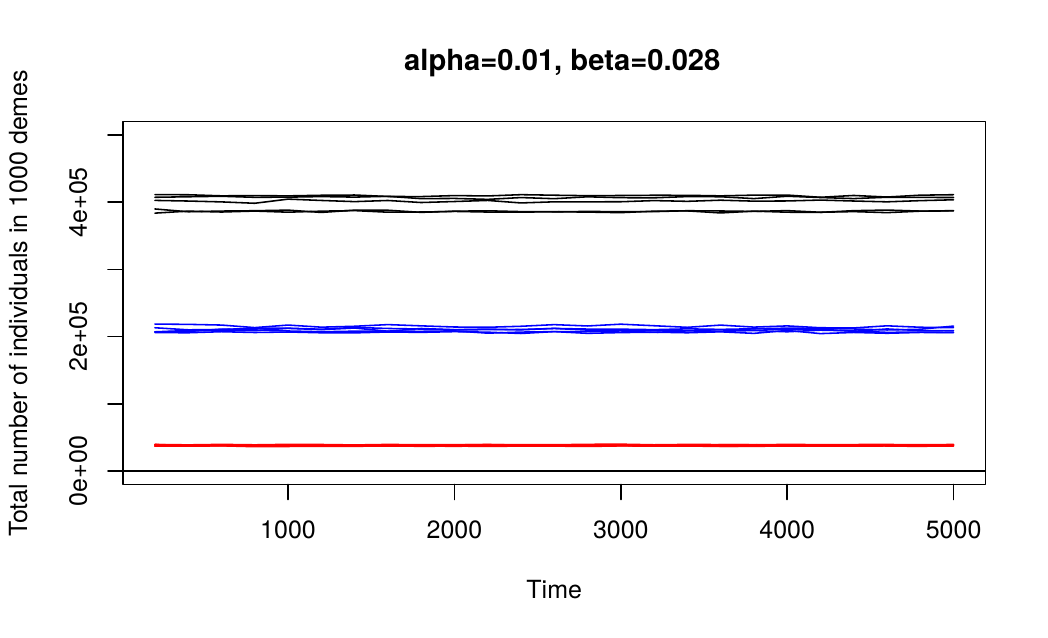}
  \includegraphics[width=8cm]{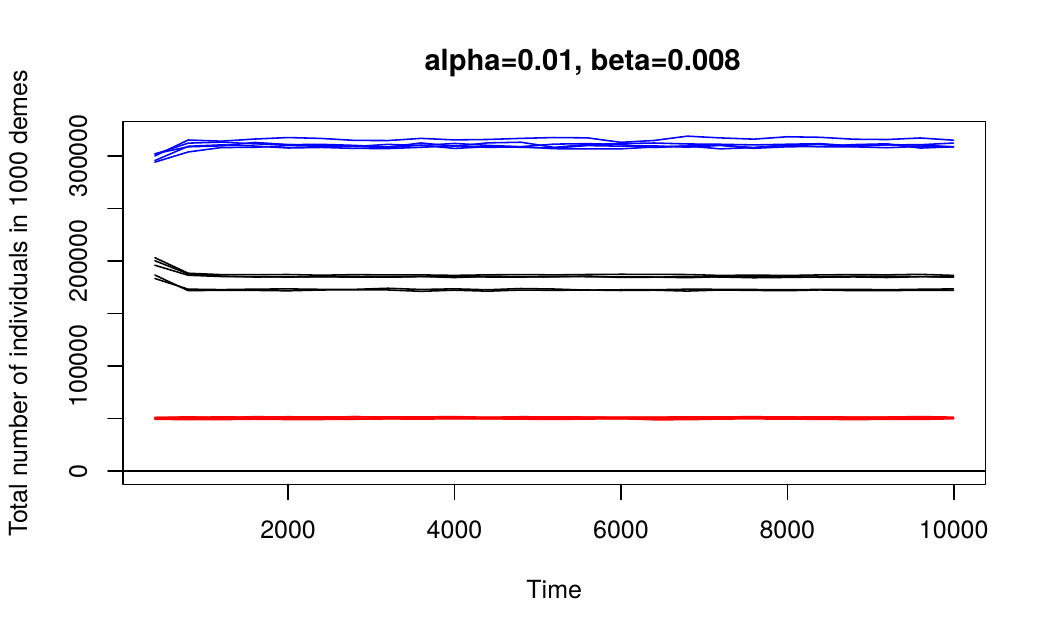}
  \caption{Total numbers of defenders (black), non-defenders (blue) and
    predators (red) in computer simulation runs with with no migration and
    all demes starting with 500 defenders, 500 non-defenders and 50 predators.}
  \label{fig:anmb}
\end{figure}
Some of the demes in which one or the other type -- defenders or non-defenders -- became
fixed may have started with very high initial frequencies in their deme.
To explore how the fixation chances of the costly defense trait are when \textit{not} starting
at high frequencies, we launched additional no-migration simulations in which all demes were
initialized with 500 defenders, 500 non-defenders and 50 predators (Fig.~\ref{fig:anmb}).
As we observed quick convergence in the previous no-migration simulations, we only simulated
5,000,000 $\tau$ leap iterations in each of the new runs.

In the simulations in which
$\beta$ was larger than the cost $\alpha$, the final total frequencies of the defenders were still
higher than those of the non-defenders (Fig.~\ref{fig:anmb}, left column).
In the cases with $\beta\le\alpha$ the final total frequencies of the non-defenders were
higher than those of the defenders, which, however, did not go extinct
(Fig.~\ref{fig:anmb}, right column).
In all those simulations there was at the end a strict dichotomy of the demes: A certain fraction
of demes was populated by defenders and neither non-defenders nor predators and all other demes
contained non-defenders and predators but no defenders.
In all cases, there were fewer demes with defenders than without defenders (Table \ref{tab:nmb}),
but in the simulations with $\alpha\le\beta$, more than 40 \% of the
demes contained defenders and the smaller numbers of demes was overcompensated by the larger
average number of inhabitants
(almost 1000 in defender demes versus 350 to 380 in non-defender demes).
\begin{table}
  \centering
  \begin{tabular}{|c|c|c|c|c|c|}
    \hline
    $\alpha$&$\beta$&\% demes with defenders&def.&non-def.&pred.\\
    \hline
    0.001&0.008&46.5&998.5&375.7&62.0\\
    0.01&0.028&40.2&993.2&351.2&64.2\\
    0.01&0.01&23.5&993.6&380.3&61.5\\
    0.01&0.008&18.1&993.9&379.6&61.8\\
    \hline
  \end{tabular}
  \caption{Final states of simulations without migration, starting with 500 defenders, 500 non-defenders and 50 predators in each deme.
    def.: average number of defenders in demes with defenders, non-def.: average
    number of non-defenders in demes without defenders, pred.: average number
  of predators in demes without defenders.}
  \label{tab:nmb}
\end{table}

The results of the no-migration simulations illustrate that larger values of $\beta$ are
beneficial for the defenders because the randon variation within a deme (``genetic drift'')
increases the probability that the defenders can become fixed despite their selective
disadvantage compared to the non-defenders.
This suggests that also in scenarios with migration between the demes random variation may not
only be beneficial for the spread of the defense trait because it leads to variation on the
deme level but also because it increases the probability that the locally less fit trait
survives in a deme until a deme-level fitness advantage emerges.

\begin{figure}
  \centering
  \includegraphics[width=10cm]{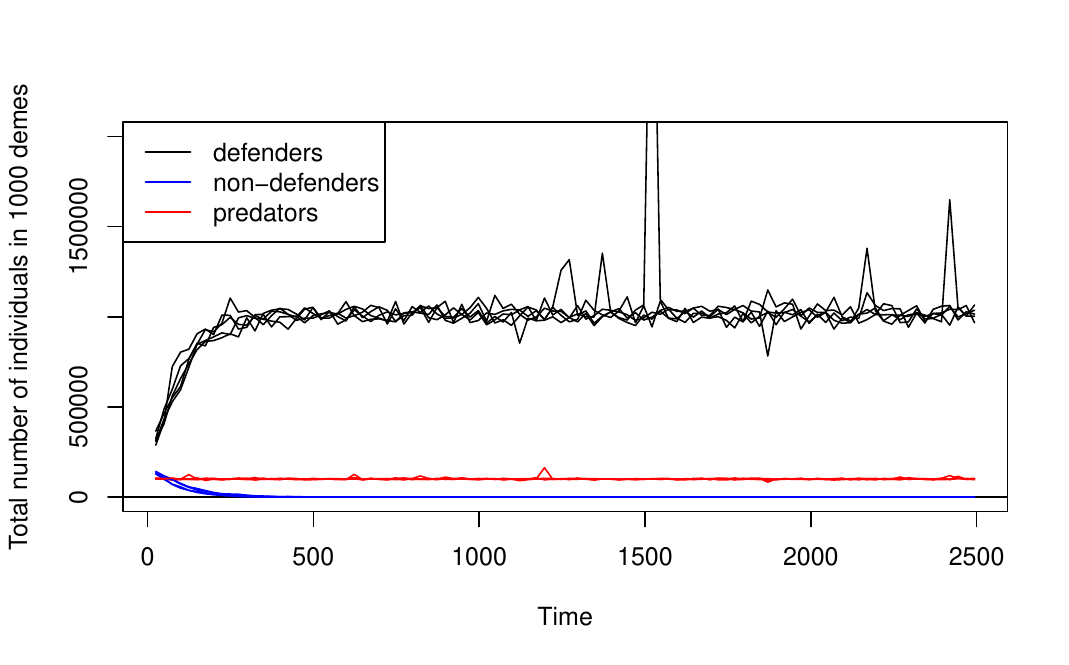}
  \caption{Five computer simulation runs with $\alpha=0.01$, $\beta=0.008$ and
    no within-species competition. (The hidden peak of the defender numbers
    in one of the simulation reaches almost 3.64 Mio.)}
  \label{fig:a01b0nK}
\end{figure}
In another series of simulations (with migration again) we removed
within-species competition from the model, that is we set $K=\infty$ and $\gamma=0$.
In these simulations all demes
were initialized with independently drawn numbers between 0 and 300 for each
sub-population.
The number of $\tau$ leap iterations was 1,000,000 with output every 10,000-th
iteration.
Even though we used $\alpha=0.01$ and $\beta=0.008$ in these simulations,
the defenders became more frequent and the non-defenders went extinct (Fig.~\ref{fig:a01b0nK}).
In the classical deterministic Lotka--Volterra model without competition (and without migration)
the frequencies of prey and predators oscillate infinitely long around
equilibrium frequencies.
Thus, a possible explanation why removing within-species competition brings an advantage for
the defenders over the non-defenders is that oscillations may be maintained
over a longer period, increasing between-deme variation of defender frequencies.
Another aspect is, however, that removing within-deme kin competition from the model
can make kin selection more effective.

In our simulations the final average population sizes per deme were
1017.8 for the prey and 100.98 for the predators.
If we for comparison consider a simple Lotka--Volterra model for a single deme
without competition and with all prey individuals being defenders,
the equilibrium frequencies are
$\frac{\nu}{\eta-\rho}=1000$ for the prey and
$\frac{\lambda-\alpha}{\delta}=99.5$
for the predators.
\subsection{Diffusion equation for frequencies of defenders}

We denote by
\begin{equation}  \begin{split}
  F_t^N(i):=\tfrac{A_t^N(i)}{A_t^N(i)+C_t^N(i)}
\end{split}     \end{equation}
the frequency
of 
defenders
in the prey population in deme $i\in\mathcal{D}$ at time $t\in[0,\infty)$.
The central goal of this article is to prove convergence of the sequence
$((F_t^N)_{t\in[0,\infty)})_{N\in\N}$ and to derive the
diffusion equation which the limit solves.
In other words, we will derive an analog of the Kimura stepping stone model
(i.e., spatially structured Wright--Fisher diffusions)
for 
costly
defense against predators.
Since we measure time in units of $N$ individuals (evolutionary time scale)
and the total migrating mass on this time scale should be finite,
%
we need to assume that $\beta_H^N$ and $\beta_P^N$ are of order $\tfrac{1}{N}$ for large $N\in\N$.
To get a nontrivial diffusion approximation we additionally assume
-- as is usual in the derivation of the Kimura stepping stone model --
slow migration and weak selection in the sense that the sequences
$(N\kappa_H^N)_{N\in\N}$, 
$(N\kappa_P^N)_{N\in\N}$
and
$(N\alpha^N)_{N\in\N}$
converge.
Thus the relative frequency of 
defenders
$\tfrac{A^N}{A^N+C^N}$ evolves on the time scale of order $N$ as $N\to\infty$.

In the special case that for some $N\in\N$ it holds that $\kappa_H^N=\kappa_P^N=\iota_H^N=\iota_P^N=\alpha^N=\beta_H^N=\beta_P^N=A_0^N=0$,
then $A^N\equiv 0$ and $(H^N(i),P^N(i))$, $i\in\D$, satisfy classical Lotka--Volterra equations.
It is well known that
if
$K\eta>\nu$,
then
the solutions of these equations converge to the nontrivial equilibrium
$(\tfrac{K\cdot(\delta\nu+\gamma\lambda)}{\lambda\gamma+K\delta\eta},\tfrac{\lambda K\eta-\lambda\nu}{\lambda\gamma+\delta K\eta})\in(0,\infty)^2$
in each deme.
Since we assume that $\kappa_H^N,\kappa_P^N,\alpha^N,\iota_H^N,\iota_P^N,\beta_H^N,\beta_P^N$ are of order $o(1)$ as $N\to\infty$
and since the 
defender
frequencies evolve slowly,
for every $i\in\mathcal{D}$,
the processes $(H^N(i),P^N(i))$ should asymptotically be close to the equilibrium of the classical Lotka--Volterra equations
with $\eta$ being replaced by $\eta-\rho F^N(i)$ as $N\to\infty$.
More precisely,
we will prove in Theorem~\ref{thm:sup.N.int.(H-h)^2.bounded} below under further assumptions that
if the local frequency of 
defenders
is $q\in[0,1]$,
then the equilibrium state for prey and predators should be
$(h_{\infty}(q),p_{\infty}(q))$
where the functions $h_\infty$ and $p_\infty$ are defined by
\begin{equation} \begin{split} \label{eq:def.h_infty.p_infty}
  [0,1]\ni x \mapsto h_\infty(x)&:=\tfrac{K(\delta\nu+\gamma\lambda)}{\lambda\gamma+\delta K(\eta-\rho x)}=\tfrac{1}{b(a-x)}\in(0,\infty)
  \\
  [0,1]\ni x \mapsto p_\infty(x)&:=\tfrac{\lambda K(\eta-\rho x)-\lambda\nu}{\lambda\gamma+\delta K(\eta-\rho x)}
  =\tfrac{\lambda}{\delta}\left(1-\tfrac{1}{Kb(a-x)}\right)\in(0,\infty).
\end{split} \end{equation}
For these functions to be well defined we will assume that $Kb(a-1)>1$ or, equivalently, that $K(\eta-\rho)>\nu$.

The above heuristic is incorrect if all populations go extinct by chance 
due to stochasticity in the offspring distributions.
To avoid this difficulty we will assume that there is sufficient immigration of prey ($2\iota_H^N\geq \beta_H^N$)
and predators ($2\iota_P^N\geq \beta_H^N$)
in order that both prey populations and predator populations cannot go extinct;
see Lemmas \ref{lem:H.positive} and \ref{lem:P.positive}, respectively.
However, note that both 
defenders
and non-defenders
can locally die out.
For our proof, which is based on the Lyapunov function \eqref{eq:def_lyapunov_function}, we additionally require
further restrictions on the parameters and on (inverse) moments of the initial configuration.
\begin{assumption} \label{ass:conv_LotkaVolterra}
  In the setting of the Section \ref{sec:model}
  it holds that
  $\lambda>\nu$, $\eta-\rho >\tfrac{\lambda}{K}$, $\gamma\geq 2\delta$,
  for all $N\in\N$ it holds that
  $\alpha^N+\kappa_H^N\leq\tfrac{\lambda}{4}$,
  $\iota_P^N\leq\tfrac{\lambda(\nu+\lambda)}{8\delta}$,
  $\kappa_P^N+\kappa_H^N+\alpha^N\leq\tfrac{\lambda-\nu}{2}$,
  $\iota_H^N\geq \tfrac{4\delta\kappa_P^N}{3(\nu+\lambda)}+\tfrac{3}{2}\beta_H^N$,
  $\iota_P^N\geq \beta_P^N$, and
  there exist
  $\sigma=(\sigma_i)_{i\in\D}\in(0,\infty)^{\D}$ and $c\in(0,\infty)$
  such that $\sum_{i\in\D}\sigma_i<\infty$,
  such that for every $j\in\D$ it holds that
  \begin{equation} \label{eq:bound.sumi.sigmai.mij}
    \sum_{i\in\D}\sigma_i m(i,j)\leq c \sigma_j,
  \end{equation}
  and such that
  $\sup_{N\in\N}\E\Big[\Big\|\Big(H_0^N+P_0^N\Big)^4
  +\tfrac{1}{(H_0^N)^2}+\tfrac{P_0^N}{(H_0^N)^2}
  +\tfrac{1}{P_0^N}+\tfrac{1}{P_0^NH_0^N}\Big\|_{\sigma}\Big]<\infty$.
\end{assumption}
    
The following theorem, which appears to be new even for non-spatial Lotka--Volterra SDEs,
implies for every $t\in[0,\infty)$ that the $L^2([0,t]\times l_{\sigma}^1\times\Omega;\R)$-distance
between $(H^N_{\cdot N},P^N_{\cdot N})$ and $(h_{\infty}(F^N_{\cdot N}),p_{\infty}(F^N_{\cdot N}))$ converges to $0$ as $N\to\infty$
at least with rate $\tfrac{1}{2}$.
Theorem~\ref{thm:sup.N.int.(H-h)^2.bounded} follows immediately from Theorem \ref{thm:bound.V} below
together with a time substitution.
For the proof of 
Theorem~\ref{thm:sup.N.int.(H-h)^2.bounded} below, we exploit uniformly bounded inverse moments
and establishing these inverse moments (see Lemmas~\ref{lem:bound.norm.1/H^2} and \ref{lem:bound.norm.1/PH} below) is somewhat technical.
\begin{theorem} \label{thm:sup.N.int.(H-h)^2.bounded}
  Assume the setting of Section \ref{sec:model},
  let Assumption \ref{ass:conv_LotkaVolterra} hold,
  let $h_\infty$ and $p_\infty$ be given by~\eqref{eq:def.h_infty.p_infty}
  and
  assume that 
  $\sup_{N\in\N}(N\max\{\kappa_H^N,\kappa_P^N,\alpha^N,\iota_H^N,\iota_P^N,\beta_H^N,\beta_P^N\})<\infty$.
  Then we get for all $t\in[0,\infty)$ that
  \begin{equation}
    \sup_{N\in\N}
    N\int_0^t
    \E\left[\sum_{i\in{\D}}\sigma_i\left(H_{uN}^N(i)-h_\infty\left(F_{uN}^N(i)\right)\right)^2
    +
    \sum_{i\in{\D}}\sigma_i\left(P_{uN}^N(i)-p_\infty\left(F_{uN}^N(i)\right)\right)^2\right]
    \,du
    <\infty.
  \end{equation}
\end{theorem}
Knowing the asymptotic behavior of the prey populations, we can formally replace the $(H^N)_{N\in\N}$ in
the diffusion equation of the 
defender
frequencies (see \eqref{eq:dF} below) and, thereby, we arrive at the diffusion equation
solved by the limit  of defender frequencies.
Our main result, Theorem~\ref{thm:conv.freq.altruist}, then proves that
the 
defender
frequencies converge to the solution of the diffusion equation~\eqref{eq:X} below.
The proof of Theorem~\ref{thm:conv.freq.altruist} is
deferred to Section~\ref{sec:convF_proof} below and is based on a general stochastic averaging result
in
\citet{Kurtz1992}.
    
\begin{theorem} \label{thm:conv.freq.altruist}
  Assume the setting of Section \ref{sec:model},
  let Assumption \ref{ass:conv_LotkaVolterra} hold,
  assume
  that there exist
  $\kappa,\alpha,\beta\in[0,\infty)$
  such that
  $\lim_{N\to\infty}\kappa_H^NN=\kappa$,
  $\lim_{N\to\infty}\alpha^NN=\alpha$
  and 
  $\lim_{N\to\infty}\beta_H^NNb=\beta $,
  assume that
  \begin{equation}  \begin{split}
 \Big( \sum_{i\in\D}\sigma_i\sup_{N\in\N}\E\Big[H_0^N(i)\Big]\Big)
  +\sup_{N\in\N}(N\max\{\kappa_P^N,\iota_H^N,\iota_P^N,\beta_P^N\})<\infty
  \end{split}     \end{equation}
  and assume that $F_0^N\Longrightarrow X_0$ as $N\to\infty$ in $l_{\sigma}^1$.
  Then the SDE
  \begin{equation} \begin{split} \label{eq:X}
    dX_t(i)=&
    \kappa\sum\limits_{j\in\D}m(i,j)
    \tfrac{a-X_t(i)}{a-X_t(j)}\Big(X_t(j)-X_t(i)\Big)\,dt
    -\alpha X_t(i)(1-X_t(i))\,dt
    \\
    &
    +\sqrt{\beta(a-X_t(i))X_t(i)(1-X_t(i))}\,dW_t(i),\quad t\in(0,\infty), i\in\mathcal{D}
  \end{split} \end{equation}
  (where $\{W(i)\colon i\in\mathcal{D}\}$ are independent standard Brownian motions)
  has a unique strong solution
  and
  \begin{equation}
    \left(F_{tN}^N\right)_{t\in[0,\infty)}
    \Longrightarrow 
    \left(X_t\right)_{t\in[0,\infty)}
  \end{equation}
  as $N\to\infty$ in $C([0,\infty),l_{\sigma}^1)$.
\end{theorem}

\subsection{Many-demes limit}
An important problem is to derive conditions under which 
defenders
persist,
that is, to derive conditions on the parameters of the SDE \eqref{eq:X} under which the process goes to fixation.
Here we simplify this problem and consider the many-demes-limit (also denoted as mean-field approximation) of the SDE \eqref{eq:X}.
More precisely, for every $D\in\N$, let $X^D\colon [0,\infty) \times  $\{1,\ldots,D\}$ \times \Omega \to [0,1]$ be the solution of the SDE \eqref{eq:X}
with  $\D$ replaced by $\{1,\ldots,D\}$
and with $m$ replaced by $(\tfrac{1}{D})_{i,j\in\{1,\ldots,D\}}$.
We will show in Proposition \ref{prop:MVL} together with Lemma \ref{lem:mvl_application} below that
if, for every $D\in\N$,
$(X_0^D(i))_{i\in\{1,\ldots,D\}}$ are exchangeable $[0,1]$-valued random variables, if
$\sup_{D\in\N}\E[(X_0^D(1))^2]<\infty$,
if $Z\colon[0,\infty)\times\Omega\to[0,1]$ is the solution of the SDE~\eqref{eq:dZ} below with respect to the
Brownian motion $W(1)$
and if $\sup_{D\in\N}\sqrt{D}\E\left[\left|X_0^D(i)- Z_0(i)\right|\right]<\infty$,
then for all $t\in[0,\infty)$ it holds that
\begin{equation}
  \sup_{D\in\N}\sqrt{D}\E\left[| X_t^D(1)-Z_t|\right] < \infty.
\end{equation}
Thus the solution of the SDE~\eqref{eq:dZ} is the many-demes limit of the SDE~\eqref{eq:X}.
For this many-demes limit we derive a simple necessary and sufficient condition ($\alpha<\beta$)
under which
the 
costly
defense trait goes to fixation when starting with a positive frequency.
The proof of Theorem~\ref{thm:longtermbehavior} is deferred to Section~\ref{sec:longtermbehavior_proof}.
\begin{theorem} \label{thm:longtermbehavior}
  Let $\alpha,\beta,\kappa \in (0,\infty)$, let $a\in (1,\infty)$,
  let $(\Omega, \mathcal{F}, \P, (\mathcal{F}_t)_{t\in[0,\infty)})$ be a filtered probability space,
  let $W\colon [0,\infty) \times \Omega \to \R$
  be a standard $(\mathcal{F}_t)_{t\in[0,\infty)}$ Brownian motion with continuous sample paths,
  and let $Z_0\colon \Omega \to [0,1]$ be an $\mathcal{F}_0$/$\mathcal{B}([0,1])$-measurable mapping.
  Then the SDE
  \begin{equation}  \begin{split} \label{eq:dZ}
    dZ_t
    &=
    \kappa (a-Z_t)\left((a-Z_t)\E\left[\tfrac{1}{a-Z_t}\right]-1\right)\,dt-\alpha Z_t\left(1-Z_t\right)\,dt
    +\sqrt{\beta(a-Z_t)Z_t(1-Z_t)}\,dW_t
  \end{split} \end{equation}
  has a unique solution.
  Furthermore, if $\E[Z_0]=1$, then $\P[Z_t=1\textup{ for all }t\in[0,\infty)]=1$,
  if $\E[Z_0]=0$, then $\P[Z_t=0\textup{ for all }t\in[0,\infty)]=1$
  and if $\E[Z_0]\in(0,1)$, then 
  \begin{equation} \begin{split}
    \lim_{t\to\infty}\E\big[\big|Z_t-0\big|\big]=0, &\text{ if } \alpha >\beta ,\\
    \lim_{t\to\infty}\E\big[\big|Z_t-1\big|\big]=0, &\text{ if } \alpha <\beta ,\\
    Z_t\stackrel{t \rightarrow \infty} {\Longrightarrow} \int_{\cdot}m(z)\,dz, &\text{ if } \alpha =\beta,
  \end{split} \end{equation}
  where
  $m(z)
  =\tfrac{1}{c}z^{\frac{2\kappa }{\beta }(a\theta-1)-1}
  (1-z)^{\frac{2\kappa }{\beta }(1-\theta(a-1))-1}
  (a-z)^{\frac{2\alpha }{\beta }-1}$
  for $z\in(0,1)$,
  where $c\in(0,\infty)$ is a normalizing constant
  and where $\theta=\E[\tfrac{1}{a-Z_0}]$.
\end{theorem}
Informally speaking, Theorem~\ref{thm:longtermbehavior} asserts that
a costly
defense allele persists in a space of infinitely many demes
if $\alpha<\beta$ and if the mean frequency of 
defenders
over all demes is positive.
This does not imply that a new mutation resulting in 
costly
defense behavior
can establish itself on one island or even in the total population.
Our final result partially closes this gap and considers
a process which could be the limit $\lim_{D\to\infty}\sum_{i=1}^DX^D(i)$ if 
for all $D\in\N$ and $i\in\{1,\ldots,D\}$ it holds that
$X_0^D(i)=Y_0\1_{i=1}$ for some $[0,1]$-valued random variable
\citep{Hutzenthaler2009,Hutzenthaler2012}.
For this limiting process, Proposition~\ref{prop:invasion} below shows in the case $\P[Y_0>0]=1$
that the process converges to $0$
in probability
as time goes to infinity  if and only if $\alpha \geq \beta$.
Informally speaking, Proposition~\ref{prop:invasion} asserts that
a costly
defense allele has a positive invasion probability in an infinite-dimensional space
if and only if $\alpha <\beta$.

\section{Convergence of the relative frequency of
defenders
}\label{sec:conv_altr_freq}
The main result of this section, Theorem~\ref{thm:bound.V} below, proves strong convergence
of spatial stochastic Lotka--Volterra processes.
Its proof uses a Lyapunov type argument.
This argument uses uniformly bounded inverse moments which we establish
in Section~\ref{sec:convergence.LV}.
In section~\ref{sec:convF_proof},
we then prove convergence of relative frequencies as stated in Theorem \ref{thm:conv.freq.altruist}.
\subsection{Setting} \label{sec:convF_setting}
  Assume the setting of Section \ref{sec:model},
Define $\bar{\kappa}_H:=\sup_{N\in\N}\kappa_H^N$,
$\bar{\kappa}_P:=\sup_{N\in\N}\kappa_P^N$,
$\betab_H:=\sup_{N\in\N}\beta_H^N$,
$\betab_P:=\sup_{N\in\N}\beta_P^N$,
$\bar{\iota}_H:=\sup_{N\in\N}\iota_H^N$,
and
$\bar{\iota}_P:=\sup_{N\in\N}\iota_P^N$.
For all $z=(z_i)_{i\in\D}\in(0,\infty)^{\D}$ and $p\in\R$ let $z^p=\left(z_i^p\right)_{i\in\D}$.
Furthermore, let $\underline{1}:=(1)_{i\in\D}\in l_\sigma^1$.
Define $E_1:=[0,1]^\D$ and 
$E_2:=l_\sigma^1\cap[0,\infty)^\D$.
For all $i\in\D$ and all $N\in\N$
let $W^{H,N}(i)\colon [0,\infty) \times \Omega \to \R$
and $W^{F,N}(i)\colon [0,\infty) \times \Omega \to \R$
be stochastic processes with continuous sample paths such that
for every $t\in[0,\infty)$ it holds $\P$-a.s.~that
\begin{equation}
  dW_t^{H,N}(i)=
  \tfrac{\sqrt{A_t^N(i)}\,dW_t^{A,N}(i)
    +\sqrt{C_t^N(i)}\,dW_t^{C,N}(i)}
  {\sqrt{H_t^N(i)}}
\end{equation}
and
\begin{equation}
  dW_t^{F,N}(i)=
  \tfrac{\sqrt{C_t^N(i)}\,dW_t^{A,N}(i)
    -\sqrt{A_t^N(i)}\,dW_t^{C,N}(i)}
  {\sqrt{H_t^N(i)}},
\end{equation}
respectively,
with $W_0^{H,N}(i)=W_0^{F,N}(i)=0$.

\subsection{Preliminaries}
Assume the setting of Section \ref{sec:convF_setting}.
In this section we collect some first results that are used
in the proofs of the statements in subsequent sections.
\begin{lemma} \label{lem:HFP}
Assume the setting of Section \ref{sec:convF_setting}.
Then $W^{H,N}(i)$ and $W^{F,N}(i)$, $N\in\N$, $i\in\D$, are independent Brownian motions and
for all $t\in[0,\infty)$, all $i\in\D$, and all $N\in\N$ it $\P$-a.s.~holds that 
\begin{align}
  \label{eq:dH}
  H_t^N(i)
  &=
  H_0^N(i)
  +\int_0^t
  \kappa_H^N\sum_{j\in\D} m(i,j)H_s^N(j)
  +(\lambda-\kappa_H^N-\alpha^NF_s^N(i))H_s^N(i)
  -\tfrac{\lambda}{K}\left(H_s^N(i)\right)^2
  \\
  &\quad
  -\delta P_s^N(i)H_s^N(i)+\iota_H^N\,ds+\int_0^t\sqrt{\beta_H^N H_s^N(i)}\,dW_s^{H,N}(i),
  \notag\\
  \label{eq:dF}
  F_t^N(i)
  &=
  F_0^N(i)+
  \int_0^t
  \kappa_H^N
  \sum\limits_{j\in\D}m(i,j)\left(F_s^N(j)-F_s^N(i)\right)\tfrac{H_s^N(j)}{H_s^N(i)}
  -\alpha^NF_s^N(i)\left(1-F_s^N(i)\right)
  \,ds
  \\
  &\quad
  +\int_0^t\sqrt{\tfrac{\beta_H^N F_s^N(i)\left(1-F_s^N(i)\right)}{H_s^N(i)}}dW_s^{F,N}(i), 
  \notag\\
  \label{eq:dP}
  P_t^N(i)
  &=
  P_0^N(i)
  +\int_0^t\kappa_P^N\sum_{j\in\D} m(i,j)P_s^N(j)
  -(\kappa_P^N+\nu)P_s^N(i)
  -\gamma \left(P_s^N(i)\right)^2
  +\left(\eta-\rho F_s^N(i)\right)P_s^N(i)H_s^N(i)
  \\
  &\quad
  +\iota_P^N\,ds+\int_0^t\sqrt{\beta_P^N P_s^N(i)}\,dW_s^{P,N}(i).
  \notag
\end{align}
\begin{proof}
For all $t\in[0,\infty)$, all $N\in\N$, and all $i\in\D$ we get
$\left<W^{H,N}(i)\right>_t=\left<W^{F,N}(i)\right>_t=t$
as well as
\begin{equation}
  \left<W^{H,N}(i),W^{F,N}(i)\right>_t
  =\int_0^t\tfrac{\sqrt{A_s^N(i)C_s^N(i)}-\sqrt{A_s^N(i)C_s^N(i)}}{H_s^N(i)}\,ds
  =0.
\end{equation}
Hence, we see that $W^{H,N}(i)$ and $W^{F,N}(i)$,
$N\in\N$, $i\in\D$, are independent Brownian motions.
Equation \eqref{eq:dH} follows from It\^o's lemma
\citep[e.g.][]{Klenke2008}
and rearranging terms.
Furthermore, applying It\^o's lemma we see for all $t\in[0,\infty)$, all $i\in\D$, and all $N\in\N$
that $\P$-a.s.~it holds that
\begin{equation} \begin{split}
  F_t^N(i)
  =&
  F_0^N(i)
  +\int_0^t\tfrac{C_s^N(i)}{\left(H_s^N(i)\right)^2}\Bigg(\kappa_H^N \sum\limits_{j\in\D}m(i,j)\left(A_s^N(j)-A_s^N(i)\right)
  +A_s^N(i)\big(\lambda\left(1-\tfrac{H_s^N(i)}{K}\right)-\delta P_s^N(i)-\alpha^N\big)
  \\
  &
  +\iota_H^N\tfrac{A_s^N(i)}{H_s^N(i)}\Bigg)\,ds
  +\int_0^t\tfrac{C_s^N(i)}{\left(H_s^N(i)\right)^2}\sqrt{\beta_H^N A_s^N(i)}\,dW^{A}_s(i)
  \\
  &
  -\int_0^t\tfrac{A_s^N(i)}{\left(H_s^N(i)\right)^2}
  \Bigg(\kappa_H^N \sum\limits_{j\in\D}m(i,j)\left(C_s^N(j)-C_s^N(i)\right)
  +C_s^N(i)\left(\lambda\left(1-\tfrac{H_s^N(i)}{K}\right)
  -\delta P_s^N(i)\right)
  \\
  &
  +\iota_H^N\tfrac{C_s^N(i)}{H_s^N(i)}\Bigg)\,ds
  -\int_0^t\tfrac{A_s^N(i)}{\left(H_s^N(i)\right)^2}\sqrt{\beta_H^N C_s^N(i)}\,dW^{C}_s(i)
  -\int_0^t\tfrac{C_s^N(i)}{\left(H_s^N(i)\right)^3}\beta_H^N A_s^N(i)
  +\tfrac{A_s^N(i)}{\left(H_s^N(i)\right)^3}\beta_H^N C_s^N(i)\,ds
  \\
  =&
  F_0^N(i)
  +\int_0^t
  \tfrac{\kappa_H^N}{H_s^N(i)}
  \sum\limits_{j\in\D}m(i,j)\left(\left(1-F_s^N(i)\right)F_s^N(j)H_s^N(j)-F_s^N(i)\left(1-F_s^N(j)\right)H_s^N(j)\right)
  \\
  &\quad
  -\alpha^NF_s^N(i)\left(1-F_s^N(i)\right)\,ds
  +\int_0^t\sqrt{\tfrac{\beta_H^N F_s^N(i)\left(1-F_s^N(i)\right)}{H_s^N(i)}}dW_s^{F,N}(i)
\end{split} \end{equation}
and \eqref{eq:dF} follows.
Finally, we obtain \eqref{eq:dP} from the definition of $\left(H^N\right)_{N\in\N}$ and $\left(F^N\right)_{N\in\N}$.
\end{proof}
\end{lemma}

\begin{lemma} \label{lem:H.positive}
Assume the setting of Section \ref{sec:convF_setting}
and assume that for all $N\in\N$ we have
$\iota_H^N\geq \tfrac{1}{2}\beta_H^N$.
Furthermore, assume that we have for all $N\in\N$ and all $i\in\D$
that $\P$-a.s.~$H_0^N(i)>0$.
Then we have
\begin{equation} \label{eq:lem.H.positive.claim}
  \P\left[H_u^N(i)>0, \text{ for all } u\in[0,\infty) \text{, all } N\in\N\text{, and all }i\in\D\right]=1.
\end{equation}
\begin{proof}
The main problem is to control the contribution of the processes $(P^N)_{N\in\N}$
and resolve this with a cut-off argument.
For every $N,M\in\N$
let $\hat{H}^{N,M}\colon [0,\infty)\times \D \times \Omega \to [0,\infty)$ be an adapted process with continuous sample paths that
for all $t\in[0,\infty)$ and all $i\in\D$ satisfies $\P$-a.s.
\begin{equation} \begin{split}
  \hat{H}_t^{N,M}(i)
  &=
  \hat{H}_0^{N,M}(i)
  +\int_0^t
  \left[\hat{H}_s^{N,M}(i)\left(\lambda-\alpha^N-\kappa_H^N-\tfrac{\lambda}{K}\hat{H}_s^{N,M}(i)-\delta M\right)+\iota_H^N\right]\,ds
  \\
  &\quad
  +\int_0^t \sqrt{\beta_H^N\hat{H}_s^{N,M}(i)}\,dW_s^{H,N}(i)
\end{split} \end{equation}
with $\hat{H}_0^{N,M}(i)=H_0^N(i)$.
Due to Feller's boundary classification
\citep[e.g., p. 366 in ][]{EthierKurtz1986}
with the assumption that
for all $N\in\N$ it holds that $\iota_H^N\geq \tfrac{1}{2}\beta_H^N$
we have for every $N,M\in\N$ and all $i\in\D$ that
\begin{equation} \label{eq:hatH_t^n.positive}
  \P\left[\hat{H}_t^{N,M}(i)>0,\text{ for all }t\in[0,\infty)\right]=1.
\end{equation}
For all $N,M\in\N$, all $i\in\D$, and all $t\in[0,\infty)$
consider the event $A_M^N(i):=\left\{\sup\limits_{s\in[0,t]}P_s^N(i)\leq M\right\}$.
The fact that the processes
$P^N(i)$, $N\in\N$, $i\in\D$, have a.s.\ \cadlag\ sample paths (which are bounded on
closed intervals) implies
for all $N,M\in\N$, all $i\in\D$, and all $t\in[0,\infty)$ that
\begin{equation} \begin{split} \label{eq:properties.A_M}
  &A_M^N(i)\subseteq A_{M+1}^N(i),
  \\
  &\P\left[\bigcup\limits_{K\in\N}A_K^N(i)\right]=\P\left[\sup\limits_{s\in[0,t]}P_s^N(i)<\infty\right]=1.
\end{split} \end{equation}
Using a comparison result due to Ikeda and Watanabe
\citep[see e.g., Theorem V.43.1 in ][]{RogersWilliams2000b}
we get for all $N,M\in\N$, all $i\in\D$ , and all $t\in[0,\infty)$ that
\begin{equation} \begin{split} \label{eq:H_t>=hatH_t^M}
  &\P\left[\exists\, u\in[0,t]:
  H_u^N(i)<\hat{H}_u^{N,M}(i), \sup\limits_{s\in[0,t]}P_s^N(i)\leq M\right]=0.
\end{split} \end{equation}
Thus, combining \eqref{eq:hatH_t^n.positive}, \eqref{eq:properties.A_M}, and \eqref{eq:H_t>=hatH_t^M} we obtain
for all $N\in\N$, all $i\in\D$, and all $t\in[0,\infty)$ that
\begin{equation} \begin{split}
  1
  &\geq
  \P\left[H_u^N(i)>0, \text{ for all } u\in[0,t]\right]
  =
  1-\P\left[\exists\, u\in[0,t]: H_u^N(i)=0\right]
  \\
  &\geq 1-\sum_{M\in\N}\P\left[\exists\, u\in[0,t]:
  H_u^N(i)=0, \sup\limits_{s\in[0,t]}P_s^N(i)\leq M\right]
  \\
  &\geq 1-\sum_{M\in\N}\P\left[\exists\, u\in[0,t]:
  H_u^N(i)<\hat{H}_u^{N,M}(i), \sup\limits_{s\in[0,t]}P_s^N(i)\leq M\right]
  = 1.
\end{split} \end{equation}
This implies for all $N\in\N$, all $i\in\D$, and all $t\in[0,\infty)$ that
$\P\left[H_u^N(i)>0, \text{ for all } u\in[0,t]\right]=1$,
which in turn implies \eqref{eq:lem.H.positive.claim}.
This completes the proof of Lemma \ref{lem:H.positive}.
\end{proof}
\end{lemma}
\begin{lemma} \label{lem:P.positive}
Assume the setting of Section \ref{sec:convF_setting} and assume 
that for all $N\in\N$ it holds that
$\iota_P^N\geq \tfrac{1}{2}\beta_P^N$.
Furthermore, assume that we have for all $N\in\N$ and all $i\in\D$
that $\P$-a.s.~$P_0^N(i)>0$.
Then we have
\begin{equation}
  \P\left[P_t^N(i)>0,\text{ for all }t\in[0,\infty), \text{ all }N\in\N,\text{ and all }i\in\D\right]
  =1.
\end{equation}
\end{lemma}
\begin{proof}
  Analogous to the proof of Lemma \ref{lem:H.positive}.
\end{proof}
\begin{lemma} \label{lem:estimate.term.sum^p}
  Let $\mathcal{D}$ be a countable set, let
  $m\in[0,\infty)^{\mathcal{D}\times\mathcal{D}}$ be a
  stochastic matrix, let $\sigma\in[0,\infty)^{\mathcal{D}}$, $c\in[0,\infty)$
  satisfy for all $j\in\mathcal{D}$ that
  \begin{equation} \label{eq:bound.sumi.sigmai.mij2}
    \sum_{i\in\D}\sigma_i m(i,j)\leq c \sigma_j,
  \end{equation}
  and
  let 
  $x=(x_i)_{i\in\D}\in [0,\infty)^{\mathcal{D}}$, $p\in[1,\infty)$,
  $\D'\subseteq\D$.
  Then
  \begin{equation} 
    \sum_{i\in\D'}\sigma_i\Big(\sum_{j\in\D} m(i,j)x_j\Big)^p
    \leq\sum_{i\in\D}c\sigma_ix_i^p.
  \end{equation}
\begin{proof}
  Jensen's inequality, the fact that $m$ is a stochastic matrix,
  the theorem of Fubini--Tonelli
  and \eqref{eq:bound.sumi.sigmai.mij2} imply
  that
  \begin{equation}  \begin{split}
    \sum_{i\in\D'}\sigma_i\Big(\sum_{j\in\D} m(i,j)x_j\Big)^p
    \leq\sum_{i\in\D}\sigma_i\sum_{j\in\D} m(i,j)x_j^p
    \leq\sum_{j\in\D}c\sigma_jx_j^p.
  \end{split}     \end{equation}
  This proves the assertion.
\end{proof}
\end{lemma}

\subsection{Strong convergence of the spatial stochastic Lotka--Volterra processes} \label{sec:convergence.LV}
In this section we will show the convergence of the time-rescaled Lotka--Volterra processes as given in \eqref{eq:dH} and \eqref{eq:dP}.
In Lemmas \ref{lem:bound.norm.(H+P)^p}, \ref{lem:bound.norm.1/H^2}, and \ref{lem:bound.norm.1/P}
we will provide bounds for the expected value of the sum (over sets of demes) of functionals of the processes weighted by $\sigma$.
The proofs of these lemmas are technical are therefore deferred to the appendix.
Lemmas \ref{lem:bound.norm.(H+P)^p}, \ref{lem:bound.norm.1/H^2}, and \ref{lem:bound.norm.1/P}
are then used in Theorem \ref{thm:bound.V} to study the behavior of a spatial analogue of a well-known Lyapunov function
\citep[e.g.,][]{DobrinevskiFrey2012}.

\begin{lemma} \label{lem:bound.norm.(H+P)^p}
Assume the setting of Section \ref{sec:convF_setting} and
let $p\in\{1\}\cup [2,\infty)$.
Then we have
\begin{equation} \begin{split}
  &\sup\limits_{N\in\N}\sup\limits_{t\in[0,\infty)}
  \E\left[\left\|\left(2\eta H_t^N+\delta P_t^N\right)^p\right\|_{\sigma}\right]
  \leq
  \sup\limits_{N\in\N}\E\left[\left\|\left(2\eta H_0^N+\delta P_0^N\right)^p\right\|_{\sigma}\right]
  \\
  &\qquad
  +\|\underline{1}\|_\sigma
  \left(\tfrac{\lambda+\left(1-\frac{1}{p}+\frac{c}{p}\right)\left(\bar{\kappa}_H+\bar{\kappa}_P\right)}
  {2\min\left\{\tfrac{1}{2\eta}\tfrac{\lambda}{K},\tfrac{1}{4},\tfrac{1}{\delta}\gamma\right\}}
  \right)^p
  \left(1+\sqrt{1+
  \tfrac{4\min\left\{\tfrac{1}{2\eta}\tfrac{\lambda}{K},\tfrac{1}{4},\tfrac{1}{\delta}\gamma\right\}\left[2\eta\bar{\iota}_H+
    \delta\bar{\iota}_P+(p-1)\left(2\eta\betab_H+\tfrac{1}{2}\delta\betab_P\right)\right]}
  {\left(\lambda+\left(1-\frac{1}{p}+\frac{c}{p}\right)\left(\bar{\kappa}_H+\bar{\kappa}_P\right)\right)^2}}\right)^p.
\end{split} \end{equation}
\end{lemma}

\begin{lemma} \label{lem:bound.norm.1/H^2} \label{lem:bound.norm.P/H^2} \label{lem:bound.norm.1/H^2+P/H^2}
Assume the setting of Section \ref{sec:convF_setting} and
assume $\gamma\geq 2\delta$.
Furthermore, assume that for all $N\in\N$ we have
$\alpha^N+\kappa_H^N\leq\tfrac{\lambda}{4}$,
$\iota_P^N\leq\tfrac{\lambda(\nu+\lambda)}{8\delta}$, and
$\iota_H^N\geq \tfrac{4\delta\kappa_P^N}{3(\nu+\lambda)}+\tfrac{3}{2}\beta_H^N$.
Let $\hat{\D}\subseteq\D$ be a set.
Then we have
\begin{equation} \begin{split} \label{eq:claim.lemma.1/H^2}
  &\sup\limits_{N\in\N}\sup\limits_{t\in[0,\infty)}
  \E\left[\sum_{i\in\hat{\D}}\sigma_i\left(
  \tfrac{2}{\lambda+\nu}\tfrac{P_t^N(i)}{\left(H_t^N(i)\right)^2}+\tfrac{1}{2\delta}\tfrac{1}{\left(H_t^N(i)\right)^2}\right)\right]
  \leq
  \sup\limits_{N\in\N}\E\left[\sum_{i\in\hat{\D}}\sigma_i\left(\tfrac{2}{\lambda+\nu}\tfrac{P_0^N}{\left(H_0^N\right)^2}
  +\tfrac{1}{2\delta}\tfrac{1}{\left(H_0^N\right)^2}\right)\right]
  \\
  &\qquad
  +\tfrac{4\bar{\kappa}_{P}c}{3\lambda(\lambda+\nu)}
  \sup\limits_{N\in\N}\sup\limits_{t\in[0,\infty)}\E\left[\sum_{i\in\hat{\D}}\sigma_i\left(P_t^N(i)\right)^3\right]
  +\tfrac{4}{\lambda(\lambda+\nu)}
  \left(\tfrac{\eta^2}{\lambda}
  +\tfrac{4\lambda}{K^2}
  \right)
  \sup\limits_{N\in\N}\sup\limits_{t\in[0,\infty)}\E\left[\sum_{i\in\hat{\D}}\sigma_iP_t^N(i)\right]
  +\tfrac{2}{K^2\delta}.
\end{split} \end{equation}
\end{lemma}

\begin{lemma} \label{lem:bound.norm.1/P} \label{lem:bound.norm.1/PH}
Assume the setting of Section \ref{sec:convF_setting}
and assume $\lambda>\nu$ and $\eta-\rho >\tfrac{\lambda}{K}$.
Furthermore, assume that for all $N\in\N$ we have
$\iota_H^N\geq \tfrac{1}{2}\beta_H^N$,
$\kappa_P^N+\kappa_H^N+\alpha^N\leq\tfrac{\lambda-\nu}{2}$,
$\iota_P^N\geq \beta_P^N$, and
$\iota_H^N\geq\beta_H^N$.
Let $\hat{\D}\subseteq\D$ be a set.
Then we have
\begin{equation} \begin{split} \label{eq:claim.lemma.1/P}
  &\sup\limits_{N\in\N}\sup\limits_{t\in[0,\infty)}
  \E\left[\sum_{i\in\hat{\D}}\sigma_i\left(\tfrac{(\eta-\rho )-\frac{\lambda}{K}}{2(\bar{\kappa}_P+\nu)}
  \tfrac{1}{P_t^N(i)}+\tfrac{1}{P_t^N(i)H_t^N(i)}\right)\right]
  \leq
  \sup\limits_{N\in\N}
  \E\left[\sum_{i\in\hat{\D}}\sigma_i\left(\tfrac{(\eta-\rho )-\frac{\lambda}{K}}{2(\bar{\kappa}_P+\nu)}
  \tfrac{1}{P_0^N(i)}+\tfrac{1}{P_0^N(i)H_0^N(i)}\right)\right]
  \\
  &\qquad\qquad
  +\tfrac{1}{\min\left\{\bar{\kappa}_P+\nu,\frac{\lambda-\nu}{2}\right\}}
  \left(\gamma \tfrac{(\eta-\rho )-\frac{\lambda}{K}}{2(\bar{\kappa}_P+\nu)}
  +(\gamma+\delta)
  \sup\limits_{N\in\N}\sup\limits_{t\in[0,\infty)}
  \E\left[\sum_{i\in\hat{\D}}\sigma_i\tfrac{1}{H_t^N(i)}\right]\right).
\end{split} \end{equation}
\end{lemma}
The following theorem implies
for every $t\in[0,\infty)$
that the $L^2([0,t]\times l_{\sigma}^1\times\Omega;\R)$-distance
between
$(H^N_{\cdot N},P^N_{\cdot N})$
and
$(h_{\infty}(F^N_{\cdot N}),p_{\infty}(F^N_{\cdot N}))$ converges to $0$ as $N\to\infty$
at least with rate $\tfrac{1}{2}$; cf.\ also
Theorem~\ref{thm:sup.N.int.(H-h)^2.bounded}.

\begin{theorem} \label{thm:bound.V}
Assume the setting of Section \ref{sec:convF_setting},
let $(h_{\infty},p_{\infty})$ satisfy
\eqref{eq:def.h_infty.p_infty},
let Assumption \ref{ass:conv_LotkaVolterra} hold
and let $u\colon (0,\infty)^2\times[0,1]\to[0,\infty)$ satisfy
for all $(x,y,z)\in(0,\infty)^2\times[0,1]$ that
\begin{equation} \label{eq:def_lyapunov_function}
  u(x,y,z):=
  (\eta-\rho z)\left(x-h_\infty(z)-h_\infty(z)\ln\left(\tfrac{x}{h_{\infty}(z)}\right)\right)+\delta\left(y-p_\infty(z)-p_{\infty}(z)
  \ln\left(\tfrac{y}{p_{\infty}(z)}\right)\right).
\end{equation}
Then $u$ is well-defined and
there exists a constant $c_0\in(0,\infty)$
such that for every set $\hat{\D}\subseteq\D$,
for every $N\in\N$,
and every $t\in[0,\infty)$
it holds that
\begin{equation} \begin{split}
  &\E\Bigg[
  \sum\limits_{i\in\hat{\D}}\sigma_i
  u\left(H_t^N(i),P_t^N(i),F_t^N(i)\right)
  \Bigg]
  \\
  &\quad
  +
  \int_0^t
  \left(\eta-\rho\right)\tfrac{\lambda}{K}
  \E\left[\sum_{i\in\hat{\D}}\sigma_i\left(H_u^N(i)-h_\infty\left(F_u^N(i)\right)\right)^2\right]
  +\delta\gamma\E\left[\sum_{i\in\hat{\D}}\sigma_i\left(P_u^N(i)-p_\infty\left(F_u^N(i)\right)\right)^2\right]
  \,du
  \\
  &\leq
  \sup_{M\in\N}\E\Bigg[
  \sum\limits_{i\in\hat{\D}}\sigma_i
  u\left(H_0^M(i),P_0^M(i),F_0^M(i)\right)
  \Bigg]
  +tc_0\max\left\{\kappa_H^N,\kappa_P^N,\alpha^N,\iota_H^N,\iota_P^N,\beta_H^N,\beta_P^N\right\}<\infty.
\end{split} \end{equation}
\end{theorem}
\begin{proof}
The main steps of this proof are  Ito's lemma (see \eqref{eq:applyIto} below)
applied to the Lyapunov-type function \eqref{eq:def.V},
 the specific relation \eqref{eq:calculation.lambda.nu}
 of the equilibrium state
 and control of inverse moments appearing in remainder terms by 
Lemmas \ref{lem:bound.norm.(H+P)^p}, \ref{lem:bound.norm.1/H^2}, and \ref{lem:bound.norm.1/P}.
For the rest of this proof fix a set $\hat{\D}\subseteq\D$.
Define $\D_0:=\emptyset$ and 
for every $n\in\N$ let $\D_n\subseteq\hat{\D}$ be a set with $|\D_n|=\min\left\{n,|\hat{\D}|\right\}$ and $\D_n\supseteq\D_{n-1}$.
Assume that $\hat{\D}=\cup_{n\in\N}\D_n$.
We will first show that $u$ is well-defined.
Define for all $x\in(0,\infty)$ the real-valued function $(0,\infty)\ni y\mapsto f_x(y):=x-y-y\ln\left(\tfrac{x}{y}\right)$.
For all $x\in(0,\infty)$ the function $f_x$ has for all $y\in(0,\infty)$ first and second order derivatives
$\tfrac{df_x}{dy}(y)=\ln(y)-\ln(x)$
and
$\tfrac{d^2f_x}{d y^2}(y)=\tfrac{1}{y}>0$.
Thus, for all $x\in(0,\infty)$ the function $f_x$ has
its global minimum at $x$ with $f_x(x)=0$.
Consequently, for any $(x,y)\in(0,\infty)^2$ we have $f_x(y)\geq f_x(x)=0$.
This shows that for all $(x,y,z)\in(0,\infty)^2\times[0,1]$ we have that $u(x,y,z)\geq 0$.
In order to prove the second part of the claim, we will make use of a Lyapunov function that is defined here analogously to the well-known Lyapunov function in the deterministic setting.
Define $D_V:=\Big(l_\sigma^1\cap (0,\infty)^\D\Big) \times \Big(l_\sigma^1\cap (0,\infty)^\D\Big) \times E_1$.
For any subset $\hat{\D}'\subseteq\hat{\D}$
define the function $V_{\hat{\D}'}\colon D_V\to [0,\infty]$
for any $(h, p, f)\in D_V$ by
\begin{equation} \begin{split} \label{eq:def.V}
  V_{\hat{\D}'}((h, p, f))
  &:=
  \sum\limits_{i\in\hat{\D}'}\sigma_iu(h_i,p_i,f_i).
\end{split} \end{equation}
Due to the non-negativity of the mapping $u$,
we obtain for any $\hat{\D}'\subseteq\hat{\D}$ and any $z\in D_V$ that
$V_{\hat{\D}'}(z)\in[0,\infty]$ is well-defined.
The fact that for all $x\in(0,\infty)$ it holds that $-\ln(x)\leq
\sqrt{\tfrac{1}{x}}$
and Young's inequality imply for all $x,y\in(0,\infty)$ that
\begin{equation}  \begin{split}
  |f_x(y)|=f_x(y)=x-y-y\ln(\tfrac{x}{y})
  \leq x+y\sqrt{\tfrac{y}{x}}\leq 2+x^4+\tfrac{y^3}{x}.
\end{split}     \end{equation}
This, the fact that $h_{\infty}$, $p_{\infty}$ are bounded and the assumption
$\sup_{N\in\N}\E\big[\big\|
\left(H_0^N+P_0^N\right)^4+\tfrac{1}{(H_0^N)^2}+\tfrac{1}{P_0^N}
\big\|_{\sigma}\big]<\infty$
yield that
\begin{align}
  &\nonumber\sup\limits_{N\in\N}\E\left[V_{\D}\left(H_0^N,P_0^N,F_0^N\right)\right]\\
  & \label{eq:bound.V0}\leq\sup\limits_{N\in\N}
  \E\Big[\sum_{i\in\mathcal{D}}\sigma_i\Big(\eta\Big(
     2+|H_0^N(i)|^4+\tfrac{|h_{\infty}(F_0^N(i))|^3}{H_0^N(i)}\Big)
     +\delta\Big(
     2+|P_0^N(i)|^4+\tfrac{|p_{\infty}(F_0^N(i))|^3}{P_0^N(i)}\Big)
     \Big]\\
  &\leq
  (\eta+\delta)\sup_{N\in\N}\Big(5+\E\Big[\|(H_0^N)^4\|_{\sigma}\Big]
  +\E\Big[\|\tfrac{\sup_{z\in[0,1]}h_{\infty}(z)}{(H_0^N)^2}\|_{\sigma}\Big]
  +
  \E\Big[\|(P_0^N)^4\|_{\sigma}\Big]
  +\E\Big[\|\tfrac{\sup_{z\in[0,1]}p_{\infty}(z)}{P_0^N}\|_{\sigma}\Big]
  \Big)
  <\infty.\nonumber
\end{align}
We now calculate the first and second order partial derivatives that we will need in the application of It\^o's lemma below.
For all $n\in\N$, $z=(h, p, f)\in D_V$, and $i\in\D_n$ we get
$\tfrac{dV_{\D_n}}{dh_i}(z)=\sigma_i(\eta-\rho f_i)\Big(1-\tfrac{h_\infty(f_i)}{h_i}\Big)$,
$\tfrac{d^2V_{\D_n}}{dh_i^2}(z)=\sigma_i(\eta-\rho f_i)\tfrac{h_\infty(f_i)}{h_i^2}$,
$\tfrac{dV_{\D_n}}{dp_i}(z)=\sigma_i\delta\Big(1-\tfrac{p_\infty(f_i)}{p_i}\Big)$, and
$\tfrac{d^2V_{\D_n}}{dp_i^2}(z)=\sigma_i\delta\tfrac{p_\infty(f_i)}{p_i^2}$
as well as 
\begin{equation} \begin{split}
  &\tfrac{dV_{\D_n}}{df_i}(z)=\sigma_i
  \Bigg[
  -\rho\left(h_i-h_\infty(f_i)-h_\infty(f_i)\ln\left(\tfrac{h_i}{h_{\infty}(f_i)}\right)\right)
  +(\eta-\rho f_i)
  \Bigg(
  - h_\infty'(f_i)- h_\infty'(f_i)\ln\left(\tfrac{h_i}{h_\infty(f_i)}\right)
  \\
  &\quad
  -\tfrac{(h_\infty(f_i))^2}{h_i}\tfrac{-h_i}{(h_\infty(f_i))^2} h_\infty'(f_i)
  \Bigg)
  +\delta\left(
  - p_\infty'(f_i)- p_\infty'(f_i)\ln\left(\tfrac{p_i}{p_\infty(f_i)}\right)
  -\tfrac{(p_\infty(f_i))^2}{p_i}\tfrac{-p_i}{(p_\infty(f_i))^2} p_\infty'(f_i)
  \right)
  \Bigg]
  \\
  &
  =\sigma_i\Bigg[
  -\rho\left(h_i-h_\infty(f_i)-h_\infty(f_i)\ln\left(\tfrac{h_i}{h_{\infty}(f_i)}\right)\right)
  -(\eta-\rho f_i) h_\infty'(f_i)\ln\left(\tfrac{h_i}{h_\infty(f_i)}\right)
  -\delta p_\infty'(f_i)\ln\left(\tfrac{p_i}{p_\infty(f_i)}\right)
  \Bigg]
\end{split} \end{equation}
and
\begin{equation} \begin{split}
  \tfrac{d^2V_{\D_n}}{df_i^2}(z)
  =&
  \sigma_i
  \Bigg[
  \rho\left(
   h_\infty'(f_i)+ h_\infty'(f_i)\ln\left(\tfrac{h_i}{h_\infty(f_i)}\right)
  +\tfrac{(h_\infty(f_i))^2}{h_i}(-1)\tfrac{h_i}{(h_\infty(f_i))^2} h_\infty'(f_i)
  \right)
  +\rho h_\infty'(f_i)\ln\left(\tfrac{h_i}{h_\infty(f_i)}\right)
  \\
  &-
  (\eta-\rho f_i)\left(
   h_\infty''(f_i)\ln\left(\tfrac{h_i}{h_\infty(f_i)}\right)
  + h_\infty'(f_i)\tfrac{h_\infty(f_i)}{h_i}\tfrac{-h_i}{(h_\infty(f_i))^2} h_\infty'(f_i)
  \right)
  \\
  &-
  \delta\left(
   p_\infty''(f_i)\ln\left(\tfrac{p_i}{p_\infty(f_i)}\right)
  + p_\infty'(f_i)\tfrac{p_\infty(f_i)}{p_i}\tfrac{-p_i}{(p_\infty(f_i))^2} p_\infty'(f_i)
  \right)
  \Bigg]
  \\
  =&
  \sigma_i\Bigg[
  2\rho h_\infty'(f_i)\ln\left(\tfrac{h_i}{h_\infty(f_i)}\right)
  -(\eta-\rho f_i)\left(
   h_\infty''(f_i)\ln\left(\tfrac{h_i}{h_\infty(f_i)}\right)
  -\left( h_\infty'(f_i)\right)^2\tfrac{1}{h_\infty(f_i)}
  \right)
  \\
  &-
  \delta\left(
   p_\infty''(f_i)\ln\left(\tfrac{p_i}{p_\infty(f_i)}\right)
  -\left( p_\infty'(f_i)\right)^2\tfrac{1}{p_\infty(f_i)}
  \right)
  \Bigg].
\end{split} \end{equation}
Next we show that It\^o integrals coming from It\^o formula have
vanishing expectation.
Recall that we have for all $x\in[0,1]$ that
$h_\infty(x)=\tfrac{1}{b(a-x)}$ and
$p_\infty(x)=\tfrac{\lambda}{\delta}\left(1-\tfrac{1}{Kb(a-x)}\right)$
and note that the assumption that $\eta-\rho>\tfrac{\nu}{K}$ implies for all $x\in[0,1]$ that $p_\infty(x)>0$.
Therefore, we get for all $x\in[0,1]$,
\begin{equation} \begin{split} \label{eq:derivatives.h_infty.p_infty}
  h_\infty'(x)&=\tfrac{1}{b(a-x)^2}>0,
  \;
  h_\infty''(x)=\tfrac{2}{b(a-x)^3}>0,
  \\
  p_\infty'(x)&=-\tfrac{\lambda}{\delta Kb(a-x)^2}<0,
  \;
  p_\infty''(x)=-\tfrac{2\lambda}{\delta Kb(a-x)^3}<0.
\end{split} \end{equation}
So $h_\infty$, $h_\infty'$, and $h_\infty''$ are strictly increasing on $[0,1]$ while
$p_\infty$, $p_\infty'$, and $p_\infty''$ are strictly decreasing on $[0,1]$.
Also we have that $\max\limits_{x\in[0,1]}\delta p_\infty(x)\leq\lambda$.
Observe that for all $x\in(0,\infty)$ we have $\left|\ln(x)\right|\leq\sqrt{x}+\tfrac{1}{\sqrt{x}}$.
Together with Young's inequality as well as
Lemmas \ref{lem:bound.norm.(H+P)^p}, \ref{lem:bound.norm.1/H^2}, and \ref{lem:bound.norm.1/P}
we get for all $t\in[0,\infty)$ and all $N,n\in\N$ that
\begin{equation} \begin{split}
  &\E\left[\sum_{i\in\D_n}\sigma_i
  \int_0^t\left(\sqrt{\beta_P^N P_u^N(i)}\delta\left(1-\tfrac{p_\infty(F_u^N(i))}{P_u^N(i)}\right)\right)^2\,du\right]
  \leq
  \betab_P\delta^2\E\left[\sum_{i\in\D_n}\sigma_i\int_0^tP_u^N(i)\left(1+\tfrac{\left(p_\infty(0)\right)^2}{\left(P_u^N(i)\right)^2}\right)\,du\right]
  \\
  &\quad
  \leq
  \betab_P\delta^2\sup_{u\in[0,t]}\E\left[\sum_{i\in\D_n}\sigma_i
  t\left(P_u^N(i)+\tfrac{\left(p_\infty(0)\right)^2}{P_u^N(i)}\right)
  \right]
  <\infty
\end{split} \end{equation}
and
\begin{equation} \begin{split}
  &\E\left[\sum_{i\in\D_n}\sigma_i\int_0^t
  \left(\sqrt{\beta_H^N H_u^N(i)}(\eta-\rho F_u^N(i))\left(1-\tfrac{h_\infty(F_u^N(i))}{H_u^N(i)}\right)\right)^2\,du
  \right]
  \\
  &\quad
  \leq
  \betab_H\eta^2\sup_{u\in[0,t]}\E\left[\sum_{i\in\D_n}\sigma_i
  t\left(H_u^N(i)+\tfrac{\left(h_\infty(1)\right)^2}{H_u^N(i)}\right)\right]
  <\infty
\end{split} \end{equation}
and
\begin{equation} \begin{split}
  &\E\Bigg[
  \sum_{i\in\D_n}\sigma_i\int_0^t
  \Bigg(\sqrt{\tfrac{\beta_H^NF_u^N(i)\left(1-F_u^N(i)\right)}{H_u^N(i)}}
  \Bigg(
  -\rho\left(H_u^N(i)-h_\infty(F_u^N(i))-h_\infty(F_u^N(i))\ln\left(\tfrac{H_u^N(i)}{h_{\infty}(F_u^N(i))}\right)\right)
  \\
  &\qquad\quad
  -(\eta-\rho F_u^N(i)) h_\infty'(F_u^N(i))\ln\left(\tfrac{H_u^N(i)}{h_\infty(F_u^N(i))}\right)
  -\delta p_\infty'(F_u^N(i))\ln\left(\tfrac{P_u^N(i)}{p_\infty(F_u^N(i))}\right)
  \Bigg)\Bigg)^2\,du\Bigg]
  \\
  &\quad
  \leq
  \betab_H\E\Bigg[\sum_{i\in\D_n}\sigma_i
  \int_0^t
  \tfrac{1}{H_u^N(i)}
  \Bigg(
  \rho H_u^N(i)
  +\rho h_\infty(1)
  +\rho h_\infty(1)\left(\tfrac{\sqrt{H_u^N(i)}}{\sqrt{h_\infty(0)}}+\tfrac{\sqrt{h_\infty(1)}}{\sqrt{H_u^N(i)}}\right)
  \\
  &\qquad\quad  
  +\eta h_\infty'(1)\left(\tfrac{\sqrt{H_u^N(i)}}{\sqrt{h_\infty(0)}}+\tfrac{\sqrt{h_\infty(1)}}{\sqrt{H_u^N(i)}}\right)
  +\delta \left|p_\infty'(1)\right|\left(\tfrac{\sqrt{P_u^N(i)}}{\sqrt{p_\infty(1)}}
  +\tfrac{\sqrt{p_\infty(0)}}{\sqrt{P_u^N(i)}}\right)
  \Bigg)^2
  \,du\Bigg]
  \\
  &\quad
  \leq
  \betab_H\sup_{u\in[0,t]}\E\Bigg[\sum_{i\in\D_n}\sigma_i2^7
  \Bigg(
  \rho^2 H_u^N(i)
  +\rho^2 \left(h_\infty(1)\right)^2
  \left(\tfrac{1}{H_u^N(i)}+\tfrac{1}{h_\infty(0)}+\tfrac{h_\infty(1)}{\left(H_u^N(i)\right)^2}\right)
  \\
  &\qquad\quad  
  +\eta^2 \left(h_\infty'(1)\right)^2\left(\tfrac{1}{h_\infty(0)}+\tfrac{h_\infty(1)}{\left(H_u^N(i)\right)^2}\right)
  +\delta^2 \left(p_\infty'(1)\right)^2\left(\tfrac{1}{p_\infty(1)}\tfrac{P_u^N(i)}{H_u^N(i)}
  +\tfrac{p_\infty(0)}{P_u^N(i)H_u^N(i)}\right)
  \Bigg)
  \Bigg]
  <\infty.
\end{split} \end{equation}
Hence, we obtain for all $t\in[0,\infty)$ and all $N,n\in\N$ that
\begin{equation} \begin{split} \label{eq:V.stoch.int.0}
  &\E\left[\int_0^t
  \sum_{i\in\D_n}\sigma_i\sqrt{\beta_P^N P_u^N(i)}\delta\left(1-\tfrac{p_\infty(F_u^N(i))}{P_u^N(i)}\right)
  \,dW_u^{P,N}(i)\right]
  =0,
  \\
  &\E\left[\int_0^t
  \sum_{i\in\D_n}\sigma_i\sqrt{\beta_H^N H_u^N(i)}(\eta-\rho F_u^N(i))\left(1-\tfrac{h_\infty(F_u^N(i))}{H_u^N(i)}\right)
  \,dW_u^{H,N}(i)\right]
  =0,
  \\
  &\E\left[\int_0^t
  \sum_{i\in\D_n}\sigma_i\sqrt{\tfrac{\beta_H^NF_u^N(i)\left(1-F_u^N(i)\right)}{H_u^N(i)}}
  \Bigg[
  -\rho\left(H_u^N(i)-h_\infty(F_u^N(i))-h_\infty(F_u^N(i))\ln\left(\tfrac{H_u^N(i)}{h_{\infty}(F_u^N(i))}\right)\right)
  \right.
  \\
  &\qquad
  \left.
  -(\eta-\rho F_u^N(i)) h_\infty'(F_u^N(i))\ln\left(\tfrac{H_u^N(i)}{h_\infty(F_u^N(i))}\right)
  -\delta p_\infty'(F_u^N(i))\ln\left(\tfrac{P_u^N(i)}{p_\infty(F_u^N(i))}\right)
  \Bigg]
  \,dW_u^{F,N}(i)\right]
  =0.
\end{split} \end{equation}
Next we bound certain remainder terms.
For all $t\in[0,\infty)$, all $N\in\N$, and all $i\in\D$
define
\begin{equation} \begin{split} \label{eq:def.barh.barp.R.b}
  R_t^N(i)&:=\max\Bigg\{
  \max\Big\{\eta c,\rho c,\rho,
  c\eta \tfrac{h_\infty'\left(1\right)}{h_\infty\left(0\right)},
  \eta \tfrac{h_\infty'\left(1\right)}{h_\infty\left(0\right)}\Big\} H_t^N(i),
  \eta h_\infty\left(1\right),
  \eta,
  \\
  &\qquad\quad
  \max\Big\{\tfrac{\eta}{2} h_\infty\left(1\right), \rho\left(h_\infty\left(1\right)\right)^2,
  \tfrac{\eta}{2}\tfrac{\left(h_\infty'\left(1\right)\right)^2}{h_\infty\left(0\right)},
  \tfrac{\delta}{2}\tfrac{\left(p_\infty'(1)\right)^2}{p_\infty(1)}\Big\} \tfrac{1}{H_t^N(i)},
  \delta cP_t^N(i),
  \delta p_\infty\left(0\right),
  \delta,
  \tfrac{\delta}{2}\tfrac{p_\infty\left(0\right)}{P_t^N(i)},
  \\
  &\qquad\quad
  \max\Big\{
  \tfrac{1}{2}\rho^2 \left(h_\infty\left(1\right)\right)^2,
  \tfrac{3}{4}\rho^{\frac{4}{3}} \left(h_\infty\left(1\right)\right)^2,
  \tfrac{3}{4}\left(\eta h_\infty'\left(1\right)\sqrt{h_\infty\left(1\right)}\right)^{\frac{4}{3}},
  \tfrac{1}{4},
  \tfrac{\eta}{2}h_\infty''\left(1\right)h_\infty\left(1\right)
  \Big\}\tfrac{1}{\left(H_t^N(i)\right)^2},
  \\
  &\qquad\quad
  \tfrac{1}{2}c\left(H_t^N(i)\right)^2,
  \tfrac{1}{4}c\left(H_t^N(i)\right)^4,
  \tfrac{1}{2}\left(\tfrac{\delta \left|p_\infty'\left(1\right)\right|}{\sqrt{p_\infty\left(1\right)}}\right)^2\tfrac{P_t^N(i)}{\left(H_t^N(i)\right)^2},
  \tfrac{1}{2}\delta^2 \left|p_\infty'\left(1\right)\right|^2\tfrac{p_\infty\left(0\right)}{P_t^N(i)H_t^N(i)},
  \\
  &\qquad\quad
  \delta \left|p_\infty'\left(1\right)\right|\sqrt{\tfrac{p_\infty\left(0\right)}{P_t^N(i)}},
  \rho \tfrac{h_\infty'\left(1\right)}{h_\infty\left(0\right)},
  \tfrac{\delta}{2}\tfrac{\left|p_\infty''(1)\right|}{p_\infty(1)}\tfrac{P_t^N(i)}{H_t^N(i)},
  \Bigg\},
  \\
  b^N&:=\max\left\{\kappa_H^N,\kappa_P^N,\alpha^N,\iota_H^N,\iota_P^N,\beta_H^N,\beta_P^N\right\}.
\end{split} \end{equation}
Note that $\lim\limits_{N\to\infty} b^N=0$.
Define $c_0:=32\sup\limits_{M\in\N}\sup\limits_{u\in[0,\infty)}\E\left[\left\|R_u^M\right\|_\sigma\right]$.
Observe that due to Lemmas \ref{lem:bound.norm.(H+P)^p}, \ref{lem:bound.norm.1/H^2}, and \ref{lem:bound.norm.1/P}
we have $c_0\in(0,\infty)$.
For all $t\in[0,\infty)$, all $N\in\N$, and all
$a\in\Big\{\eta,\rho,\eta \tfrac{h_\infty'\left(1\right)}{h_\infty\left(0\right)}\Big\}$
we have that
\begin{equation} \label{eq:estimate.R.1}
  \sum_{i\in\D}\sigma_ia\sum_{j\in\D} m(i,j)H_t^N(j)
  \leq
  \sum_{i\in\D}\sigma_icaH_t^N(i)
  \leq
  \sum_{i\in\D}\sigma_iR_t^N(i).
\end{equation}
Furthermore, we have for all $t\in[0,\infty)$ and all $N\in\N$ that
\begin{equation} \label{eq:estimate.R.2}
  \sum_{i\in\D}\sigma_i\delta\sum_{j\in\D} m(i,j)P_t^N(j)
  \leq
  \sum_{i\in\D}\sigma_i\delta cP_t^N(i)
  \leq
  \sum_{i\in\D}\sigma_iR_t^N(i).
\end{equation}
Using Young's inequality and Lemma \ref{lem:estimate.term.sum^p}
we get for all $t\in[0,\infty)$ and all $N\in\N$ that
\begin{equation} \begin{split} \label{eq:estimate.R.3}
  \sum_{i\in\D}\sigma_i
  \rho \tfrac{h_\infty\left(F_t^N(i)\right)}{H_t^N(i)}
  \sum\limits_{j\in\D}m(i,j)H_t^N(j)
  &\leq
  \sum_{i\in\D}\sigma_i
  \Bigg(
  \tfrac{1}{2}
  \Big(\rho \tfrac{h_\infty\left(F_t^N(i)\right)}{H_t^N(i)}\Big)^2
  +\tfrac{1}{2}
  \Big(\sum\limits_{j\in\D}m(i,j)H_t^N(j)\Big)^2
  \Bigg)
  \\
  &\leq
  \sum_{i\in\D}\sigma_i
  \Bigg(
  R_t^N(i)
  +
  \tfrac{1}{2}c\left(H_t^N(i)\right)^2
  \Bigg)
  \leq
  \sum_{i\in\D}\sigma_i2R_t^N(i),
\end{split} \end{equation}
and
\begin{equation} \begin{split} \label{eq:estimate.R.4}
  \sum_{i\in\D}\sigma_i
  \tfrac{-\delta p_\infty'\left(F_t^N(i)\right)}{\sqrt{p_\infty\left(F_t^N(i)\right)}}\tfrac{\sqrt{P_t^N(i)}}{H_t^N(i)}\sum\limits_{j\in\D}m(i,j)H_t^N(j)
  &\leq
  \sum_{i\in\D}\sigma_i
  \tfrac{1}{2}\Bigg(
  \Big(
  \tfrac{\delta \left|p_\infty'\left(1\right)\right|}{\sqrt{p_\infty\left(1\right)}}\tfrac{\sqrt{P_t^N(i)}}{H_t^N(i)}
  \Big)^2
  +
  \Big(\sum\limits_{j\in\D}m(i,j)H_t^N(j)\Big)^2
  \Bigg)
  \\
  &\leq
  \sum_{i\in\D}\sigma_i\Bigg(R_t^N(i)+\tfrac{1}{2}c\left(H_t^N(i)\right)^2\Bigg)
  \leq
  \sum_{i\in\D}\sigma_i2R_t^N(i),
\end{split} \end{equation}
and
\begin{equation} \begin{split} \label{eq:estimate.R.5}
  &\sum_{i\in\D}\sigma_i
  (-1)\delta p_\infty'\left(F_t^N(i)\right)\tfrac{\sqrt{p_\infty\left(F_t^N(i)\right)}}{\sqrt{P_t^N(i)}H_t^N(i)}\sum\limits_{j\in\D}m(i,j)H_t^N(j)
  \\
  &\quad
  \leq
  \sum_{i\in\D}\sigma_i
  \Bigg(
  \tfrac{1}{2}\delta^2 \left|p_\infty'\left(1\right)\right|^2\tfrac{p_\infty\left(0\right)}{P_t^N(i)H_t^N(i)}
  +\tfrac{1}{2}\Big(
  \sum\limits_{j\in\D}m(i,j)H_t^N(j)\Big)^2\tfrac{1}{H_t^N(i)}
  \Bigg)
  \\
  &\quad
  \leq
  \sum_{i\in\D}\sigma_i
  \Bigg(
  R_t^N(i)
  +\tfrac{1}{4}\Big(
  \sum\limits_{j\in\D}m(i,j)H_t^N(j)\Big)^4
  +\tfrac{1}{4}\tfrac{1}{\left(H_t^N(i)\right)^2}
  \Bigg)
  \leq
  \sum_{i\in\D}\sigma_i
  \Bigg(
  R_t^N(i)
  +\tfrac{1}{4}c\left(H_t^N(i)\right)^4
  +R_t^N(i)
  \Bigg)
  \\
  &\quad
  \leq
  \sum_{i\in\D}\sigma_i
  3R_t^N(i).
\end{split} \end{equation}
Again using Young's inequality and Lemma \ref{lem:estimate.term.sum^p} we get for all
$a\in\Big\{\rho \big(h_\infty\left(1\right)\big)^{\frac{3}{2}},\eta h_\infty'(1)\sqrt{h_\infty(1)}\Big\}$,
all $t\in[0,\infty)$, and all $N\in\N$
that
\begin{equation} \begin{split} \label{eq:estimate.R.6}
  \sum_{i\in\D}\sigma_i
  a\Big(\tfrac{1}{H_t^N(i)}\Big)^{\frac{3}{2}}
  \sum\limits_{j\in\D}m(i,j)H_t^N(j)
  &\leq
  \sum_{i\in\D}\sigma_i
  \Bigg(
  \tfrac{3}{4}
  a^{\frac{4}{3}} \Big(\tfrac{1}{H_t^N(i)}\Big)^2
  +
  \tfrac{1}{4}\Big(\sum\limits_{j\in\D}m(i,j)H_t^N(j)\Big)^4
  \Bigg)
  \\
  &\leq
  \sum_{i\in\D}\sigma_i
  \Bigg(
  \R_t^N(i)
  +\tfrac{1}{4}c\left(H_t^N(i)\right)^4
  \Bigg)
  \leq
  \sum_{i\in\D}\sigma_i2R_t^N(i).
\end{split} \end{equation}
Due to Lemma \ref{lem:HFP} we have that $W^{H,N}(i)$, $W^{F,N}(i)$, $N\in\N$, $i\in\D$, are independent Brownian motions
and due to
Lemmas \ref{lem:bound.norm.(H+P)^p}, \ref{lem:bound.norm.1/H^2}, and \ref{lem:bound.norm.1/P}
we have for all $t\in[0,\infty)$ and all $N\in\N$ that $\P$-a.s.~$\left(H_t^N, P_t^N, F_t^N\right)\in D_V$.
Thus, applying It\^o's lemma and using \eqref{eq:V.stoch.int.0}
we obtain for all $t\in[0,\infty)$ and all $N,n\in\N$ that
\begin{equation} \begin{split}\label{eq:applyIto}
  &\E\left[V_{\D_n}\left(\left(H_t^N, P_t^N, F_t^N\right)\right)\right]
  -\E\left[V_{\D_n}\left(\left(H_0^N, P_0^N, F_0^N\right)\right)\right]
  \\
  &=\E\Bigg[\int_0^t
  \sum_{i\in\D_n}\sigma_i\Bigg(
  \left(\eta-\rho F_u^N(i)\right)\left(1-\tfrac{h_\infty\left(F_u^N(i)\right)}{H_u^N(i)}\right)
  \Big\{\kappa_H^N\sum_{j\in\D} m(i,j)\left(H_u^N(j)-H_u^N(i)\right)
  \\
  &\quad
  +H_u^N(i)\left[\lambda\left(1-\tfrac{H_u^N(i)}{K}\right)
  -\delta P_u^N(i)-\alpha^NF_u^N(i)\right]
  +\iota_H^N\Big\}
  +\tfrac{\left(\eta-\rho F_u^N(i)\right)}{2}\tfrac{h_\infty\left(F_u^N(i)\right)}{\left(H_u^N(i)\right)^2}
  \beta_H^NH_u^N(i)
  \\
  &\quad
  +\delta\left(1-\tfrac{p_\infty\left(F_u^N(i)\right)}{P_u^N(i)}\right)
  \Big\{\kappa_P^N\sum_{j\in\D} m(i,j)\left(P_u^N(j)-P_u^N(i)\right)
  \\
  &\quad
  +P_u^N(i)\left[-\nu-\gamma P_u^N(i)
  +\left(\eta-\rho F_u^N(i)\right)H_u^N(i)\right] 
  +\iota_P^N\Big\}
  +\tfrac{\delta}{2}\tfrac{p_\infty(F_u^N(i))}{\left(P_u^N(i)\right)^2}
  \beta_P^NP_u^N(i)
  +\Bigg[
  -\rho\Big(
  H_u^N(i)
  \\
  &\quad
  -h_\infty\left(F_u^N(i)\right)
  -h_\infty\left(F_u^N(i)\right)\ln\left(\tfrac{H_u^N(i)}{h_\infty\left(F_u^N(i)\right)}\right)
  \Big)
  -\left(\eta-\rho F_u^N(i)\right)
  h_\infty'\left(F_u^N(i)\right)\ln\left(\tfrac{H_u^N(i)}{h_\infty\left(F_u^N(i)\right)}\right)
  \\
  &\quad
  -\delta p_\infty'\left(F_u^N(i)\right)\ln\left(\tfrac{P_u^N(i)}{p_\infty\left(F_u^N(i)\right)}\right)
  \Bigg]
  \Big\{
  \kappa_H^N
  \sum\limits_{j\in\D}m(i,j)\left(F_u^N(j)-F_u^N(i)\right)\tfrac{H_u^N(j)}{H_u^N(i)}
  -\alpha^NF_u^N(i)\left(1-F_u^N(i)\right)
  \Big\}
  \\
  &\quad
  +\Big\{
  \rho h_\infty'\left(F_u^N(i)\right)\ln\left(\tfrac{H_u^N(i)}{h_\infty\left(F_u^N(i)\right)}\right)
  -\tfrac{\eta-\rho F_u^N(i)}{2}
  \Big(
   h_\infty''\left(F_u^N(i)\right)\ln\left(\tfrac{H_u^N(i)}{h_\infty\left(F_u^N(i)\right)}\right)
  -\tfrac{\left(h_\infty'\left(F_u^N(i)\right)\right)^2}{h_\infty\left(F_u^N(i)\right)}
  \Big)
  \\
  &\quad
  -\tfrac{\delta}{2}\Big(
  p_\infty''(F_u^N(i))\ln\left(\tfrac{P_u^N(i)}{p_\infty(F_u^N(i))}\right)
  -\tfrac{\left(p_\infty'(F_u^N(i))\right)^2}{p_\infty(F_u^N(i))}
  \Big)
  \Big\}
  \tfrac{\beta_H^N F_u^N(i)\left(1-F_u^N(i)\right)}{H_u^N(i)}
  \Bigg)
  \,du
  \Bigg].
\end{split} \end{equation}
Next we show that the right-hand side is negative up to small remainder terms.
Note that for all $x\in[0,1]$ it holds that $0< \eta-\rho x\leq \eta$.
Together with the fact that for all $x\in(0,\infty)$ we have $\ln(x)\leq\sqrt{x}$, $\ln(x)\leq x$,
$\Big|\ln(x)\Big|\leq \sqrt{x}+\sqrt{\tfrac{1}{x}}$, and
$\Big|\ln(x)\Big|\leq x+\sqrt{\tfrac{1}{x}}$
and dropping negative terms,
this implies for all $t\in[0,\infty)$ and all $N,n\in\N$ that
\begin{equation} \begin{split}
  &\E\left[V_{\D_n}\left(\left(H_t^N, P_t^N, F_t^N\right)\right)\right]
  -\E\left[V_{\D_n}\left(\left(H_0^N, P_0^N, F_0^N\right)\right)\right]
  \\
  &\leq
  \E\Bigg[\int_0^t\sum_{i\in\D_n}\sigma_i\Bigg(
  \eta\kappa_H^N\sum_{j\in\D} m(i,j)H_u^N(j)
  +\eta h_\infty\left(F_u^N(i)\right)\kappa_H^N
  \\
  &\quad
  +\left(\eta-\rho F_u^N(i)\right)\left(H_u^N(i)-h_\infty\left(F_u^N(i)\right)\right)
  \left[\lambda\left(1-\tfrac{H_u^N(i)}{K}\right)
  -\delta P_u^N(i)\right]
  +\eta h_\infty\left(F_u^N(i)\right)\alpha^N
  +\eta\iota_H^N
  \\
  &\quad
  +\tfrac{\eta}{2} h_\infty\left(F_u^N(i)\right)\beta_H^N\tfrac{1}{H_u^N(i)}
  +\delta\kappa_P^N\sum_{j\in\D} m(i,j)P_u^N(j)
  +\delta p_\infty\left(F_u^N(i)\right)\kappa_P^N
  \\
  &\quad
  +\delta\left(P_u^N(i)-p_\infty\left(F_u^N(i)\right)\right)\left[-\nu-\gamma P_u^N(i)
  +\left(\eta-\rho F_u^N(i)\right)H_u^N(i)\right]
  +\delta\iota_P^N
  +\tfrac{\delta}{2}\tfrac{p_\infty\left(F_u^N(i)\right)}{P_u^N(i)}
  \beta_P^N
  \\
  &\quad
  +\rho \kappa_H^N
  \sum\limits_{j\in\D}m(i,j)H_u^N(j)
  +\rho H_u^N(i)\alpha^N
  +\rho \tfrac{h_\infty\left(F_u^N(i)\right)}{H_u^N(i)}
  \Bigg(1+\tfrac{H_u^N(i)}{h_\infty\left(F_u^N(i)\right)}
  \\
  &\quad
  +\sqrt{\tfrac{h_\infty\left(F_u^N(i)\right)}{H_u^N(i)}}\Bigg)
  \kappa_H^N
  \sum\limits_{j\in\D}m(i,j)H_u^N(j)
  +\rho h_\infty\left(F_u^N(i)\right)
  \tfrac{h_\infty\left(F_u^N(i)\right)}{H_u^N(i)}
  \alpha^N
  +\Bigg[
  \eta h_\infty'\left(F_u^N(i)\right)
  \Bigg(\sqrt{\tfrac{h_\infty\left(F_u^N(i)\right)}{H_u^N(i)}}
  \\
  &\quad
  +\tfrac{H_u^N(i)}{h_\infty\left(F_u^N(i)\right)}
  \Bigg)
  -\delta p_\infty'\left(F_u^N(i)\right)
  \Bigg(\sqrt{\tfrac{P_u^N(i)}{p_\infty\left(F_u^N(i)\right)}}
  +\sqrt{\tfrac{p_\infty\left(F_u^N(i)\right)}{P_u^N(i)}}
  \Bigg)
  \Bigg]
  \kappa_H^N
  \sum\limits_{j\in\D}m(i,j)\tfrac{H_u^N(j)}{H_u^N(i)}
  \\
  &\quad
  +\left[
  \eta
   h_\infty'\left(F_u^N(i)\right)\tfrac{H_u^N(i)}{h_\infty\left(F_u^N(i)\right)}
  -\delta p_\infty'\left(F_u^N(i)\right)\sqrt{\tfrac{p_\infty\left(F_u^N(i)\right)}{P_u^N(i)}}
  \right]
  \alpha^N
  \\
  &\quad
  +\rho h_\infty'\left(F_u^N(i)\right)\tfrac{H_u^N(i)}{h_\infty\left(F_u^N(i)\right)}
  \tfrac{\beta_H^N}{H_u^N(i)}
  +\tfrac{\eta}{2}
  \Big(
   h_\infty''\left(F_u^N(i)\right)\tfrac{h_\infty\left(F_u^N(i)\right)}{H_u^N(i)}
  +\tfrac{\left( h_\infty'\left(F_u^N(i)\right)\right)^2}{h_\infty\left(F_u^N(i)\right)}
  \Big)
  \tfrac{\beta_H^N}{H_u^N(i)}
  \\
  &\quad  
  +\tfrac{\delta}{2}\left(
  - p_\infty''(F_u^N(i))\tfrac{P_u^N(i)}{p_\infty(F_u^N(i))}
  +\tfrac{\left( p_\infty'(F_u^N(i))\right)^2}{p_\infty(F_u^N(i))}
  \right)
  \tfrac{\beta_H^N}{H_u^N(i)}
  \Bigg)
  \,du
  \Bigg].
\end{split} \end{equation}
Using
\eqref{eq:estimate.R.1}, \eqref{eq:estimate.R.2}, \eqref{eq:estimate.R.3},
\eqref{eq:estimate.R.4}, \eqref{eq:estimate.R.5}, and \eqref{eq:estimate.R.6}
we get for all $t\in[0,\infty)$ and all $N,n\in\N$ that
\begin{equation} \begin{split}
  &\E\left[V_{\D_n}\left(\left(H_t^N, P_t^N, F_t^N\right)\right)\right]
  -\E\left[V_{\D_n}\left(\left(H_0^N, P_0^N, F_0^N\right)\right)\right]
  \\
  &\quad
  \leq
  \E\Bigg[\int_0^t\sum_{i\in\D}\sigma_i\Bigg(
  b^N32R_u^N(i)
  +\left(\eta-\rho F_u^N(i)\right)\left(H_u^N(i)-h_\infty\left(F_u^N(i)\right)\right)
  \Big[\lambda\left(1-\tfrac{H_u^N(i)}{K}\right)
  -\delta P_u^N(i)\Big]
  \\
  &\qquad\qquad
  +\delta\left(P_u^N(i)-p_\infty\left(F_u^N(i)\right)\right)
  \Big[-\nu-\gamma P_u^N(i)
  +\left(\eta-\rho F_u^N(i)\right)H_u^N(i)
  \Big] 
  \Bigg)
  \,du
  \Bigg].
\end{split} \end{equation}
Note that 
\eqref{eq:def.h_infty.p_infty}
implies for all $x\in[0,1]$ that
\begin{equation} \begin{split} \label{eq:calculation.lambda.nu}
  &\delta p_\infty(x)+\tfrac{\lambda}{K}h_\infty(x)-\lambda
  =\tfrac{\delta\lambda K(\eta-\rho x)-\delta\lambda\nu+\lambda \delta\nu+\lambda^2\gamma}{\lambda\gamma+\delta K(\eta-\rho x)}-\lambda
  =\tfrac{\delta K(\eta-\rho x)+\lambda\gamma}{\lambda\gamma+\delta K(\eta-\rho x)}\lambda-\lambda
  =0,
  \\
  &\nu-(\eta-\rho x)h_\infty(x)+\gamma p_\infty(x)
  =\nu-\tfrac{(\eta-\rho x)K\delta\nu+(\eta-\rho x)K\gamma\lambda-\gamma\lambda K(\eta-\rho x)+\gamma\lambda\nu}{\lambda\gamma+\delta K(\eta-\rho x)}
  =\nu-\tfrac{(\eta-\rho x)K\delta+\gamma\lambda}{\lambda\gamma+\delta K(\eta-\rho x)}\nu
  =0.
\end{split} \end{equation}
From \eqref{eq:calculation.lambda.nu}
we see that for all $t\in[0,\infty)$ and all $N,n\in\N$ it holds that
\begin{equation} \begin{split}
  &\E\left[V_{\D_n}\left(\left(H_t^N, P_t^N, F_t^N\right)\right)\right]
  -\E\left[V_{\D_n}\left(\left(H_0^N, P_0^N, F_0^N\right)\right)\right]
  \\
  &\leq
  \E\Bigg[\int_0^t\sum_{i\in\D}\sigma_i\Bigg(
  b^N32R_u^N(i)
  +\left(\eta-\rho F_u^N(i)\right)\left(H_u^N(i)-h_\infty\left(F_u^N(i)\right)\right)
  \Big[\lambda-\tfrac{\lambda}{K}\Big(H_u^N(i)-h_\infty\left(F_u^N(i)\right)\Big)
  \\
  &\quad
  -\delta \Big(P_u^N(i)-p_\infty\left(F_u^N(i)\right)\Big)
  -\lambda\Big]
  +\delta\left(P_u^N(i)-p_\infty\left(F_u^N(i)\right)\right)
  \Big[-\nu-\gamma \Big(P_u^N(i)-p_\infty\left(F_u^N(i)\right)\Big)
  \\
  &\quad
  +\left(\eta-\rho F_u^N(i)\right)\Big(H_u^N(i)-h_\infty\left(F_u^N(i)\right)\Big)
  +\nu\Big] 
  \Bigg)
  \,du
  \Bigg].
\end{split} \end{equation}
Hence, we obtain for every $N,n\in\N$ and every $t\in[0,\infty)$ that
\begin{equation} \begin{split}
  &\E\left[V_{\D_n}\left(H_t^N,P_t^N,F_t^N\right)\right]
  +
  \int_0^t
  \left(\eta-\rho\right)\tfrac{\lambda}{K}
  \E\left[\sum_{i\in\D_n}\sigma_i\left(H_u^N(i)-h_\infty\left(F_u^N(i)\right)\right)^2\right]
  \\
  &\enspace
  +\delta\gamma\E\left[\sum_{i\in\D_n}\sigma_i\left(P_u^N(i)-p_\infty\left(F_u^N(i)\right)\right)^2\right]
  \,du
  \leq
  \E\left[V_{\hat{\D}}\left(H_0^N,P_0^N,F_0^N\right)\right]
  +
  tb^N32
  \sup\limits_{M\in\N}
  \sup\limits_{u\in[0,\infty)}
  \E\left[\left\|R_u^M\right\|_\sigma\right].
\end{split} \end{equation}
Applying monotone convergence we now see that for every $N\in\N$ and every $t\in[0,\infty)$ we have
\begin{equation} \begin{split}
  &\E\left[V_{\hat{\D}}\left(H_t^N,P_t^N,F_t^N\right)\right]
  +
  \int_0^t
  \left(\eta-\rho\right)\tfrac{\lambda}{K}
  \E\left[\sum_{i\in\hat{\D}}\sigma_i\left(H_u^N(i)-h_\infty\left(F_u^N(i)\right)\right)^2\right]
  \\
  &\qquad
  +\delta\gamma\E\left[\sum_{i\in\hat{\D}}\sigma_i\left(P_u^N(i)-p_\infty\left(F_u^N(i)\right)\right)^2\right]
  \,du
  \\
  &
  =
  \lim_{n\to\infty}
  \Bigg(
  \E\left[V_{\D_n}\left(H_t^N,P_t^N,F_t^N\right)\right]
  +
  \int_0^t
  \left(\eta-\rho\right)\tfrac{\lambda}{K}
  \E\left[\sum_{i\in\D_n}\sigma_i\left(H_u^N(i)-h_\infty\left(F_u^N(i)\right)\right)^2\right]
  \\
  &\qquad
  +\delta\gamma\E\left[\sum_{i\in\D_n}\sigma_i\left(P_u^N(i)-p_\infty\left(F_u^N(i)\right)\right)^2\right]
  \,du
  \Bigg)
  \leq
  \E\left[V_{\hat{\D}}\left(H_0^N,P_0^N,F_0^N\right)\right]
  +
  tb^N
  c_0.
\end{split} \end{equation}
The set $\hat{\D}\subseteq\D$ was arbitrarily chosen and thus, this 
and \eqref{eq:bound.V0}
complete the proof of Theorem \ref{thm:bound.V}.
\end{proof}
\subsection{Convergence of relative frequency of
defenders
}
\subsubsection{A relative compactness condition}
For convenience of the reader, we restate Lemma 3.3 of
\citet{KlenkeMytnik2012}.
\begin{lemma} \label{lem:set.rel.compactness}
  Let $\D$ be a countable set,
  let $\sigma \in (0,\infty)^\D$ such that $\sum_{i\in\D}\sigma_i<\infty$,
  and let $l_\sigma^1:=\{z\in\R^\D\colon\|z\|_\sigma:=\sum_{i\in\D}\sigma_iz_i<\infty\}$.
  A subset $K\subseteq l_\sigma^1$ is relatively compact if and only if
  \begin{itemize}
    \item[(i)] $\sup_{x\in K}\|x\|_\sigma < \infty$
    \item[(ii)] for every $\eps\in(0,\infty)$ there exists a finite subset $\mathcal{E}\subseteq \D$ such that
      $\sup_{x\in K} \|x\1_{\D\setminus \mathcal{E}}\|_\sigma <\eps$.
  \end{itemize}
\end{lemma}

\begin{lemma} \label{lem:family.rv.rel.compactness}
Let $(\Omega,\mathcal{F},\P)$ be a probability space,
let $\D$ be a countable set,
let $\sigma \in (0,\infty)^\D$ such that $\sum_{i\in\D}\sigma_i<\infty$,
let $l_\sigma^1:=\{z\in\R^\D\colon\|z\|_\sigma:=\sum_{i\in\D}\sigma_iz_i<\infty\}$,
let $E_2:=l_\sigma^1\cap [0,\infty)^\D$,
let $I$ be a set,
and let
$Z^i\colon \Omega\to E_2$, $i\in I$, be a family of random variables.
Assume that
$\sup_{i\in I} \E[\|Z^i\|_\sigma]<\infty$
and
$\inf_{\mathcal{S}\subseteq\D,|\mathcal{S}|<\infty}\sup_{i\in I}\sum_{k\in\D\setminus\mathcal{S}}\sigma_k\E[Z^i_k]=0$.
Then the family $\{Z^i:i\in I\}$ is relatively compact in $E_2$.
\begin{proof}
Fix $\eps\in(0,\infty)$.
For each $m\in\N$ by assumption there exists a finite set
$\mathcal{S}_{m,\eps}\subseteq\D$ such that
\begin{equation}
  \sup_{i\in I} \sum_{k\in\D\setminus\mathcal{S}_{m,\eps}}\sigma_k\E[Z^i_k]<\tfrac{\eps}{2m^2(m+1)}.
\end{equation}
Define the set $K_{\eps}\subseteq E_2$ by
\begin{equation}
  K_{\eps}:=
  \Bigg\{
  x\in E_2:
  \|x\|_\sigma
  \leq
  \tfrac{2\sup_{i\in I} \E[\|Z^i\|_\sigma]}{\eps},
  \sup_{m\in\N}\Big\{m\sum_{k\in\D\setminus\mathcal{S}_{m,\eps}}\sigma_k |x_k|\Big\}\leq 1
  \Bigg\}.
\end{equation}
Due to the Heine--Borel theorem we can apply Lemma \ref{lem:set.rel.compactness}
to obtain relative compactness of $K_{\eps}$.
By Markov's inequality we get
\begin{equation} \begin{split}
  &\sup_{i\in I}\P\Big[Z^i\notin \overline{K_{\eps}}\Big]
  \leq
  \sup_{i\in I}\P\Big[Z^i\notin K_{\eps}\Big]
  \\
  &\quad
  \leq
  \sup_{i\in I}\P\Big[\|Z^i\|_\sigma>\tfrac{2\sup_{j\in I} \E[\|Z^j\|_\sigma]}{\eps}\Big]
  +
  \sup_{i\in I}\sum_{m=1}^\infty\P\Big[\sum_{k\in\D\setminus\mathcal{S}_{m,\eps}}\sigma_k Z^i_k>\tfrac{1}{m}\Big]
  \\
  &\quad
  \leq
  \tfrac{\eps}{2\sup_{j\in I} \E[\|Z^j\|_\sigma]}\sup_{i\in I}\E\Big[\|Z^i\|_\sigma\Big]
  +
  \sum_{m=1}^\infty m\sup_{i\in I}\sum_{k\in\D\setminus\mathcal{S}_{m,\eps}}\sigma_k \E\Big[Z^i_k\Big]
  \leq
  \tfrac{\eps}{2}
  +
  \sum_{m=1}^\infty m\tfrac{\eps}{2m^2(m+1)}
  =\eps.
\end{split} \end{equation}
Since $\eps$ was arbitrarily chosen it follows that
$\{Z^i:i\in I\}$ is tight in $E_2$. Due to Prohorov's theorem
\citep[e.g., Theorem 3.2.2 in][]{EthierKurtz1986}
the claim follows.
\end{proof}
\end{lemma}

\subsubsection{Proof of Theorem \ref{thm:conv.freq.altruist}} \label{sec:convF_proof}
In our proof of Theorem \ref{thm:conv.freq.altruist} we will apply Lemma
\ref{lem:family.rv.rel.compactness} to show for every $T\in[0,\infty)$
that $\{H_{tN}^N\colon t\in[0,T],N\in\N\}$ is relatively compact in
$l_{\sigma}^1\cap[0,\infty)^{\mathcal{D}}$. The difficult part is the last condition
of Lemma
\ref{lem:family.rv.rel.compactness} which leads us to show for every $T\in(0,\infty)$
that $\sum_{i\in\mathcal{D}}\sup_{N\in\N,t\in[0,T]}\E[H_{tN}^N(i)]<\infty$.
The following lemma will allow us to prove this condition where supremum over $N$ is inside
the sum over the demes.
\begin{lemma} \label{lem:unif.pos.moments.H}
  Assume the setting of Section \ref{sec:convF_setting} and assume
  that for all $N\in\N$ we have
  $\sum_{i\in\D}\sigma_i\E\Big[H_0^N(i)\Big]<\infty$.
  For all $n\in\N$ denote by $m^n$ the $n$-fold matrix product of $m$.
  Then we get for all $t\in[0,\infty)$, all $i\in\D$, and all $N\in\N$ that
  \begin{equation} \begin{split}
  \E\left[H_t^N(i)\right]
  \leq
  \E\left[\sum_{j\in\D}\sum_{n=0}^\infty e^{-t\kappa_H^N}\tfrac{(t\kappa_H^N)^n}{n!}m^n(i,j)H_0^N(j)\right]
  +\tfrac{K}{2}\Big(1+\sqrt{1+\tfrac{4\bar{\iota}_H}{K\lambda}}\Big).
  \end{split} \end{equation}
\end{lemma}
\begin{proof}
  If $m$ is the identity matrix, then the
  assertion follows essentially from Gronwall's lemma
  since the drift grows linearly in upward direction.
  In the general case we average over space with the migration semigroup
  \eqref{eq:migration.semigroup} to get rid of the term involving $m$;
  cf.\ the cancelation in the second step of
  \eqref{eq:bound.Y}.
  We have for every $n\in\N$ and every $i,j\in\D$ that $m^n(i,j)\in[0,1]$.
  Hence, we get for all $T\in[0,\infty)$ and all $i,j\in\D$ that
  \begin{equation} \label{eq:m_t.with.sup.bounded}
  \sum_{n=0}^\infty\sup_{t\in[0,T]}e^{-t}\tfrac{t^n}{n!}m^n(i,j)<\infty.
  \end{equation}
  Thereby, for all $t\in[0,\infty)$ and all $i,j\in\D$ we can define
  \begin{equation}\label{eq:migration.semigroup}
  m_t(i,j):=\sum_{n=0}^\infty e^{-t}\tfrac{t^n}{n!}m^n(i,j).
  \end{equation}
  By \eqref{eq:m_t.with.sup.bounded} and using dominated convergence, we can compute
  for all $t\in[0,\infty)$ and all $i,j\in\D$ that
  \begin{equation} \begin{split} \label{eq:dmt}
  \tfrac{d}{dt}m_t(i,j)
  &=-m_t(i,j)+\sum_{n=1}^\infty e^{-t}\tfrac{t^{n-1}}{(n-1)!}m^n(i,j)
  =-m_t(i,j)+\sum_{n=0}^\infty e^{-t}\tfrac{t^{n}}{n!}m^{n+1}(i,j)
  \\
  &=-m_t(i,j)+\sum_{n=0}^\infty e^{-t}\tfrac{t^{n}}{n!}\sum_{k\in\D}m^n(i,k)m(k,j)
  =\sum_{k\in\D}m_t(i,k)(m(k,j)-\1_{j=k}).
  \end{split} \end{equation}
  Furthermore, note that for all $t\in[0,\infty)$ and all $i\in\D$ we have
  \begin{equation} \label{eq:m_s.sums.to.one}
  \sum_{j\in\D}m_t(i,j)
  =\sum_{j\in\D}\sum_{n=0}^\infty e^{-t}\tfrac{t^n}{n!}m^n(i,j)
  =\sum_{n=0}^\infty e^{-t}\tfrac{t^n}{n!}
  =1.
  \end{equation}
  For all $t\in[0,\infty)$, $s\in[0,t]$, $i\in\D$, $N\in\N$ define
  \begin{equation}
  Y_s^{N,t}(i):=\sum_{j\in\D}m_{(t-s)\kappa_H^N}(i,j)H_s^N(j).
  \end{equation}
  Observe that since for all $i,j\in\D$ it holds that $m_0(i,j)=\1_{i=j}$
  we have for all $t\in[0,\infty)$, all $i\in\D$, and all $N\in\N$ that
  \begin{equation} \label{eq:Yt=Ht}
  Y_t^{N,t}(i)=H_t^N(i).
  \end{equation}
  Furthermore, using \eqref{eq:bound.sumi.sigmai.mij} we have for all $t\in[0,\infty)$ and all $N\in\N$ that
  \begin{equation} \begin{split} \label{eq:bound.norm.Y0}
  &\sum_{i\in\D}\sigma_i\E\left[Y_0^{N,t}(i)\right]
  =
  \sum_{i\in\D}\sigma_i\E\left[\sum_{j\in\D}m_{t\kappa_H^N}(i,j)H_0^N(j)\right]
  =
  \sum_{i\in\D}\sigma_i\sum_{j\in\D}\sum_{n=0}^\infty e^{-t\kappa_H^N}\tfrac{(t\kappa_H^N)^n}{n!}m^n(i,j)\E\left[H_0^N(j)\right]
  \\
  &\quad
  =
  \sum_{j\in\D}\sum_{n=0}^\infty e^{-t\kappa_H^N}\tfrac{(t\kappa_H^N)^n}{n!}\E\left[H_0^N(j)\right]\sum_{i\in\D}\sigma_im^n(i,j)
  \leq
  \sum_{j\in\D}\sum_{n=0}^\infty e^{-t\kappa_H^N}\tfrac{(t\kappa_H^N)^n}{n!}\E\left[H_0^N(j)\right]c^n\sigma_j
  \\
  &\quad
  =
  \sum_{j\in\D}\sum_{n=0}^\infty e^{-t\kappa_H^N}\tfrac{(t\kappa_H^Nc)^n}{n!}\E\left[H_0^N(j)\right]\sigma_j
  =
  e^{t\kappa_H^N(c-1)}\E\left[\left\|H_0^N\right\|_{\sigma}\right].
  \end{split} \end{equation}
  For all $t\in[0,\infty)$, $s\in[0,t]$, $i,j\in\D$, $N\in\N$ we see from \eqref{eq:dmt} that we have 
  \begin{equation} \label{eq:dm(t-s)kappa}
  \tfrac{d}{ds}m_{(t-s)\kappa_H^N}(i,j)=-\kappa_H^N\sum_{k\in\D}m_{(t-s)\kappa_H^N}(i,k)(m(k,j)-\1_{j=k}).
  \end{equation}
  For $t\in[0,\infty)$, $N, l\in\N$, $i\in\D$ define
  \begin{equation}
  \tau_l^{N,t}(i):=\inf\left(\left\{u\in[0,t]:Y_u^{N,t}(i)>l\right\}\cup\infty\right).
  \end{equation}
  Using the fact that
  for all $t\in[0,\infty)$, all $u\in[0,t]$, all $N\in\N$, and all $i,j\in\D$ we have
  $m_{(t-u)\kappa_H^N}(i,j)\in[0,1]$
  we get for all $t\in[0,\infty)$, all $s\in[0,t]$, all $N, l\in\N$, and all $i\in\D$ that
  \begin{equation} \begin{split} \label{eq:quadr.var.Y.bounded}
  &\int_0^{s\land\tau_l^{N,t}(i)}\sum_{j\in\D}\left(m_{(t-u)\kappa_H^N}(i,j)\sqrt{\beta_H^NH_u^N(j)}\right)^2\,du
  \leq
  \int_0^{s\land\tau_l^{N,t}(i)}\sum_{j\in\D}m_{(t-u)\kappa_H^N}(i,j)\beta_H^NH_u^N(j)\,du
  \\
  &\qquad
  =
  \int_0^{s\land\tau_l^{N,t}(i)}\beta_H^NY_u^{N,t}(i)\,du
  \leq
  \int_0^s\beta_H^NY_{u\land\tau_l^{N,t}(i)}^{N,t}(i)\,du
  \leq
  t\beta_H^Nl.
  \end{split} \end{equation}
  For all $t\in[0,\infty)$, $s\in[0,t]$, $i\in\D$, $N\in\N$
  using It\^o's lemma with \eqref{eq:m_s.sums.to.one} and \eqref{eq:dm(t-s)kappa}
  we get $\P$-a.s.
  \begin{equation} \begin{split} \label{eq:bound.Y}
  &Y_s^{N,t}(i)-Y_0^{N,t}(i)
  =
  \int_0^s\sum_{j\in\D}m_{(t-u)\kappa_H^N}(i,j)
  \Bigg(\kappa_H^N\sum_{k\in\D}m(j,k)H_u^N(k)
  +\left(\lambda -\kappa_H^N-\alpha^NF_t^N(j)\right) H_u^N(j)
  \\
  &\quad
  -\tfrac{\lambda}{K}\left(H_u^N(j)\right)^2
  -\delta H_u^N(j)P_u^N(j)+\iota_H^N\Bigg)
  -\sum_{j\in\D}\kappa_H^N\sum_{k\in\D}m_{(t-u)\kappa_H^N}(i,k)(m(k,j)-\1_{j=k})H_u^N(j)\,du
  \\
  &\quad
  +\sum_{j\in\D}\int_0^sm_{(t-u)\kappa_H^N}(i,j)\sqrt{\beta_H^NH_u^N(j)}\,dW_u^{N,H}(j)
  \\
  &=
  \int_0^s\sum_{j\in\D}m_{(t-u)\kappa_H^N}(i,j)
  \Bigg(
  \left(\lambda -\alpha^NF_t^N(j)\right) H_u^N(j)
  -\tfrac{\lambda}{K}\left(H_u^N(j)\right)^2
  -\delta H_u^N(j)P_u^N(j)\Bigg)+\iota_H^N
  \,du
  \\
  &\quad
  +\sum_{j\in\D}\int_0^sm_{(t-u)\kappa_H^N}(i,j)\sqrt{\beta_H^NH_u^N(j)}\,dW_u^{N,H}(j)
  \\
  &\leq
  \int_0^s\sum_{j\in\D}m_{(t-u)\kappa_H^N}(i,j)
  \Bigg(
  \lambda H_u^N(j)
  -\tfrac{\lambda}{K}\left(H_u^N(j)\right)^2\Bigg)
  +\iota_H^N
  \,du
  \\
  &\quad
  +\sum_{j\in\D}\int_0^sm_{(t-u)\kappa_H^N}(i,j)\sqrt{\beta_H^NH_u^N(j)}\,dW_u^{N,H}(j).
  \end{split} \end{equation}
  Thus, using \eqref{eq:quadr.var.Y.bounded} and \eqref{eq:bound.Y} we get for all $t\in[0,\infty)$, all $s\in[0,t]$, all $i\in\D$, and all $N, l\in\N$ that
  \begin{equation} \begin{split}
  \E\left[Y_{s\land\tau_l^{N,t}(i)}^{N,t}(i)\right]
  -
  \E\left[Y_0^{N,t}(i)\right]
  &\leq
  \E\left[\int_0^{s\land\tau_l^{N,t}(i)}\sum_{j\in\D}m_{(t-u)\kappa_H^N}(i,j)
  \lambda H_u^N(j)
  +\iota_H^N
  \,du\right]
  \\
  &\leq
  \E\left[\int_0^s\sum_{j\in\D}m_{(t-{u\land\tau_l^{N,t}(i)})\kappa_H^N}(i,j)
  \lambda H_{u\land\tau_l^{N,t}(i)}^N(j)
  +\iota_H^N
  \,du\right]
  \\
  &=\int_0^s\lambda\E\left[Y_{u\land\tau_l^{N,t}(i)}^{N,t}(i)\right]+\iota_H^N\,du
  \leq
  t\iota_H^N+\lambda\int_0^s\E\left[Y_{u\land\tau_l^{N,t}(i)}^{N,t}(i)\right]\,du.
  \end{split} \end{equation}
Now, using Gronwall's lemma
\citep[e.g., ][]{Klenke2008},
we get for all $t\in[0,\infty)$, all $s\in[0,t]$, all $i\in\D$, and all $N, l\in\N$ that
  \begin{equation} \begin{split} \label{eq:bound.stopped.Y}
  \E\left[Y_{s\land\tau_l^{N,t}(i)}^{N,t}(i)\right]
  &
  \leq
  \left(\E\left[Y_0^{N,t}(i)\right]+t\iota_H^N\right)e^{\lambda s}
  \leq
  \left(\E\left[Y_0^{N,t}(i)\right]+t\iota_H^N\right)e^{\lambda t}.
  \end{split} \end{equation}
  For all $t\in[0,\infty)$, $N\in\N$, $i\in\D$ the $\P$-a.s.~continuous paths of $\left(Y_u^{N,t}(i)\right)_{u\in[0,t]}$ imply
  $\P\Big[\sup_{u\in[0,t]}Y_u^{N,t}(i)<\infty\Big]=1$.
  Hence, we get for all $t\in[0,\infty)$, all $N\in\N$, and all $i\in\D$ that
  \begin{equation} \label{eq:tauY.to.infty}
  \P\left[\lim_{l\to\infty}\tau_l^{N,t}(i)=\infty\right]=1.
  \end{equation}
  Using the assumption that for all $N\in\N$ we have $\sum_{i\in\D}\sigma_i\E\Big[H_0^N(i)\Big]<\infty$
  together with
  \eqref{eq:Yt=Ht}, \eqref{eq:bound.norm.Y0}, \eqref{eq:bound.stopped.Y}, and \eqref{eq:tauY.to.infty}
  with Fatou's lemma we obtain for all $t\in[0,\infty)$ and all $N\in\N$ that
  \begin{equation} \begin{split} \label{eq:boundY.fix.t}
  \sum_{i\in\D}\sigma_i\E\left[H_t^N(i)\right]
  &=
  \sum_{i\in\D}\sigma_i\E\left[Y_t^{N,t}(i)\right]
  =
  \sum_{i\in\D}\sigma_i\E\left[\lim\limits_{l\to\infty}Y_{t\land\tau_l^{N,t}(i)}^{N,t}(i)\right]
  \leq
  \sum_{i\in\D}\sigma_i\liminf\limits_{l\to\infty}\E\left[Y_{t\land\tau_l^{N,t}(i)}^{N,t}(i)\right]
  \\
  &\leq
  \sum_{i\in\D}\sigma_i\liminf\limits_{l\to\infty}\left(\E\left[Y_0^{N,t}(i)\right]+t\iota_H^N\right)e^{\lambda t}
  =
  \sum_{i\in\D}\sigma_i\left(\E\left[Y_0^{N,t}(i)\right]+t\iota_H^N\right)e^{\lambda t}
  \\
  &\leq
  \left(e^{t\kappa_H^N(c-1)}\E\left[\left\|H_0^N\right\|_{\sigma}\right]+\sum_{i\in\D}\sigma_it\iota_H^N\right)e^{\lambda t}
  <
  \infty.
  \end{split} \end{equation}
  Using the fact that
  for all $t\in[0,\infty)$, all $u\in[0,t]$, all $N\in\N$, and all $i,j\in\D$ we have
  $m_{(t-u)\kappa_H^N}(i,j)\in[0,1]$
  this implies for all $t\in[0,\infty)$, all $s\in[0,t]$, all $N\in\N$, and all $i\in\D$, that
  \begin{equation} \begin{split}
  &\E\left[\sum_{j\in\D}\int_0^{s}\left(m_{(t-u)\kappa_H^N}(i,j)\sqrt{\beta_H^NH_u^N(j)}\right)^2\,du\right]
  =
  \int_0^{s}\E\left[\sum_{j\in\D}\left(m_{(t-u)\kappa_H^N}(i,j)\sqrt{\beta_H^NH_u^N(j)}\right)^2\right]\,du
  \\
  &\quad
  \leq
  \beta_H^N\int_0^{s}\E\left[\sum_{j\in\D}m_{(t-u)\kappa_H^N}(i,j)H_u^N(j)\right]\,du
  =
  \beta_H^N\int_0^{s}\E\left[Y_u^{N,t}(i)\right]\,du
  <\infty.
  \end{split} \end{equation}
  Thus, taking expectations in \eqref{eq:bound.Y} gives for all $t\in[0,\infty)$, all $s\in[0,t]$, all $i\in\D$, and all $N\in\N$ using Jensen's inequality
  \begin{equation} \begin{split}
  &\E\left[Y_s^{N,t}(i)\right]-\E\left[Y_0^{N,t}(i)\right]
  \leq
  \int_0^s
  \Bigg(
  \lambda \E\left[Y_u^{N,t}(i)\right]
  -\tfrac{\lambda}{K}\E\left[\sum_{j\in\D}m_{(t-u)\kappa_H^N}(i,j)\left(H_u^N(j)\right)^2\right]
  +\iota_H^N\Bigg)
  \,du
  \\
  &\quad
  \leq
  \int_0^s
  \Bigg(
  \lambda \E\left[Y_u^{N,t}(i)\right]
  -\tfrac{\lambda}{K}\E\left[\left(Y_u^{N,t}(i)\right)^2\right]
  +\iota_H^N\Bigg)
  \,du
  \\
  &\quad
  \leq
  \int_0^s
  \Bigg(
  \lambda \E\left[Y_u^{N,t}(i)\right]
  -\tfrac{\lambda}{K}\left(\E\left[Y_u^{N,t}(i)\right]\right)^2
  +\bar{\iota}_H\Bigg)
  \,du.
  \end{split} \end{equation}
  For $t\in[0,\infty)$, $i\in\D$, $N\in\N$
  let $z^{N,t}(i)\colon [0,\infty)\to\R$ be a process that for all $s\in[0,\infty)$ satisfies
  \begin{equation}
  z_s^{N,t}(i)=
  z_0^{N,t}(i)
  +\int_0^s
  \left(
  \lambda z_u^{N,t}(i)
  -\tfrac{\lambda}{K}\left(z_u^{N,t}(i)\right)^2
  +\bar{\iota}_H
  \right)
  \,du
  \end{equation}
  with
  $z_0^{N,t}(i)=\E\left[Y_0^{N,t}(i)\right]$
  where uniqueness follows from local Lipschitz continuity.
  Define $c_1:=\tfrac{K}{2}+\sqrt{\tfrac{K^2}{4}+\tfrac{K\bar{\iota}_H}{\lambda}}\in(0,\infty)$.
  Using classical comparison results from the theory of ODEs,
  the above computation shows that for all $N\in\N$, all $i\in\D$, and all $t\in[0,\infty)$ we have
  \begin{equation} \begin{split}
  \E\left[H_t^N(i)\right]
  &=
  \E\left[Y_t^{N,t}(i)\right]
  \leq
  z_t^{N,t}(i)
  \leq
  \max\left\{\E\left[Y_0^{N,t}(i)\right],\limsup\limits_{s\to\infty}z_s^{N,t}(i)\right\}
  =
  \max\left\{\E\left[Y_0^{N,t}(i)\right],c_1\right\}
  \\
  &\leq
  \E\left[Y_0^{N,t}(i)\right]+c_1
  =
  \E\left[\sum_{j\in\D}m_{t\kappa_H^N}(i,j)H_0^N(j)\right]+c_1
  \end{split} \end{equation}
  This finishes the proof of Lemma \ref{lem:unif.pos.moments.H}.
\end{proof}
  
\begin{proof}[Proof of Theorem \ref{thm:conv.freq.altruist}.]
The main steps of the proof are to show (see \eqref{eq:epsNsmall} below)
that the generator of
$F^N$ is close to $\mathcal{A}_1$ (defined in \eqref{eq:A1} below) as $N\to\infty$
and to prove with Theorem \ref{thm:sup.N.int.(H-h)^2.bounded} that
the host population is asymptotically immediately in equilibrium,
that is, $H^N$ is close to $h_{\infty}(F^N)$ for large $N\in\N$.
We will use stochastic averaging
\citep[see Theorem 2.1 in ][]{Kurtz1992}
to prove the result.
So we first check that all conditions of the aforementioned theorem are fulfilled.
Define $\beta_H:=\frac{\beta}{b}=\lim_{N\to\infty}N\beta_H^n$.
Note that $E_1=[0,1]^\D$ and $E_2=l_\sigma^1\cap[0,\infty)^\D$ are complete separable metric spaces.
Tychonoff's theorem implies that $E_1$ is compact.
Since for all $N\in\N$ and all $t\in[0,\infty)$ the random variable
$F_{tN}^N$ takes values in the compact space $E_1$,
the compact containment condition holds for
$\big\{\left(F_{tN}^N\right)_{t\in[0,\infty)}:N\in\N\big\}$.
We will now use Lemma \ref{lem:family.rv.rel.compactness} to show for each $T\in[0,\infty)$ that
the family $\left\{H_{tN}^N : t\in[0,T],N\in\N\right\}$ is relatively compact in $E_2$.
From Lemma \ref{lem:bound.norm.(H+P)^p}
and the assumption $\sup_{N\in\N}\E\Big[\Big\|\Big(H_0^N+P_0^N\Big)^4\Big]<\infty$ we see that
\begin{equation}
  \sup_{N\in\N}\sup_{t\in[0,\infty)}\E\left[\left\|H_{tN}^N\right\|_{\sigma}\right]<\infty.
\end{equation}
Define $\D_0:=\emptyset$ and
for all $n\in\N$
let $\D_n\subseteq \D$ be a set with $|\D_n|=\min\{n,|\D|\}$
and $\D_n\supseteq \D_{n-1}$.
Assume that $\D=\cup_{n\in\N}\D_n$.
Define $c_1:=\tfrac{K}{2}\Big(1+\sqrt{1+\tfrac{4\bar{\iota}_H}{K\lambda}}\Big)$.
From Lemma \ref{lem:unif.pos.moments.H} with the assumption that
$\sum_{i\in\D}\sup_{N\in\N}\sigma_i\E\Big[H_0^N(i)\Big]<\infty$
we get for all $T\in[0,\infty)$ that
\begin{equation} \begin{split}
  &\sum_{i\in\D}\sigma_i\sup\limits_{N\in\N}\sup\limits_{t\in[0,T]}\E\left[H_{tN}^N(i)\right]
  \leq
  \sum_{i\in\D}\sigma_i\sup\limits_{N\in\N}\sup\limits_{t\in[0,T]}
  \left(
  \sum_{j\in\D}\sum_{n=0}^\infty e^{-tN\kappa_H^N}\tfrac{(tN\kappa_H^N)^n}{n!}m^n(i,j)\E\left[H_0^N(j)\right]+c_1
  \right)
  \\
  &
  \leq
  \sum_{j\in\D} \sum_{n=0}^\infty\left(\sum_{i\in\D}\sigma_im^n(i,j)\right)\sup\limits_{N\in\N}\sup\limits_{t\in[0,TN\kappa_H^N]}
  e^{-t}\tfrac{t^n}{n!}\E\left[H_0^N(j)\right]+c_1\sum_{i\in\D}\sigma_i
  \\
  &
  \leq
  \sum_{j\in\D}\sum_{n=0}^\infty c^n\sigma_j\sup\limits_{N\in\N}
  \tfrac{(TN\kappa_H^N)^n}{n!}\E\left[H_0^N(j)\right]+c_1\|\underline{1}\|_\sigma
  \leq
  e^{cT\sup\limits_{M\in\N}M\kappa_H^M}\sum_{j\in\D}\sigma_j\sup\limits_{N\in\N}
  \E\left[H_0^N(j)\right]+c_1\|\underline{1}\|_\sigma
  <
  \infty.
\end{split} \end{equation}
Now we can use the dominated convergence theorem to obtain for all $T\in[0,\infty)$ that
\begin{equation} \begin{split}
  &\lim_{n\to\infty}\sup_{N\in\N}\sup_{t\in[0,T]}\sum_{k\in\D\setminus \D_n}\sigma_k\E\Big[H_{tN}^N(k)\Big]
  \leq
  \lim_{n\to\infty}\sum_{k\in\D\setminus \D_n}\sup_{N\in\N}\sup_{t\in[0,T]}\sigma_k\E\Big[H_{tN}^N(k)\Big]
  =0.
\end{split} \end{equation}
Hence, for all $T\in[0,\infty)$
we can apply Lemma \ref{lem:family.rv.rel.compactness} to the family
$\left\{H_{tN}^N : t\in[0,T],N\in\N\right\}$
and conclude that it is relatively compact in $E_2$.
Denote by $C_b(E_1,\mathbb{R})$ the set of bounded, continuous real-valued functions on $E_1$ and by
$C_b^2(E_1,\mathbb{R})$ the set of all real-valued functions on $E_1$
that are twice continuously differentiable and bounded, with bounded first and second order partial derivatives.
For $f\in C_b^2(E_1,\mathbb{R})$ let $c_f\in (0,\infty)$ be such that for all $x\in E_1$ and all $i\in\D$
we have $\Big|\tfrac{df}{dx_i}(x)\Big|+\Big|\tfrac{d^2f}{dx_i^2}(x)\Big|\leq c_f$.
Define
\begin{equation}
  \operatorname{Dom}(\mathcal{A})
  :=\left\{f\in C_b^2(E_1,\mathbb{R}):f \text{ depends only on finitely many coordinates}\right\}
\end{equation}
and for any $f\in\operatorname{Dom}(\mathcal{A})$ denote by $\D_f$ the finite set of coordinates that $f$ depends on.
Due to the Stone--Weierstrass theorem \citep[see e.g.][Theorem 15.2]{Klenke2008} we see that
$\operatorname{Dom}(\mathcal{A})$ is dense in $C_b(E_1,\mathbb{R})$ in the topology of uniform convergence.
Denote by $C(E_1\times E_2,\mathbb{R})$ the set of real-valued continuous functions on $E_1\times E_2$
and define the operator $\mathcal{A}_1:\operatorname{Dom}(\mathcal{A})\to C(E_1\times E_2,\mathbb{R})$
for all $f\in\operatorname{Dom}(\mathcal{A})$, all $x\in E_1$, and all $y\in E_2$ by
\begin{equation} \begin{split}\label{eq:A1}
  \left(\mathcal{A}_1f\right)(x,y):=
  &
  \sum\limits_{i\in\D}
  \1_{y_i>0}
  \Bigg(
  \Big[\kappa_H\sum\limits_{j\in\D}\left(m(i,j)
  \tfrac{y_j}{y_i}(x_j-x_i)\right)
  -\alpha x_i(1-x_i)\Big]
  \tfrac{df}{dx_i}(x)
  +\tfrac{1}{2}
  \beta_H \tfrac{x_i(1-x_i)}{y_i}
  \tfrac{d^2f}{dx_i^2}(x)
  \Bigg).
\end{split} \end{equation}
For all $f\in\operatorname{Dom}(\mathcal{A})$, all $N\in\N$, and all $t\in[0,\infty)$ define
\begin{equation} \begin{split}
  \eps_f^N(t)
  :=&
  \tfrac{1}{N}
  \int_0^t\left(\mathcal{A}_1f\right)\left(F_u^N,H_u^N\right)\,du
  -
  \sum_{i\in\D}
  \int_0^{t}
  \tfrac{df}{dx_i}(F_u^N)
  \Bigg[
  \kappa_H^N
  \sum\limits_{j\in\D}m(i,j)\left(F_u^N(j)-F_u^N(i)\right)\tfrac{H_u^N(j)}{H_u^N(i)}
  \\
  &\qquad
  -\alpha^NF_u^N(i)\left(1-F_u^N(i)\right)
  \Bigg]
  +
  \tfrac{1}{2}\tfrac{d^2f}{dx_i^2}(F_u^N)
  \tfrac{\beta_H^N F_u^N(i)\left(1-F_u^N(i)\right)}{H_u^N(i)}
  \,du.
\end{split} \end{equation}
We now show that this process vanishes on the evolutionary time scale
as $N\to\infty$.
From It\^o's lemma and Lemma \ref{lem:estimate.term.sum^p} we get for all $f\in\operatorname{Dom}(\mathcal{A})$, all $N\in\N$, and all $t\in[0,\infty)$ that $\P$-a.s.
\begin{equation} \begin{split}
  &f\left(F_t^N\right)
  -
  f\left(F_0^N\right)
  =
  \sum_{i\in\D}
  \int_0^t
  \tfrac{df}{dx_i}(F_u^N)
  \,dF_u^N(i)
  +
  \tfrac{1}{2}\sum_{i,j\in\D}
  \int_0^t
  \Big(\tfrac{d^2f}{dx_idx_j}(F_u^N)\Big)
  d\left<F^N(i),F^N(j)\right>_u
  \\
  &\quad
  =
  \sum_{i\in\D}
  \int_0^t
  \tfrac{df}{dx_i}(F_u^N)
  \Bigg[
  \kappa_H^N
  \sum\limits_{j\in\D}m(i,j)\left(F_u^N(j)-F_u^N(i)\right)\tfrac{H_u^N(j)}{H_u^N(i)}
  -\alpha^NF_u^N(i)\left(1-F_u^N(i)\right)
  \Bigg]
  \\
  &\qquad
  +
  \tfrac{1}{2}\tfrac{d^2f}{dx_i^2}(F_u^N)
  \tfrac{\beta_H^N F_u^N(i)\left(1-F_u^N(i)\right)}{H_u^N(i)}
  \,du
  +
  \sum_{i\in\D}
  \int_0^t
  \tfrac{df}{dx_i}(F_u^N)
  \sqrt{\tfrac{\beta_H^N F_u^N(i)\left(1-F_u^N(i)\right)}{H_u^N(i)}}
  \,dW_u^{F,N}(i).
\end{split} \end{equation}
Hence, we get for all $f\in\operatorname{Dom}(\mathcal{A})$, all $N\in\N$, and all $t\in[0,\infty)$ that $\P$-a.s.
\begin{equation} \begin{split} \label{eq:martingale.term.f}
  &f\left(F_{tN}^N\right)
  -\int_0^t\left(\mathcal{A}_1f\right)\left(F_{uN}^N,H_{uN}^N\right)\,du
  +\eps_f^N(tN)
  \\
  &\quad
  =
  f\left(F_0^N\right)
  +
  \sum_{i\in\D}
  \int_0^{tN}
  \tfrac{df}{dx_i}(F_{u}^N)
  \sqrt{\tfrac{\beta_H^N F_{u}^N(i)\left(1-F_{u}^N(i)\right)}{H_{u}^N(i)}}
  \,dW_{u}^{F,N}(i).
\end{split} \end{equation}
From Tonelli's theorem and Lemma \ref{lem:bound.norm.1/H^2} we obtain for all $f\in\operatorname{Dom}(\mathcal{A})$, all $N\in\N$, and all $t\in[0,\infty)$ that
\begin{equation} \begin{split}
  &\E\Bigg[\int_0^{tN}
  \Bigg(
  \sum_{i\in\D}
  \tfrac{df}{dx_i}(F_u^N)
  \sqrt{\tfrac{\beta_H^N F_u^N(i)\left(1-F_u^N(i)\right)}{H_u^N(i)}}
  \Bigg)^2
  \,du\Bigg]
  \leq
  tN|\D_f|c_f^2\beta_H^N\max_{i\in\D_f}\sup_{M\in\N}\sup_{u\in[0,\infty)}\E\Big[\tfrac{1}{H_u^M(i)}\Big]
  \\
  &\quad
  \leq
  t|\D_f|c_f^2N\beta_H^N\max_{i\in\D_f}\tfrac{1}{\sigma_i}\sup_{M\in\N}\sup_{u\in[0,\infty)}\E\Big[\Big\|\tfrac{1}{H_u^M}\Big\|_\sigma\Big]
  <\infty.
\end{split} \end{equation} 
Thus for all $f\in\operatorname{Dom}(\mathcal{A})$, all $N\in\N$, and all $t\in[0,\infty)$ the left-hand side of \eqref{eq:martingale.term.f} is a martingale.
Next, for all $f\in\operatorname{Dom}(\mathcal{A})$ and all $T\in[0,\infty)$ it holds that
\begin{equation} \begin{split}
  &\sup_{N\in\N}\E\left[\int_0^T\left|\left(\mathcal{A}_1f\right)\left(F_{tN}^N,H_{tN}^N\right)\right|^\frac{4}{3}\,dt\right]
  \\
  &\quad
  =
  \sup_{N\in\N}\E\Bigg[\int_0^T\Bigg|
  \sum\limits_{i\in\D_f}\Bigg(\kappa_H\sum_{j\in\D}\left(m(i,j)
  \tfrac{H_{tN}^N(j)}{H_{tN}^N(i)}\left(F_{tN}^N(j)-F_{tN}^N(i)\right)\right)
  \\
  &\qquad\qquad\qquad
  -\alpha F_{tN}^N(i)\left(1-F_{tN}^N(i)\right)\Bigg)
  \tfrac{df}{dx_i}\left(F_{tN}^N\right)
  +\tfrac{1}{2}
  \sum_{i\in\D_f}\beta_H \tfrac{F_{tN}^N(i)\left(1-F_{tN}^N(i)\right)}{H_{tN}^N(i)}
  \tfrac{d^2f}{dx_i^2}\left(F_{tN}^N\right)
  \Bigg|^{\frac{4}{3}}\,dt\Bigg]
  \\
  &\quad
  \leq
  \sup_{N\in\N}\E\Bigg[
  \int_0^T
  \Bigg(
  \sum\limits_{i\in\D_f}
  \Bigg(
  \Big|\kappa_H\sum_{j\in\D}\left(m(i,j)
  \tfrac{H_{tN}^N(j)}{H_{tN}^N(i)}\right)
  c_f
  \Big|
  +\Big|
  \alpha 
  c_f
  \Big|
  +\Big|
  \tfrac{1}{2}
  \beta_H \tfrac{1}{H_{tN}^N(i)}
  c_f\Big|
  \Bigg)
  \Bigg)^{\frac{4}{3}}
  \,dt\Bigg].
\end{split} \end{equation}
Using Young's inequality and Jensen's inequality we get for all $f\in\operatorname{Dom}(\mathcal{A})$ and all $T\in[0,\infty)$ that
\begin{equation} \begin{split}
  &\sup_{N\in\N}\E\left[\int_0^T\left|\left(\mathcal{A}_1f\right)\left(F_{tN}^N,H_{tN}^N\right)\right|^\frac{4}{3}\,dt\right]
  \\
  &\quad
  \leq
  \sup_{N\in\N}\E\Bigg[
  \int_0^T
  \Bigg(
  \sum\limits_{i\in\D_f}
  \Bigg(
  \tfrac{2}{3}
  \Big(\kappa_Hc_f\tfrac{1}{H_{tN}^N(i)}\Big)^{\frac{3}{2}}
  +
  \tfrac{1}{3}
  \Big(\sum_{j\in\D}m(i,j)H_{tN}^N(j)\Big)^3
  +\alpha c_f
  +
  \tfrac{1}{2}
  \beta_H \tfrac{1}{H_{tN}^N(i)}
  c_f
  \Bigg)
  \Bigg)^{\frac{4}{3}}
  \,dt\Bigg]
  \\
  &\quad
  \leq
  \sup_{N\in\N}\E\Bigg[
  \int_0^T
  \tfrac{(4\left|\D_f\right|)^{\frac{1}{3}}}{\min_{k\in\D_f}\{\sigma_k\}}
  \sum\limits_{i\in\D_f}
  \sigma_i
  \Bigg(
  \Big(\tfrac{2}{3}\Big)^{\frac{4}{3}}
  \Big(\kappa_Hc_f\tfrac{1}{H_{tN}^N(i)}\Big)^2
  +
  \Big(\tfrac{1}{3}\Big)^{\frac{4}{3}}
  \Big(\sum_{j\in\D}m(i,j)H_{tN}^N(j)\Big)^4
  \\
  &\qquad\qquad\qquad
  +
  \Big(\alpha c_f\Big)^\frac{4}{3}
  +
  \Big(\tfrac{1}{2}\beta_H \tfrac{1}{H_{tN}^N(i)}c_f\Big)^{\frac{4}{3}}\Bigg)
  \,dt\Bigg].
\end{split} \end{equation}
Using Lemma \ref{lem:estimate.term.sum^p}, Tonelli's theorem, and Lemmas
\ref{lem:bound.norm.(H+P)^p} and \ref{lem:bound.norm.1/H^2} we obtain
for all $f\in\operatorname{Dom}(\mathcal{A})$ and all $T\in[0,\infty)$ that
\begin{equation} \begin{split}
  &\sup_{N\in\N}\E\left[\int_0^T\left|\left(\mathcal{A}_1f\right)\left(F_{tN}^N,H_{tN}^N\right)\right|^\frac{4}{3}\,dt\right]
  \\
  &\quad
  \leq
  \sup_{N\in\N}
  \tfrac{(4\left|\D_f\right|)^{\frac{2}{3}}}{\min_{k\in\D_f}\{\sigma_k\}}
  \int_0^T
  \Big(\tfrac{4}{9}\Big)^{\frac{2}{3}}
  \left(\kappa_Hc_f\right)^2
  \E\Bigg[\left\|\tfrac{1}{\left(H_{tN}^N\right)^2}\right\|_\sigma\Bigg]
  +
  \Big(\tfrac{1}{9}\Big)^{\frac{2}{3}}c
  \E\Bigg[\left\|\Big(H_{tN}^N\Big)^4\right\|_\sigma
  +
  \left(\alpha c_f\right)^\frac{4}{3}\|\underline{1}\|_\sigma\Bigg]
  \\
  &\qquad\qquad\qquad
  +
  \Big(\tfrac{1}{4}\Big)^{\frac{2}{3}}
  \left(\beta_H c_f\right)^\frac{4}{3}
  \E\Bigg[\left\|\Big(\tfrac{1}{H_{tN}^N}\Big)^{\frac{4}{3}}\right\|_\sigma\Bigg]
  \,dt
  <\infty.
\end{split} \end{equation}
Furthermore, for all $f\in\operatorname{Dom}(\mathcal{A})$, all $N\in\N$, and all $T\in[0,\infty)$ we have that
\begin{equation} \begin{split}
  &\E\left[\sup_{t\in[0,T]}\Big|\eps_f^N(tN)\Big|\right]
  \\
  &\quad
  =
  \E\Bigg[\sup_{t\in[0,T]}\Bigg|
  \sum_{i\in\D_f}
  \int_0^{t}
  \tfrac{df}{dx_i}(F_{uN}^N)
  \Big[
  \left(\kappa_H-N\kappa_H^N\right)
  \sum\limits_{j\in\D}m(i,j)\left(F_{uN}^N(j)-F_{uN}^N(i)\right)\tfrac{H_{uN}^N(j)}{H_{uN}^N(i)}
  \\
  &\qquad
  -
  \left(\alpha-N\alpha^N\right)F_{uN}^N(i)\left(1-F_{uN}^N(i)\right)
  \Big]
  +
  \tfrac{1}{2}\tfrac{d^2f}{dx_i^2}(F_{uN}^N)
  \left(\beta_H-N\beta_H^N\right) \tfrac{F_{uN}^N(i)\left(1-F_{uN}^N(i)\right)}{H_{uN}^N(i)}
  \,du
  \Bigg|\Bigg]
  \\
  &\quad
  \leq
  \E\Bigg[
  \int_0^T
  \sum_{i\in\D_f}
  c_f
  \Big(
  \left|\kappa_H-N\kappa_H^N\right|
  \sum\limits_{j\in\D}m(i,j)\tfrac{H_{uN}^N(j)}{H_{uN}^N(i)}
  +\left|\alpha-N\alpha^N\right|
  +\tfrac{1}{2}\left|\beta_H-N\beta_H^N\right|\tfrac{1}{H_{uN}^N(i)}
  \Big)
  \,du\Bigg].
\end{split} \end{equation}
Using Young's inequality, Lemma \ref{lem:estimate.term.sum^p}, and Tonelli's theorem we
get for all $f\in\operatorname{Dom}(\mathcal{A})$, all $N\in\N$, and all $T\in[0,\infty)$ that
\begin{equation} \begin{split}
  \E\left[\sup_{t\in[0,T]}\Big|\eps_f^N(tN)\Big|\right]
  &\leq
  \E\Bigg[
  \int_0^T
  \sum_{i\in\D_f}
  \tfrac{\sigma_i c_f}{\min_{k\in\D_f}\{\sigma_k\}}
  \Bigg(
  \tfrac{\left|\kappa_H-N\kappa_H^N\right|}{2}
  \Bigg(
  \Big(\tfrac{1}{H_{uN}^N(i)}\Big)^2
  +\Big(\sum\limits_{j\in\D}m(i,j)H_{uN}^N(j)\Big)^2
  \Bigg)
  \\
  &\qquad\qquad
  +\left|\alpha-N\alpha^N\right|
  +\tfrac{\left|\beta_H-N\beta_H^N\right|}{2}\tfrac{1}{H_{uN}^N(i)}
  \Bigg)
  \,du\Bigg]
  \\
  &
  \leq
  \tfrac{c_f}{\min_{k\in\D_f}\{\sigma_k\}}
  \int_0^T
  \tfrac{\left|\kappa_H-N\kappa_H^N\right|}{2}
  \Bigg(
  \E\left[\left\|\tfrac{1}{\left(H_{uN}^N\right)^2}\right\|_\sigma\right]
  +c\E\left[\left\|\left(H_{uN}^N\right)^2\right\|_\sigma\right]
  \Bigg)
  \\
  &\qquad\qquad
  +\left|\alpha-N\alpha^N\right|\|\underline{1}\|_\sigma
  +\tfrac{\left|\beta_H-N\beta_H^N\right|}{2}\E\left[\left\|\tfrac{1}{H_{uN}^N}\right\|_\sigma\right]
  \,du.
\end{split} \end{equation}
Hence, from Lemmas \ref{lem:bound.norm.(H+P)^p} and \ref{lem:bound.norm.1/H^2} we see
for all $f\in\operatorname{Dom}(\mathcal{A})$ and all $T\in[0,\infty)$ that
\begin{equation} \begin{split}\label{eq:epsNsmall}
  0
  &\leq
  \lim_{N\to\infty}\E\Bigg[\sup_{t\in[0,T]}\Big|\eps_f^N(tN)\Big|\Bigg]
  \\
  &\leq
  \lim_{N\to\infty}
  \tfrac{Tc_f}{\min_{k\in\D_f}\{\sigma_k\}}
  \Bigg(
  \tfrac{\left|\kappa_H-N\kappa_H^N\right|}{2}
  \sup_{M\in\N}
  \sup_{t\in[0,\infty)}
  \Big(
  \E\left[\left\|\tfrac{1}{\left(H_t^M\right)^2}\right\|_\sigma\right]
  +c\E\left[\left\|\left(H_t^M\right)^2\right\|_\sigma\right]
  \Big)
  +\left|\alpha-N\alpha^N\right|\|\underline{1}\|_\sigma
  \\
  &\qquad\qquad
  +\tfrac{\left|\beta_H-N\beta_H^N\right|}{2}
  \sup_{M\in\N}\sup_{t\in[0,\infty)}\E\left[\left\|\tfrac{1}{H_t^M}\right\|_\sigma\right]
  \Bigg)
  =
  0.
\end{split} \end{equation}
Next we show that the host population is asymptotically immediately in the equilibrium state.
Define the set
$\mathcal{R}:=\Big\{\bigtimes\limits_{i\in\D}B_i\colon (B_i)_{i\in\D}\subseteq\mathcal{B}([0,\infty)^\D),
B_i=[0,\infty) \text{ for all but finitely many } i\in\D\Big\}$.
For all $N\in\N$, all $t\in[0,\infty)$, and all $B\in\mathcal{R}$ define the measure-valued random variables
\begin{equation} \begin{split}
\Lambda^N([0,t]\times B)
:=&
\int_0^t\1_B\left(H_{uN}^N\right)\,du
=\int_0^t\prod\limits_{i\in\D}\1_{B_i}\left(H_{uN}^N(i)\right)\,du,
\end{split} \end{equation}
Due to Carath\'eodory's theorem \citep[see e.g.][Theorem 1.41]{Klenke2008} there is a unique extension of this pre-measure
to a measure on $[0,\infty)\times E_2$,
which we will denote by the same name.
Define the space
$\ell(E_2):=\{\mu \colon \mu \text{ is a measure on } [0,\infty)\times E_2 \text{ such that for all } t\in[0,\infty) \text{ it holds that } \mu([0,t]\times E_2)=t\}$
and the space $D([0,\infty)):=\{f\colon [0,\infty) \to E_1 | f \text{ is \cadlag}\}$.
Having checked all assumptions, we can now apply Theorem 2.1 from \cite{Kurtz1992}
and conclude that the sequence $\big\{\big(\big(F_{tN}^N\big)_{t\in[0,\infty)},\Lambda^N\big):N\in\N\big\}$ is relatively compact in
$D([0,\infty))\times \ell(E_2)$.
Let $(F, \Lambda)$ be a $D([0,\infty))\times \ell(E_2)$-valued random variable and let $\left(N_k\right)_{k\in\N}\subseteq\N$ be an increasing sequence such that
w-$\lim_{k\to\infty}\big(\big(F_{tN_k}^{N_k}\big)_{t\in[0,\infty)},\Lambda^{N_k}\big)=(F, \Lambda)$.
Due to Skorohod's representation theorem \citep[Theorem 3.1.8 of ][]{EthierKurtz1986} we can assume without loss of generality and for ease of notation that $(F,\Lambda)$
acts on the probability space $\left(\Omega,\mathcal{F},\P\right)$.
Using H\"older's inequality and Theorem \ref{thm:sup.N.int.(H-h)^2.bounded}
we see for all $t\in[0,\infty)$ that
\begin{equation} \begin{split} \label{eq:lim.int.H-h.0}
  0
  \leq&
  \lim_{N\to\infty}\int_0^t\E\left[\left\|H_{uN}^N-\left(h_\infty\left(F_{uN}^N(i)\right)\right)_{i\in\D}\right\|_{\sigma}\right]\,du
  =
  \lim_{N\to\infty}\int_0^t\E\left[\sum_{i\in\D}\sigma_i\left|H_{uN}^N(i)-h_\infty\left(F_{uN}^N(i)\right)\right|\right]\,du
  \\
  \leq&
  \lim_{N\to\infty}
  \sqrt{\int_0^t\E\left[\sum_{i\in\D}\sigma_i\left(H_{uN}^N(i)-h_\infty\left(F_{uN}^N(i)\right)\right)^2\right]\,du}
  \sqrt{t\sum_{k\in\D}\sigma_k}
  =0.
\end{split} \end{equation}
For any bounded Lipschitz continuous function $f\colon l_\sigma^1\to\R$, with Lipschitz constant $\bar{c}_f$,
and all $t\in[0,\infty)$, applying \eqref{eq:lim.int.H-h.0}, we then have
\begin{equation} \begin{split}
  0
  &\leq
  \E\Big[\Big|\int_0^t \int_{E_2} f(y)\,\Lambda(du\times dy)-\int_0^tf(h_\infty(F_u))\,du\Big|\Big]
  \\
  &=
  \lim_{k\to\infty}
  \E\Big[\Big|\int_0^t f\big(H_{uN_k}^{N_k}\big)\,du-\int_0^tf\big(h_\infty\big(F_{uN_k}^{N_k}\big)\big)\,du\Big|\Big]
  \leq
  \bar{c}_f \lim_{k\to\infty}
  \E\Big[\int_0^t \big\|H_{uN_k}^{N_k}-h_\infty\big(F_{uN_k}^{N_k}\big)\big\|_{\sigma}\,du\Big]
  =
  0.
\end{split} \end{equation}
This implies that $\Lambda(du\times dy)=\delta_{h_{\infty}(F_u)}(dy)\,du$
where $\delta_x$ is the Dirac measure on $x$.
Define the operator $\mathcal{A}_2:\operatorname{Dom}(\mathcal{A})\to C(E_1,\mathbb{R})$
for all $f\in\operatorname{Dom}(\mathcal{A})$ and all $x\in E_1$ by
\begin{equation} \begin{split}
  \left(\mathcal{A}_2 f\right)(x):=
  &
  \sum\limits_{i\in\D}\left(\kappa_H\sum\limits_{j\in\D}\left(m(i,j)
  \tfrac{a-x_i}{a-x_j}(x_j-x_i)\right)
  -\alpha x_i(1-x_i)\right)
  \tfrac{df}{dx_i}(x)
  \\
  &\quad
  +\tfrac{1}{2}
  \sum\limits_{i\in\D}\beta_H b(a-x_i)x_i(1-x_i)
  \tfrac{d^2f}{dx_i^2}(x).
\end{split} \end{equation}
Note
for all $f\in\operatorname{Dom}(\mathcal{A})$ and all $x\in E_1$
that 
  $\left(\mathcal{A}_2 f\right)(x)=\left(\mathcal{A}_1 f\right)(x,h_{\infty}(x))$.
Therefore, for all $t\in[0,\infty)$, all $f\in\operatorname{Dom}(\mathcal{A})$, and all $x\in E_1$
we have $\P$-a.s.
\begin{equation} \begin{split} \label{eq:average.A.over.y}
  &\int_0^t\int_{E_2} \left(\mathcal{A}_1f\right)(F_s,y)\Lambda(ds\times dy)
  =\int_0^t \left(\mathcal{A}_1f\right)(F_s,h_{\infty}(F_s))\,ds
  =
  \int_0^t\left(\mathcal{A}_2f\right)(F_s)\,ds.
\end{split} \end{equation}
Applying Theorem 2.1 of \cite{Kurtz1992} together with \eqref{eq:average.A.over.y},
we see for each $f\in\operatorname{Dom}(\mathcal{A})$ that
\begin{equation}
  \big(f(F_t)-\int_0^t\left(\mathcal{A}_2f\right)(F_s)ds\big)_{t\in[0,\infty)}
\end{equation}
is a martingale.
Hence, $F$ is a (weak) solution of \eqref{eq:X}.
It remains to prove uniqueness.
Note that for all $z_1,z_2\in[0,1]$ we have that
$\tfrac{a-z_1}{a-z_2}(z_2-z_1)=(a-z_1)(\tfrac{a-z_1}{a-z_2}-1)$.
Using this and \eqref{eq:bound.sumi.sigmai.mij} we then have
for any subset $\mathcal{S}\subseteq \D$ and any $x,y\in E_2$ that
\begin{equation} \begin{split}
  &\sum_{i\in\mathcal{S}}\sigma_i\1_{x_i\geq y_i}
  \Big(
  \kappa_H\sum_{j\in\D}m(i,j)\big(\tfrac{(a-x_i)^2}{a-x_j}-(a-x_i)-\tfrac{(a-y_i)^2}{a-y_j}+(a-y_i)\big)
  -\alpha(x_i(1-x_i)-y_i(1-y_i))
  \Big)
  \\
  &
  =
  \sum_{i\in\mathcal{S}}\sigma_i\1_{x_i\geq y_i}
  \Big(
  \kappa_H\sum_{j\in\D}m(i,j)\big((x_i-y_i)+((a-x_i)^2-(a-y_i)^2)\tfrac{1}{a-x_j}-(a-y_i)^2\big(\tfrac{1}{a-y_j}-\tfrac{1}{a-x_j}\big)\big)
  \\
  &\qquad\qquad
  +\alpha(-(x_i-y_i)+x_i^2-y_i^2)
  \Big)
  \\
  &
  \leq
  \sum_{i\in\mathcal{S}}\sigma_i
  \Big(
  \kappa_H\sum_{j\in\D}m(i,j)\1_{x_j\geq y_j}(a-y_i)^2\big(\tfrac{1}{a-x_j}-\tfrac{1}{a-y_j}\big)\Big)
  +\sum_{i\in\mathcal{S}}\sigma_i(\kappa_H+2\alpha)\1_{x_i\geq y_i}(x_i-y_i)
  \\
  &
  \leq
  \sum_{i\in\mathcal{S}}\sigma_i
  \Big(
  \kappa_H\sum_{j\in\D}m(i,j)\1_{x_j\geq y_j}\tfrac{a^2}{(a-1)^2}(x_j-y_j)\Big)
  +\sum_{i\in\mathcal{S}}\sigma_i(\kappa_H+2\alpha)\1_{x_i\geq y_i}(x_i-y_i)
  \\
  &
  \leq
  \sum_{i\in\mathcal{S}}\sigma_i
  c\kappa_H\1_{x_i\geq y_i}\tfrac{a^2(x_i-y_i)}{(a-1)^2}
  +\sum_{i\in\mathcal{S}}\sigma_i(\kappa_H+2\alpha)\1_{x_i\geq y_i}(x_i-y_i)
  =
  \sum_{i\in\mathcal{S}}\sigma_i
  \big(\tfrac{c\kappa_Ha^2}{(a-1)^2}+\kappa_H+2\alpha\big)(x_i-y_i)^+.
\end{split} \end{equation}
This implies that equation (26) of \cite{HutzenthalerWakolbinger2007} is fulfilled.
Together with the assumptions on $m$ in Assumption \ref{ass:conv_LotkaVolterra} we now infer, analogous to Proposition 2.1 of \cite{HutzenthalerWakolbinger2007},
that the system \eqref{eq:X} has a unique strong solution with a.s.~continuous paths.
We conclude that any limit point of $\big\{\left(F_{tN}^N\right)_{t\in[0,\infty)}:N\in\N\big\}$ solves \eqref{eq:X}.
Combining this with the fact that $\big\{\left(F_{tN}^N\right)_{t\in[0,\infty)}:N\in\N\big\}$ is relatively compact we obtain
$\left(F_{tN}^N\right)_{t\in[0,\infty)}\Longrightarrow\left(X_t\right)_{t\in[0,\infty)}$, as $N\to\infty$.
This finishes the proof of Theorem \ref{thm:conv.freq.altruist}.
\end{proof}

\section{McKean-Vlasov limit} \label{sec:mckeanvlasov}
In this section we investigate convergence of a sequence of exchangeable systems of stochastic differential equations
and its application to our model.
\subsection{Setting} \label{sec:mckeanvlasov_setting}
Let $\left(\Omega, \mathcal{F}, \P\right)$ be a probability space,
let $I\subset[0,\infty)$ be an interval of length $|I|\in(0,\infty]$ which is either
of the form $[0,|I|]$ if $|I|<\infty$ or of the form
$[0,\infty)$ if $|I|=\infty$,
let $A\subseteq \R$ be a convex set,
and let $\psi\colon I\to A$,
$\xi\colon A \times I\to \R$, and
$\sigma^2\colon I\to [0,\infty)$
be functions.
The function
$\sigma^2\colon I\to[0,\infty)$
is locally Lipschitz continuous in $I$ and satisfies
$\sigma^2(0)=0$ and
if $|I|<\infty$, then $\sigma^2(|I|)=0$. 
Furthermore, the function $\sigma^2$ is strictly positive on $(0,|I|)$.
There exists a constant $L\in(0,\infty)$ such that
$\sigma^2$ satisfies
the growth condition that for all $y\in I$ we have $\sigma^2(y)\leq L(y+y^2)$
and such that $\xi$ satisfies for all
$(u,x),(v,y)\in A\times I$ that
\begin{equation} \label{eq:ass.xi}
  \1_{x\geq y}\big(\xi(u,x)-\xi(v,y)\big)\leq L|u-v| +L(x-y)^+.
\end{equation}
The function $\psi\colon I\to [0,\infty)$ satisfies for all $x,y\in I$ that $|\psi(x)-\psi(y)|\leq L |x-y|$.
Let
$W(i)\colon [0,\infty) \times \Omega \to \R$, $i\in\N$,
be independent Brownian motions with continuous sample paths.
For all $D\in\N$ let
$X^D\colon [0,\infty) \times \{1,\ldots,D\} \times \Omega \to I$
be an adapted stochastic process with continuous sample paths
that for all $t\in[0,\infty)$ and all $i\in\{1,\ldots,D\}$ $\P$-a.s.~satisfies
\begin{equation} \begin{split} \label{eq:dX}
  X_t^D(i)=X_0^D(i)+&\int_0^t\xi\Big(\tfrac{1}{D}\sum_{j\in\{1,\ldots,D\}}\psi\left(X_s^D(j)\right),X_s^D(i)\Big)\,ds
  +\int_0^t\sqrt{\sigma^2\left(X_s^D(i)\right)}\,dW_s(i).
\end{split} \end{equation}
Let
$M\colon [0,\infty) \times \Omega \to I$
be an adapted stochastic process with continuous sample paths
that for all $t\in[0,\infty)$ $\P$-a.s.~satisfies
\begin{equation} \begin{split} \label{eq:dM}
  M_t
  =M_0+\int_0^t\xi(\E[\psi(M_s)],M_s)\,ds
  +\int_0^t\sqrt{\sigma^2(M_s)}\ dW_s(1).
\end{split} \end{equation}
\subsection{McKean--Vlasov limit} \label{sec:mvl_mvl}
The following proposition, Proposition \ref{prop:MVL}, partly generalizes Proposition 4.29 in \cite{Hutzenthaler2012}
where $\xi$ depends linearly on its first argument.
\begin{proposition} \label{prop:MVL}
  Assume the setting of Section \ref{sec:mckeanvlasov_setting},
  let $M_0$ be an $I$-valued random variable,
  for every $D\in\N$ let $\left(X_0^D(j)\right)_{j\in\{1,\ldots,D\}}$ be exchangeable and integrable random variables with values in $I$.
  Then, there exists a unique solution
  $M$ of~\eqref{eq:dM}
  and for all $D\in\N$ and all $t\in[0,\infty)$ we have that
  \begin{equation}  \label{eq:L1.XM}
    \sqrt{D}\E\Big[\big|X_t^D(1)-M_t\big|\Big]
    \leq e^{(L^2+L+L_\mu)t}
    \left(
    \sqrt{D}
    \E\Big[\big|X_0^D(1)-M_0\big|\Big]
    +L\int_0^t\Big(\Var\big(\psi\big(M_s\big)\big)\Big)^{\frac{1}{2}}ds
    \right).
  \end{equation}
\end{proposition}
\begin{proof}
  Existence of a weak solution is straightforward using a tightness argument.
  Next we show pathwise uniqueness for the SDE \eqref{eq:dM}.
  Let $M,\bar{M}\colon [0,\infty)\times\Omega\to I$ be two solutions of the SDE \eqref{eq:dM}.
  Then our assumptions and a standard Yamada-Watanabe argument (cf., e.g., Theorem 1 \cite{YamadaWatanabe1971})
  shows for all $t\in[0,\infty)$ that $\P$-a.s.
  \begin{equation}
    |M_t-\bar{M}_t|
    =
    |M_0-\bar{M}_0|+\int_0^t\sgn(M_s-\bar{M}_s)d(M_s-\bar{M}_s).
  \end{equation}
  Let $(\tau_l)_{l\in\N}$ be a localizing sequence for the local martingale
  $\big(\int_0^t\sgn(M_s-\bar{M}_s)(\sigma^2(M_s)-\sigma^2(\bar{M}_s))\,dW_s\big)_{t\in[0,\infty)}$.
  Then Fatou's Lemma and our assumptions imply for all $t\in[0,\infty)$ that
  \begin{equation} \begin{split}
    \E[|M_t-\bar{M}_t|]
    &\leq
    \lim_{l\to\infty} \E[|M_{t\land\tau_l}-\bar{M}_{t\land\tau_l}|]
    \\
    &\leq
    \E[|M_0-\bar{M}_0|]
    +
    \E\Big[\int_0^t\sgn(M_s-\bar{M}_s)\big(\xi(\E[\psi(M_s)],M_s)-\xi(\E[\psi(\bar{M}_s)],\bar{M}_s)\big)\,ds\Big]
    \\
    &\leq
    \E[|M_0-\bar{M}_0|]
    +
    L\int_0^t\big|\E[\psi(M_s)]-\E[\psi(\bar{M}_s)]\big|+\E\big[|M_s-\bar{M}_s|\big]\,ds
    \\
    &\leq
    \E[|M_0-\bar{M}_0|]
    +
    (L+1)^2\int_0^t\E\big[|M_s-\bar{M}_s|\big]\,ds.
  \end{split} \end{equation}
  This together with Gronwall's lemma implies pathwise uniqueness for the SDE \eqref{eq:dM}.
  Therefore, the theorem of \cite{YamadaWatanabe1971}
  implies that the SDE \eqref{eq:dM} is exact.
  The rest of the proof is analogous to the proof of Proposition 4.29 in \cite{Hutzenthaler2012} and we omit it here.
\end{proof}
\subsection{Application to costly defense in structured populations}
In this section we verify the applicability of Proposition \ref{prop:MVL} to the case of 
costly
defense in structured populations.
\begin{lemma} \label{lem:mvl_application}
Let $\alpha,\beta,\kappa\in (0,\infty)$ and $a\in (1,\infty)$,
let $I=[0,1]$ and define the function
$\sigma^2\colon I\to[0,\infty)$ by $I\ni x\mapsto\sigma^2(x):=\beta (a-x)x(1-x)$,
the function $\psi\colon I\to [0,\infty)$ by $I\ni x\mapsto\psi(x):=\tfrac{1}{a-x}$, and
the function $\xi\colon [0,\infty)\times I\to\R$ by $[0,\infty)\times I\ni (u,x)\mapsto \xi(u,x):=\kappa(a-x)\big((a-x)u-1\big)-\alpha x(1-x)$.
Then the interval $I$ and the functions $\sigma^2$, $\psi$, and $\xi$ satisfy
the setting of Section \ref{sec:mckeanvlasov_setting}
with
$L=\max\big\{\beta a,\kappa a^2, \kappa+\alpha, \tfrac{1}{(a-1)^2}\big\}$.
\end{lemma}
\begin{proof}
For all $(u,x),(v,y)\in[0,\infty)\times[0,1]$ it holds that
\begin{equation} \begin{split}
  &\1_{x\geq y}\big(\xi(u,x)-\xi(v,y)\big)
  \\
  &\quad
  =
  \1_{x\geq y}\big(\kappa (a-x)((a-x)u-1)-\kappa (a-y)((a-y)v-1)-\alpha x(1-x)+\alpha y(1-y)\big)
  \\
  &\quad
  =
  \1_{x\geq y}\big(\kappa[ (a-x)^2u-(a-x)- (a-y)^2v+(a-y)]-\alpha(1-(x+y))(x-y)\big)
  \\
  &\quad
  =
  \1_{x\geq y}\big(\kappa[ (x-y)+((a-x)^2-(a-y)^2)u- (a-y)^2(v-u)]-\alpha(1-(x+y))(x-y)\big)
  \\
  &\quad
  \leq
  (\kappa+\alpha) (x-y)^+ +\kappa a^2(u-v)^+
  \leq
  L (x-y)^++L(u-v)^+.
\end{split} \end{equation}
Moreover, for all $x,y\in I$ it holds that
$\sigma^2(x)=\beta (a-x)x(1-x)\leq\beta ax\leq L(x+x^2)$
and that
\begin{equation}
  |\psi(x)-\psi(y)|=\big|\tfrac{1}{a-x}-\tfrac{1}{a-y}\big|
  =
  \big|\int_y^x\tfrac{1}{(a-z)^2}\,dz\big|
  \leq
  \tfrac{1}{(a-1)^2}|x-y|
  \leq
  L|x-y|.
\end{equation}
This completes the proof of Lemma \ref{lem:mvl_application}.
\end{proof}
\section{Long-term behavior of the average defender frequency} \label{sec:longtermbehavior}
Subject of this section is the proof of Theorem~\ref{thm:longtermbehavior}
which states a necessary and sufficient condition
under which the 
costly
defense trait goes to fixation in the many-demes limit~\eqref{eq:dZ}.
\subsection{Setting} \label{sec:longtermbehavior_setting}
Let $(\Omega, \mathcal{F}, \P)$ be a probability space,
let $\kappa,\alpha,\beta \in (0,\infty)$, $a\in (1,\infty)$, $c\in(0,1)$,
let
$W\colon [0,\infty) \times \Omega \to \R$
be a Brownian motion with continuous sample paths,
let $Z\colon [0,\infty) \times \Omega \to [0,1]$ be an adapted process with continuous sample paths
that for all $t\in[0,\infty)$ satisfies $\P$-a.s.
\begin{equation} \label{eq:longtermbehavior_dZ}
  Z_t=Z_0+\int_0^t\big(\kappa (a-Z_s)\big((a-Z_s)\E\big[\tfrac{1}{a-Z_s}\big]-1\big)-\alpha Z_s(1-Z_s)\big)\,ds+\int_0^t\sqrt{\beta(a-Z_s)Z_s(1-Z_s)}\,dW_s.
\end{equation}
Moreover, for all $\theta\in\big(\tfrac{1}{a},\tfrac{1}{a-1}\big)$
let $Z^\theta\colon [0,\infty) \times \Omega \to [0,1]$ be an adapted process with continuous sample paths
that for all $t\in[0,\infty)$ satisfies $\P$-a.s.
\begin{equation} \label{eq:longtermbehavior_dZtheta}
  Z_t^\theta=Z_0^\theta+\int_0^t\left(\kappa \left(a-Z_s^\theta\right)\left(\left(a-Z_s^\theta\right)\theta-1\right)
  -\alpha Z_s^\theta\left(1-Z_s^\theta\right)\right)\,ds
  +\int_0^t\sqrt{\beta \left(a-Z_s^\theta\right)Z_s^\theta\left(1-Z_s^\theta\right)}\,dW_s.
\end{equation}
For all $\theta\in\big(\tfrac{1}{a},\tfrac{1}{a-1}\big)$ and all $z\in[0,1]$ define
\begin{equation}\begin{split} \label{eq:longterm_m}
  m_\theta(z):=&
  \beta c^{\frac{2\kappa }{\beta }(a\theta-1)}
  (1-c)^{\frac{2\kappa }{\beta }(1-\theta(a-1))}
  (a-c)^\frac{2\alpha }{\beta }\tfrac{1}{\beta (a-z)z(1-z)}\exp\left(\int_c^z 2\tfrac{\kappa (a-y)((a-y)\theta-1)-\alpha y(1-y)}{\beta (a-y)y(1-y)}dy\right)
  \\
  =&
  z^{\frac{2\kappa }{\beta }(a\theta-1)-1}
  (1-z)^{\frac{2\kappa }{\beta }(1-\theta(a-1))-1}
  (a-z)^{\frac{2\alpha }{\beta }-1}.
\end{split}\end{equation}
Note that this defines the speed density \citep[][p.~95]{KarlinTaylor1981b}
for \eqref{eq:longtermbehavior_dZtheta}.
Furthermore, note that for all $\theta\in\big(\tfrac{1}{a},\tfrac{1}{a-1}\big)$ it holds that
\begin{equation} \label{eq:speedfinite}
  \int_0^1 m_\theta(z)dz <\infty.
\end{equation}
For all $\theta\in\big(\tfrac{1}{a},\tfrac{1}{a-1}\big)$
define $c_\theta:=\int_0^1 m_\theta(z)dz$,
for all $x\in\{0,1\}$ denote by $\delta_x$ the Dirac measure on $[0,1]$,
and for all
$\theta\in\big[\tfrac{1}{a},\tfrac{1}{a-1}\big]$
define the mapping
$\Psi_\theta:\mathcal{B}([0,1])\to[0,1]$ by
\begin{equation} \label{eq:Gammatheta}
  \mathcal{B}([0,1])\ni A \mapsto \Psi_\theta(A):=
  \begin{cases}
    \delta_0(A), &\text{ if } \theta=\tfrac{1}{a},\\
    \delta_1(A), &\text{ if } \theta=\tfrac{1}{a-1},\\
    \int_{A}\tfrac{1}{c_\theta} m_\theta(z)\,dz, &\text{ if } \theta\in\big(\tfrac{1}{a},\tfrac{1}{a-1}\big).
  \end{cases}
\end{equation}

\subsection{Results for the equilibrium distribution} \label{sec:longtermbehavior_preliminaries}
Assume the setting of Section \ref{sec:longtermbehavior_setting}.
Existence and uniqueness of the solution of \eqref{eq:longtermbehavior_dZ} follow from Proposition \ref{prop:MVL}.
When $\theta\in\big(\tfrac{1}{a},\tfrac{1}{a-1}\big)$ we have that
$\Psi_\theta$ defines a probability distribution by \eqref{eq:speedfinite},
and we can apply Theorem V.54.5 of \cite{RogersWilliams2000b}
to conclude that it is the unique equilibrium distribution for \eqref{eq:longtermbehavior_dZtheta}.
The proof of the following lemma, Lemma \ref{lem:equil.dist.Z}, is clear and therefore omitted.
\begin{lemma} \label{lem:equil.dist.Z}
  Assume the setting of Section \ref{sec:longtermbehavior_setting}.
  A probability measure $\Phi\colon\mathcal{B}([0,1])\to[0,1]$ is an equilibrium distribution of the dynamics \eqref{eq:longtermbehavior_dZ}
  if and only if there exists a $\theta\in\big[\tfrac{1}{a},\tfrac{1}{a-1}\big]$
  such that $\Phi=\Psi_\theta$.
\end{lemma}
\begin{lemma} \label{lem:condition.sb}
  Assume the setting of Section \ref{sec:longtermbehavior_setting} and
  let $\theta\in\big(\tfrac{1}{a},\tfrac{1}{a-1}\big)$.
  Then we have
  \begin{equation}
    \int_0^1 \tfrac{1}{a-z}\,\Psi_\theta(dz) \begin{cases}
      &<\theta, \text{ if } \alpha >\beta ,\\
      &=\theta, \text{ if } \alpha =\beta ,\\
      &>\theta, \text{ if } \alpha <\beta .
    \end{cases}
  \end{equation}
\begin{proof}
  Define $u:=\tfrac{2\kappa }{\beta }(a\theta-1)$ and
  $v:=\tfrac{2\kappa }{\beta }(1-\theta(a-1))$ and note that $u,v\in(0,\infty)$.
  Let $\Gamma\colon (0,\infty)\to (0,\infty)$ be the Gamma function, i.e.,
  for all $x\in(0,\infty)$ let $\Gamma(x):=\int_0^\infty z^{x-1}e^{-z}\,dz$.
  It is well-known that for all $x\in(0,\infty)$ the Gamma function satisfies $\Gamma(x+1)=x\Gamma(x)$
  and that for all $x,y\in(0,\infty)$ it holds that
  $\int_0^1 z^{x-1}(1-z)^{y-1}\,dz=\tfrac{\Gamma(x)\Gamma(y)}{\Gamma(x+y)}$.
  Thus, we obtain
  \begin{equation} \begin{split} \label{eq:integral.0}
    &\int_0^1 z^{u-1}(1-z)^{v-1}
    (a-z)\left(\tfrac{1}{a-z}-\theta\right)\,dz
    \\
    &=
    \int_0^1 z^{u-1}(1-z)^{v-1}\,dz 
    -a\theta\int_0^1 z^{u-1}(1-z)^{v-1}\,dz
    +\theta\int_0^1 z^{u}(1-z)^{v-1}\,dz
    \\
    &=\tfrac{\Gamma(u)\Gamma(v)}{\Gamma(u+v)}-a\theta\tfrac{\Gamma(u)\Gamma(v)}{\Gamma(u+v)}
    +\theta\tfrac{\Gamma(u+1)\Gamma(v)}{\Gamma(u+v+1)}
    =\left((1-a\theta)\tfrac{(u+v)\Gamma(u)\Gamma(v)}{\Gamma(u+v+1)}+\theta\tfrac{u\Gamma(u)\Gamma(v)}{\Gamma(u+v+1)}\right)
    \\
    &=(u(1-a\theta+\theta)+v(1-a\theta))\tfrac{\Gamma(u)\Gamma(v)}{\Gamma(u+v+1)}
    =\tfrac{2\kappa }{\beta }\left((a\theta-1)(1-\theta(a-1))+(1-\theta(a-1))(1-a\theta)\right)
    \tfrac{\Gamma(u)\Gamma(v)}{\Gamma(u+v+1)}
    \\
    &=\left(\tfrac{2\kappa }{\beta }(1-\theta(a-1))(a\theta-1+1-a\theta)\right)\tfrac{\Gamma(u)\Gamma(v)}{\Gamma(u+v+1)}
    =0.
  \end{split} \end{equation}
  
  First, consider the case $\alpha =\beta $.
  Using \eqref{eq:integral.0} we see that
  \begin{equation} \begin{split}
    \int_0^1
    \tfrac{1}{a-z}\,\Psi_\theta(dz) - \theta
    &=
    \int_0^1 c_\theta z^{\frac{2\kappa }{\beta }(a\theta-1)-1}(1-z)^{\frac{2\kappa }{\beta }(1-\theta(a-1))-1}
    (a-z)^{\frac{2\alpha }{\beta }-1}\left(\tfrac{1}{a-z}-\theta\right)\,dz
    \\
    &=
    c_\theta\int_0^1 z^{u-1}(1-z)^{v-1}
    (a-z)\left(\tfrac{1}{a-z}-\theta\right)\,dz
    =0.
  \end{split} \end{equation}
  
  Now, consider the case $\alpha >\beta $.
  Let $\hat{\delta}:=\alpha -\beta$, $\delta:=\tfrac{2\hat{\delta}}{\beta}$,
  and $z^*:=\sup\{z\in(0,1):\tfrac{1}{a-z}-\theta<0\}$.
  Note that $\hat{\delta},\delta>0$ and $z^*=a-\tfrac{1}{\theta}\in (0,1)$.
  Also note that for all $z\in(0,z^*)$ we have $\tfrac{1}{a-z}-\theta<0$ and $(a-z)^\delta>(a-z^*)^\delta$.
  Furthermore, for all $z\in(z^*,1)$ we have $\tfrac{1}{a-z}-\theta>0$ and $(a-z)^\delta<(a-z^*)^\delta$.
  Together with \eqref{eq:integral.0} we thereby obtain
  \begin{equation} \begin{split}
    &\int_0^1 \tfrac{1}{a-z}\,\Psi_\theta(dz)-\theta
    =\int_0^1 \left(\tfrac{1}{a-z}-\theta\right)\,\Psi_\theta(dz)
    =
    \int_0^1 c_\theta z^{u-1}(1-z)^{v-1}
    (a-z)^{\frac{2\alpha }{\beta }-1}\left(\tfrac{1}{a-z}-\theta\right)\,dz
    \\
    &=
    \int_0^{z^*} c_\theta z^{u-1}(1-z)^{v-1}
    (a-z)^{1+\delta}\left(\tfrac{1}{a-z}-\theta\right)\,dz
    +\int_{z^*}^1 c_\theta z^{u-1}(1-z)^{v-1}
    (a-z)^{1+\delta}\left(\tfrac{1}{a-z}-\theta\right)\,dz
    \\
    &<
    c_\theta (a-z^*)^\delta\bigg(
    \int_0^{z^*}  z^{u-1}(1-z)^{v-1}
    (a-z)\left(\tfrac{1}{a-z}-\theta\right)\,dz
    +\int_{z^*}^1 z^{u-1}(1-z)^{v-1}
    (a-z)\left(\tfrac{1}{a-z}-\theta\right)\,dz
    \bigg)
    \\
    &=
    c_\theta(a-z^*)^\delta\int_0^1 z^{u-1}(1-z)^{v-1}
    (a-z)\left(\tfrac{1}{a-z}-\theta\right)\,dz=0.
  \end{split} \end{equation}
  The case $\alpha <\beta$ can be proved analogously and thereby, we omit it here.
  This completes the proof.
\end{proof}
\end{lemma}
\subsection{Proof of Theorem \ref{thm:longtermbehavior}} \label{sec:longtermbehavior_proof}
\begin{proof}[Proof of Theorem \ref{thm:longtermbehavior}.]
  Applying It\^o's lemma, we get for all $t\in[0,\infty)$ that
  \begin{equation} \begin{split}
    &\tfrac{1}{a-Z_t}-\tfrac{1}{a-Z_0}
    =\int_0^t\tfrac{1}{(a-Z_s)^2}\left(\kappa (a-Z_s)\left((a-Z_s)\E\left[\tfrac{1}{a-Z_s}\right]-1\right)-\alpha Z_s(1-Z_s)\right)
    \\
    &\qquad\qquad\qquad\qquad
    +\tfrac{1}{2}\tfrac{2(a-Z_s)}{(a-Z_s)^4}\beta (a-Z_s)Z_s(1-Z_s)\,ds+\int_0^t \tfrac{1}{(a-Z_s)^2} \sqrt{\beta (a-Z_s)Z_s(1-Z_s)}\,dW_s
    \\
    &
    =\int_0^t \kappa \left(\E\left[\tfrac{1}{a-Z_s}\right]-\tfrac{1}{a-Z_s}\right)
    -\tfrac{\alpha Z_s(1-Z_s)}{(a-Z_s)^2}+\tfrac{\beta Z_s(1-Z_s)}{(a-Z_s)^2}\,ds
    +\int_0^t \tfrac{1}{(a-Z_s)^2} \sqrt{\beta (a-Z_s)Z_s(1-Z_s)}\,dW_s.
  \end{split} \end{equation}
  After taking expectations we can apply Fubini's theorem to obtain
  for all $t\in[0,\infty)$ that
  \begin{equation} \begin{split} \label{eq:expectation.monotone}
    \E\big[\tfrac{1}{a-Z_t}\big]-\E\left[\tfrac{1}{a-Z_0}\right]
    &=
    \int_0^t \kappa \left(\E\left[\tfrac{1}{a-Z_s}\right]-\E\left[\tfrac{1}{a-Z_s}\right]\right)
    -\alpha \E\left[\tfrac{Z_s(1-Z_s)}{(a-Z_s)^2}\right]+\beta \E\left[\tfrac{Z_s(1-Z_s)}{(a-Z_s)^2}\right]\,ds
    \\
    &=(\beta -\alpha )\int_0^t \E\left[\tfrac{Z_s(1-Z_s)}{(a-Z_s)^2}\right]\,ds.
  \end{split} \end{equation}
  Since for all $s\in[0,\infty)$ it holds that $\E\big[\tfrac{Z_s(1-Z_s)}{(a-Z_s)^2}\big]\geq0$
  we conclude that the function $[0,\infty)\ni t\mapsto\E\big[\tfrac{1}{a-Z_t}\big]\in\big[\tfrac{1}{a},\tfrac{1}{a-1}\big]$ converges monotonically
  non-increasing
  as $t\to\infty$ if $\alpha >\beta $, monotonically non-decreasing
  if $\alpha <\beta $,
  or is constant if $\alpha =\beta $.
  
  First, assume $\alpha >\beta $. From \eqref{eq:longtermbehavior_dZ} we see that $\delta_1$ is
  an invariant measure for $Z$.
  So if $\P[Z_0=1]=1$, then for all $t\in[0,\infty)$ it holds that $\P[Z_t=1]=1$.
  Now let $\P[Z_0=1]<1$, implying $\E\left[\tfrac{1}{a-Z_0}\right]\in\big[\tfrac{1}{a},\tfrac{1}{a-1}\big)$.
  Define $\theta:=\lim\limits_{t\to\infty}\E\big[\tfrac{1}{a-Z_t}\big]$ and fix it for the rest of the paragraph.
  Note that due to the monotonicity stated above we have $\theta\in\big[\tfrac{1}{a},\tfrac{1}{a-1}\big)$.
  Aiming at a contradiction, we assume that $\theta\in\big(\tfrac{1}{a},\tfrac{1}{a-1}\big)$.
  Choose any $\eps\in\big(0,\tfrac{1}{a-1}-\theta\big)$ and fix it for the rest of the proof.
  By definition of $\theta$ there exists an $s_\eps\in(0,\infty)$,
  such that for all $t\in[s_\eps,\infty)$ it holds that $\E\big[\tfrac{1}{a-Z_t}\big]<\theta+\eps$.
  Let
  $\tilde{W}\colon [0,\infty) \times \Omega \to \R$
  be a Brownian motion with continuous sample paths,
  let $\tilde{Z}\colon [0,\infty) \times \Omega \to [0,1]$
  and 
  $\tilde{Z}^{\theta+\eps}\colon [0,\infty) \times \Omega \to [0,1]$
  be adapted processes with continuous sample paths  
  that satisfy for all $t\in[0,\infty)$ $\P$-a.s.
  \begin{equation} \begin{split} \label{eq:coupling}
    \tilde{Z}_t
    &=
    \tilde{Z}_0+\int_0^t\left(\kappa (a-\tilde{Z}_s)\left((a-\tilde{Z}_s)\E\left[\tfrac{1}{a-\tilde{Z}_s}\right]-1\right)
    -\alpha \tilde{Z}_s(1-\tilde{Z}_s)\right)\,ds
    \\
    &\quad
    +\int_0^t\sqrt{\beta (a-\tilde{Z}_s)\tilde{Z}_s(1-\tilde{Z}_s)}\,d\tilde{W}_s,
    \\
    \tilde{Z}_t^{\theta+\eps}
    &=
    \tilde{Z}_0^{\theta+\eps}+\int_0^t\left(\kappa (a-\tilde{Z}_s^{\theta+\eps})\left((a-\tilde{Z}_s^{\theta+\eps})(\theta+\eps)-1\right)
    -\alpha \tilde{Z}_s^{\theta+\eps}(1-\tilde{Z}_s^{\theta+\eps})\right)\,ds
    \\
    &\quad
    +\int_0^t\sqrt{\beta (a-\tilde{Z}_s^{\theta+\eps})\tilde{Z}_s^{\theta+\eps}(1-\tilde{Z}_s^{\theta+\eps})}\,d\tilde{W}_s,
  \end{split} \end{equation}
  such that $\tilde{Z}_0^{\theta+\eps}=\tilde{Z}_0$ and such that $\tilde{Z}_0$ and $Z_{s_\eps}$ are equal in distribution.
  Then for each $t\in[s_\eps,\infty)$ we have that $Z_t$ and $\tilde{Z}_{t-s_\eps}$ are equal in distribution
  and the drift term of
  $\tilde{Z}_{t-s_\eps}$ is lower than that of $\tilde{Z}_{t-s_\eps}^{\theta+\eps}$.
  Together with the fact that the mapping $[0,1]\ni z\mapsto\tfrac{1}{a-z}$ is strictly monotonically increasing
  this implies for all $t\in[s_\eps,\infty)$ that
  \begin{equation} \label{eq:bound.coupling}
    \E\big[\tfrac{1}{a-Z_t}\big]
    =\E\Big[\tfrac{1}{a-\tilde{Z}_{t-s_\eps}}\Big]
    \leq\E\Big[\tfrac{1}{a-\tilde{Z}_{t-s_\eps}^{\theta+\eps}}\Big].
  \end{equation}
  Recall from Section \ref{sec:longtermbehavior_preliminaries} that for any $\eta\in\big(\tfrac{1}{a},\tfrac{1}{a-1}\big)$ we have that
  $\Psi_\eta$ is the unique equilibrium distribution of $\tilde{Z}^\eta$.
  Combining this with \eqref{eq:bound.coupling} we obtain (see, e.g., Theorem V.54.5 \cite{RogersWilliams2000b})
  \begin{equation} \label{eq:theta<int_psitheta+eps}
    \theta=\lim\limits_{t\to\infty}\E\big[\tfrac{1}{a-Z_t}\big]
    \leq\lim\limits_{t\to\infty}\E\Big[\tfrac{1}{a-\tilde{Z}_{t-s_\eps}^{\theta+\eps}}\Big]
    =\int_0^1 \tfrac{1}{a-z}\,\Psi_{\theta+\eps}(dz).
  \end{equation}
  The dominated convergence theorem yields that
  the mapping $\big(\tfrac{1}{a},\tfrac{1}{a-1}\big)\ni \eta \mapsto \Psi_\eta$
  is continuous with respect to the weak topology.
  Applying this, \eqref{eq:theta<int_psitheta+eps} together with the fact that $\eps\in\big(0,\tfrac{1}{a-1}-\theta\big)$ was arbitrarily chosen,
  and Lemma \ref{lem:condition.sb}, we obtain the contradiction 
  \begin{equation} \begin{split}
    \theta
    \leq\lim\limits_{\delta\to0} \int_0^1 \tfrac{1}{a-z}\,\Psi_{\theta+\delta}(dz)
    =\int_0^1 \tfrac{1}{a-z}\,\Psi_\theta(dz)<\theta.
  \end{split} \end{equation}
  Hence, we have $\theta=\tfrac{1}{a}$, implying
  \begin{equation} \label{eq:lim.EZt.alpha>beta}
    0\leq \lim\limits_{t\to\infty}\E\left[Z_t\right]
    \leq
    \lim\limits_{t\to\infty}a^2\E\left[\tfrac{Z_t}{a(a-Z_t)}\right]
    =
    \lim\limits_{t\to\infty}a^2\E\big[\tfrac{1}{a-Z_t}\big]-a^2\tfrac{1}{a}
    = 0.
  \end{equation}
  
  The case $\alpha < \beta$ can be proved analogously and we omit it here.
  
  Finally, assume $\alpha =\beta $, define $\theta:=\E[\tfrac{1}{a-Z_0}]$, and fix it for the rest of the proof.
  We see from \eqref{eq:expectation.monotone} that $\E[\tfrac{1}{a-Z_t}]$ is constant in $t\in[0,\infty)$.
  Thus, assuming that $Z_0$ and $Z_0^\theta$ are equal in distribution we see from \eqref{eq:longtermbehavior_dZ} and \eqref{eq:longtermbehavior_dZtheta} that for all $t\in[0,\infty)$ it holds that
  $Z_t$ and $Z_t^\theta$ are equal in distribution.
  Recall from Section \ref{sec:longtermbehavior_preliminaries} that $\Psi_\theta$ is the unique equilibrium distribution of $Z^\theta$.
  Consequently, $\Psi_\theta$ is the unique equilibrium distribution of $Z$.
  This completes the proof of Theorem \ref{thm:longtermbehavior}.
\end{proof}

\section{Invasion of 
a costly
defense allele} \label{sec:invasion}
In this section we investigate a tree of excursions which could be the many-demes limit of the total mass process
when there is only one island occupied initially.
For this tree of excursions,
Proposition~\ref{prop:invasion} establishes a necessary and sufficient condition for extinction.
Thus, informally speaking, we get an explicit condition when the invasion probability is positive in an infinite-dimensional space.
\subsection{Setting} \label{sec:invasion_setting}
Let $\left(\Omega, \mathcal{F}, \P\right)$ be a probability space,
let $\kappa,\alpha,\beta \in (0,\infty)$, $a\in (1,\infty)$,
and let 
$W(i)\colon [0,\infty) \times \Omega \to \R$, $i\in\N$,
be independent Brownian motions with continuous sample paths.
For all $D\in\N$ let
$X^D\colon [0,\infty) \times \{1,\ldots,D\} \times \Omega \to [0,1]$
be an adapted process with continuous sample paths
that for all $t\in[0,\infty)$ and all $i\in\{1,\ldots,D\}$ $\P$-a.s.~satisfies
\begin{equation} \begin{split} \label{eq:invasion.dX}
  X_t^D(i)=
  X_0^D(i)
  &+\int_0^t\kappa(a-X_s^D(i))\bigg((a-X_s^D(i))\tfrac{1}{D}\sum_{j=1}^{D}\tfrac{1}{a-X_s^D(j)}-1\bigg)-\alpha X_s^D(i)(1-X_s^D(i))\,ds
  \\
  &+\int_0^t\sqrt{\beta (a-X_s^D(i))X_s^D(i)(1-X_s^D(i))}\,dW_s(i).
\end{split} \end{equation}
Let $\tilde{a}\colon[0,1]\to[0,\infty)$ be a function defined by
\begin{equation}
[0,1]\ni x \mapsto \tilde{a}(x):=\kappa a \tfrac{x}{a-x}.
\end{equation}
Then, assuming there is positive mass only in deme 1,
the dynamics in deme 1 follows asymptotically the following process $Y$.
Let $Y\colon[0,\infty)\times\Omega\to[0,1]$
be an adapted process with continuous sample paths
such that for all $t\in[0,\infty)$ it $\P$-a.s.~holds that
\begin{equation} \label{eq:invasion.dY}
  Y_t=Y_0
  -\int_0^t\tfrac{\kappa}{a}Y_s(a-Y_s)+\alpha Y_s(1-Y_s)\,ds
  +\int_0^t\sqrt{\beta (a-Y_s)Y_s(1-Y_s)}\,dW_s(1).
\end{equation}
In addition, let $\mathrm{Q}_Y$ be the excursion measure which satisfies
$\mathrm{Q}_Y=\lim_{0<\eps \to 0} \tfrac{1}{\eps} \P[Y \in \cdot | Y_0=\eps]$
in a suitable sense;
see \cite{PitmanYor1982} and \cite{Hutzenthaler2009} for details.
Asymptotically in the many-demes limit, every deme with population path $\chi \in C( [0,\infty), [0,1])$ populates demes through migration
and these new populations are given by a Poisson point process with intensity measure $\tilde{a}(\chi_t)dt \times \mathrm{Q}_Y(d\psi)$.
Now let $(V_t)_{t\in[0,\infty)}$ be the total mass process of the associated tree of excursions with initial island measure
that equals the distribution of $Y$ in \eqref{eq:invasion.dY} and excursion measure $\mathrm{Q}_Y$.
\subsection{Survival or extinction of an invading 
costly
defense allele} \label{sec:invasion_result}
\begin{proposition} \label{prop:invasion}
  Assume the setting of Section \ref{sec:invasion_setting}.
  Let $x\in(0,1]$ and assume $Y_0=x=V_0$.
  Then the total mass process dies out (i.e., converges in probability to zero as $t\to\infty$) if and only if
  \begin{equation}
    \alpha \geq \beta.
  \end{equation}
\end{proposition}
\begin{proof}
  Define the functions $s\colon[0,1]\to [0,\infty)$ and $S\colon[0,1]\to[0,\infty)$ for all $y\in[0,1]$ by
  $s(y):=\exp\Big(-\int_{0}^y\tfrac{-\frac{\kappa}{a}x(a-x)-\alpha x(1-x)}{\frac{1}{2}\beta (a-x)x(1-x)}\,dx\Big)$
  and by
  $S(y):=\int_0^ys(z)\,dz$.
  Note that for all $z\in[0,1]$ it holds that
  \begin{equation} \label{eq:invasion_s}
    s(z)
    =
    \exp\Big(\int_{0}^{z}\tfrac{2\kappa}{a\beta}\tfrac{1}{1-x}+\tfrac{2\alpha}{\beta}\tfrac{1}{a-x}\,dx\Big)
    =
    (1-z)^\frac{-2\kappa}{a\beta}\Big(\tfrac{a-z}{a}\Big)^\frac{-2\alpha}{\beta}
  \end{equation}
  and
  \begin{equation} \label{eq:invasion_S}
    S(z)=\int_0^z s(x)\,dx\leq z s(z).
  \end{equation}
  We will apply Theorem 5 from \cite{Hutzenthaler2009} to show the result.
  First, we verify that the assumptions of the aforementioned theorem are satisfied.
  Using \eqref{eq:invasion_S}, we see that
  \begin{equation} \label{eq:invasion_verify_1}
    \int_0^\frac{1}{2}S(y)\tfrac{2}{\beta(a-y)y(1-y)s(y)}\,dy
    \leq
    \int_0^\frac{1}{2}\tfrac{2}{\beta(a-y)(1-y)}\,dy
    \leq
    \tfrac{1}{2}\tfrac{2}{\beta(a-\frac{1}{2})(1-\frac{1}{2})}
    <\infty.
  \end{equation}
  Furthermore, we get
  \begin{equation} \begin{split} \label{eq:invasion_verify_2}
    \lim_{\eps\to 0}\int_{\eps}^{\frac{1}{2}}\tfrac{-\frac{\kappa}{a}(a-y)y-\alpha y(1-y)}{\frac{1}{2}\beta(a-y)y(1-y)}\,dy
    &=
    \lim_{\eps\to 0}\int_{\eps}^{\frac{1}{2}}\tfrac{-2\kappa}{a\beta(1-y)}-\tfrac{2\alpha}{\beta(a-y)}\,dy
    \\
    &
    =
    \lim_{\eps\to 0}\Big(\tfrac{2\kappa}{a\beta}(\ln(1-\tfrac{1}{2})-\ln(1-\eps))+\tfrac{2\alpha}{\beta}(\ln(a-\tfrac{1}{2})-\ln(a-\eps))\Big)
    \\
    &\quad
    =
    \tfrac{2\kappa}{a\beta}\ln(1-\tfrac{1}{2})+\tfrac{2\alpha}{\beta}(\ln(a-\tfrac{1}{2})-\ln(a))
    \in(-\infty,\infty).
  \end{split} \end{equation}
  From \eqref{eq:invasion_s} as well as the fact that $\tfrac{2\kappa}{a\beta}>0$ we see that
  \begin{equation} \begin{split} \label{eq:invasion_verify_3}
    \int_{\frac{1}{2}}^{1}\tfrac{\tilde{a}(y)}{\frac{1}{2}\beta(a-y)y(1-y) s(y)}\,dy
    &=
    \int_{\frac{1}{2}}^{1}\tfrac{\kappa a\frac{y}{a-y}}{\frac{1}{2}\beta(a-y)y(1-y)}
    (1-y)^\frac{2\kappa}{a\beta}\Big(\tfrac{a-y}{a}\Big)^\frac{2\alpha}{\beta}\,dy
    \\
    &=
    2\tfrac{\kappa a}{\beta}a^{-\frac{2\alpha}{\beta}}
    \int_{\frac{1}{2}}^{1}
    (1-y)^{\frac{2\kappa}{a\beta}-1}
    (a-y)^{\frac{2\alpha}{\beta}-2}
    \,dy
    \\
    &\leq
    2\tfrac{\kappa a}{\beta}a^{-\frac{2\alpha}{\beta}}
    \Big((a-\tfrac{1}{2})^{\frac{2\alpha}{\beta}-2}+(a-1)^{\frac{2\alpha}{\beta}-2}\Big)
    \int_{\frac{1}{2}}^{1}
    (1-y)^{\frac{2\kappa}{a\beta}-1}
    \,dy
    <\infty.
  \end{split} \end{equation}
  We obtain from \eqref{eq:invasion_verify_1}, \eqref{eq:invasion_verify_2}, and \eqref{eq:invasion_verify_3}
  together with a straightforward adaptation of Lemmas 9.6, 9.9, and 9.10 in \cite{Hutzenthaler2009} to the state space $[0,1]$
  that the assumptions of Theorem 5 in \cite{Hutzenthaler2009} are satisfied.
  Applying the aforementioned theorem shows that the total mass process dies out if and only if
 \begin{equation} \label{eq:invasion_application_thm5VIM}
   \int \int_0^\infty \tilde{a}(\chi_t)\,dt\,\mathrm{Q}_Y(d\chi) \leq 1.
  \end{equation}
  Moreover, a straight forward adaptation of Lemma 9.8 in \cite{Hutzenthaler2009} to the state space $[0,1]$
  together with  \eqref{eq:invasion_verify_1} shows that
  \begin{equation} \label{eq:invasion_application_lemma9.8VIM}
    \int \int_0^\infty \tilde{a}(\chi_t)\,dt\,\mathrm{Q}_Y(d\chi)
    =
    \int_0^1
    \tfrac{\kappa a\frac{y}{a-y}}{\frac{1}{2}\beta(a-y)y(1-y)}
    (1-y)^\frac{2\kappa}{a\beta}\Big(\tfrac{a-y}{a}\Big)^\frac{2\alpha}{\beta}\,dy.
  \end{equation}
  Observe that we have $\tfrac{2\kappa}{a\beta}\int_0^1(1-y)^{\frac{2\kappa}{a\beta}-1}\,dy=1$.
  Combining this with \eqref{eq:invasion_application_thm5VIM} and \eqref{eq:invasion_application_lemma9.8VIM}
  we see that the total mass process dies out if and only if
  \begin{equation} \begin{split} \label{eq:invasion_condition_extinction}
    0
    &\geq
    \int_0^1
    \tfrac{\kappa a\frac{y}{a-y}}{\frac{1}{2}\beta(a-y)y(1-y)}
    (1-y)^\frac{2\kappa}{a\beta}\Big(\tfrac{a-y}{a}\Big)^\frac{2\alpha}{\beta}\,dy
    -1
    \\
    &=
    \tfrac{2\kappa}{a\beta}
    \int_0^1
    (1-y)^{\frac{2\kappa}{a\beta}-1}\Big(\tfrac{a-y}{a}\Big)^{\frac{2\alpha}{\beta}-2}\,dy
    -1
    =
    \tfrac{2\kappa}{a\beta}\int_0^1
    (1-y)^{\frac{2\kappa}{a\beta}-1}
    \bigg(\Big(\tfrac{a-y}{a}\Big)^{\frac{2\alpha}{\beta}-2}-1\bigg)\,dy.
  \end{split} \end{equation}
  Consequently, the total mass process dies out if and only if $\alpha\geq\beta$.
  This completes the proof of Proposition \ref{prop:invasion}.
\end{proof}
\section*{Appendix}
In this appendix we prove
Lemmas \ref{lem:bound.norm.(H+P)^p}, \ref{lem:bound.norm.1/H^2}, and \ref{lem:bound.norm.1/P}.
The main step in the proofs of these lemmas is a comparison with 
the solution of
an ordinary differential equation whose derivative is  negative near infinity.
Solutions of such ordinary differential equations are bounded uniformly in time.

\begin{proof}[Proof of Lemma \ref{lem:bound.norm.(H+P)^p}]
If we assume  $\sup_{N\in\N}\E\left[\left\|\left(H_0^N+P_0^N\right)^p\right\|_{\sigma}\right]=\infty$, then the claim trivially holds.
For the remainder of the proof assume $\sup_{N\in\N}\E\left[\left\|\left(H_0^N+P_0^N\right)^p\right\|_{\sigma}\right]<\infty$.
Define $\D_0:=\emptyset$ and 
for every $n\in\N$ let $\D_n\subseteq\D$ be a set with $|\D_n|=\min\left\{n,|\D|\right\}$ and $\D_n\supseteq\D_{n-1}$.
Define real numbers
$c_0:=\min\left\{\tfrac{1}{2\eta}\tfrac{\lambda}{K},\tfrac{1}{4},\tfrac{1}{\delta}\gamma\right\}$,
$c_1:=p\left[2\eta\bar{\iota}_H+\delta\bar{\iota}_P+(p-1)\left(2\eta\betab_H+\tfrac{1}{2}\delta\betab_P\right)\right]\in(0,\infty)$,
$c_2:=\lambda p+(p-1+c)(\bar{\kappa}_{H}+\bar{\kappa}_{P})\in(0,\infty)$,
$c_3:=c_0p\left(\sum_{k\in\D}\sigma_k\right)^{-\tfrac{1}{p}}\in(0,\infty)$, and
$c_4:=c_1\left(\sum_{k\in\D}\sigma_k\right)^{\tfrac{1}{p}}\in(0,\infty)$.
For all $N\in\N$, $t\in[0,\infty)$ define
$Y_t^N:=2\eta H_t^N+\delta P_t^N$
and for all $N,n\in\N$ and all $t\in[0,\infty)$
let $M_t^{N,n}$ be a real-valued random variable such that $\P$-a.s.~it holds that
\begin{equation} \begin{split}
  M_t^{N,n}=
  \sum_{i\in\D_n}\sigma_i\Bigg(
  &\int_0^t
  2\eta p\left(Y_u^N(i)\right)^{p-1}\sqrt{\beta_H^N H_u^N(i)}
  \,dW_u^{H,N}(i)
  +
  \int_0^t
  \delta p\left(Y_u^N(i)\right)^{p-1}\sqrt{\beta_P^N P_u^N(i)}
  \,dW_u^{P,N}(i)\Bigg).
\end{split} \end{equation}
Applying It\^o's lemma
we get for all $N,n\in\N$ and all $t\in[0,\infty)$
that $\P$-a.s.
\begin{equation} \begin{split}
  &\sum_{i\in\D_n}\sigma_i\left(Y_t^N(i)\right)^p
  -\sum_{i\in\D_n}\sigma_i\left(Y_0^N(i)\right)^p
  \\
  &\quad
  =
  \sum_{i\in\D_n}\sigma_i\int_0^t
  2\eta p\left(Y_u^N(i)\right)^{p-1}
  \Bigg(\kappa_H^N\sum_{j\in\D} m(i,j)H_u^N(j)
  +(\lambda-\kappa_H^N-\alpha^NF_u^N(i))H_u^N(i)
  \\
  &\qquad\quad
  -\tfrac{\lambda}{K}\left(H_u^N(i)\right)^2
  -\delta P_u^N(i)H_u^N(i)
  +\iota_H^N\Bigg)
  +\delta p\left(Y_u^N(i)\right)^{p-1}
  \Bigg(\kappa_P^N\sum_{j\in\D} m(i,j)P_u^N(j)
  \\
  &\qquad\quad
  -(\kappa_P^N+\nu)P_u^N(i)
  -\gamma \left(P_u^N(i)\right)^2
  +\left(\eta-\rho F_u^N(i)\right)P_u^N(i)H_u^N(i)
  +\iota_P^N\Bigg)
  \\
  &\qquad\quad
  +\tfrac{1}{2}4\eta^2 p(p-1)\left(Y_u^N(i)\right)^{p-2}\beta_H^NH_u^N(i)
  +\tfrac{1}{2}\delta^2 p(p-1)\left(Y_u^N(i)\right)^{p-2}\beta_P^NP_u^N(i)
  \,du
  +M_t^{N,n}.
\end{split} \end{equation}
Because $1\geq c_0 4$, $\tfrac{\lambda}{K}\geq c_0 2\eta$, and $\gamma\geq \delta c_0$
we get for all $N,n\in\N$ and all $t\in[0,\infty)$
that $\P$-a.s.
\begin{equation} \begin{split}
  &\sum_{i\in\D_n}\sigma_i\left(Y_t^N(i)\right)^p
  -\sum_{i\in\D_n}\sigma_i\left(Y_0^N(i)\right)^p
  \\
  &\quad
  \leq
  \sum_{i\in\D_n}\sigma_i\int_0^t
  p\left(Y_u^N(i)\right)^{p-1}
  \Bigg(2\eta\bar{\kappa}_{H}\sum_{j\in\D} m(i,j)H_u^N(j)
  +2\eta\lambda H_u^N(i)
  -c_0(2\eta H_u^N(i))^2
  \\
  &\qquad\quad
  -\left[\eta\delta+c_04\eta\delta\right] P_u^N(i)H_u^N(i)
  +2\eta\bar{\iota}_H\Bigg)
  \\
  &\qquad\quad
  +p\left(Y_u^N(i)\right)^{p-1}
  \Bigg(\delta\bar{\kappa}_{P}\sum_{j\in\D} m(i,j)P_u^N(j)
  +\lambda\delta P_u^N(i)
  -c_0(\delta P_u^N(i))^2
  +\eta\delta P_u^N(i)H_u^N(i)
  +\delta\bar{\iota}_P\Bigg)
  \\
  &\qquad\quad
  + p(p-1)\left(Y_u^N(i)\right)^{p-2}
  \Bigg(
  \left(2\eta\betab_H+\tfrac{1}{2}\delta\betab_P\right) 2\eta H_u^N(i)
  +\left(\tfrac{1}{2}\delta\betab_P+2\eta\betab_H\right) \delta P_u^N(i)
  \Bigg)
  \,du
  +M_t^{N,n}
  \\
  &\quad
  =
  \int_0^t
  \sum_{i\in\D_n}\sigma_ip\left(Y_u^N(i)\right)^{p-1}2\eta\bar{\kappa}_{H}\sum_{j\in\D} m(i,j)H_u^N(j)
  +\sum_{i\in\D_n}\sigma_ip\left(Y_u^N(i)\right)^{p-1}
  \\
  &\qquad\quad
  \Bigg(\lambda\left(Y_u^N(i)\right)
  -c_0\left(Y_u^N(i)\right)^2
  +(2\eta\bar{\iota}_H+\delta\bar{\iota}_P)\Bigg)
  +\sum_{i\in\D_n}\sigma_ip\left(Y_u^N(i)\right)^{p-1}
  \delta\bar{\kappa}_{P}\sum_{j\in\D} m(i,j)P_u^N(j)
  \\
  &\qquad\quad
  +\sum_{i\in\D_n}\sigma_ip(p-1)\left(2\eta\betab_H+\tfrac{1}{2}\delta\betab_P\right)\left(Y_u^N(i)\right)^{p-1}
  \,du
  +M_t^{N,n}.
\end{split} \end{equation}
Using Young's inequality and Lemma \ref{lem:estimate.term.sum^p}
we get for all $N,n\in\N$ and all $t\in[0,\infty)$
that $\P$-a.s.
\begin{equation} \begin{split} \label{eq:bound.(H+P)^p.with.stoch.integral}
  &\sum_{i\in\D_n}\sigma_i\left(Y_t^N(i)\right)^p
  -\sum_{i\in\D_n}\sigma_i\left(Y_0^N(i)\right)^p
  \\
  &\quad
  \leq
  \int_0^t
  \sum_{i\in\D_n}\sigma_i\tfrac{p-1}{p}p\left(Y_u^N(i)\right)^p\bar{\kappa}_{H}
  +\sum_{i\in\D}\sigma_i\tfrac{1}{p}p\bar{\kappa}_{H}c\left(2\eta H_u^N(i)\right)^p
  +\sum_{i\in\D_n}\sigma_i\lambda p\left(Y_u^N(i)\right)^p
  \\
  &\qquad
  +\sum_{i\in\D_n}\sigma_ic_1\left(Y_u^N(i)\right)^{p-1}
  -\sum_{i\in\D_n}\sigma_ic_0p\left(Y_u^N(i)\right)^{p+1}
  +\sum_{i\in\D_n}\sigma_i\tfrac{p-1}{p}p\left(Y_u^N(i)\right)^p\bar{\kappa}_{P}
  \\
  &\qquad
  +\sum_{i\in\D}\sigma_i\tfrac{1}{p}p\bar{\kappa}_{P}c\left(\delta P_u^N(i)\right)^p
  \,du
  +M_t^{N,n}
  \\
  &\quad
  \leq
  \int_0^t
  \sum_{i\in\D}\sigma_ic_2\left(Y_u^N(i)\right)^p
  +\sum_{i\in\D_n}\sigma_ic_1\left(Y_u^N(i)\right)^{p-1}
  -\sum_{i\in\D_n}\sigma_ic_0p
  \left(Y_u^N(i)\right)^{p+1}
  \,du
  +M_t^{N,n}.
\end{split} \end{equation}
For $N,n,l\in\N$ define $[0,\infty]$-valued stopping times
\begin{equation}
  \tau_l^{N,n}:=\inf\left(\left\{t\in[0,\infty):\sum_{i\in\D_n}\sigma_i\left(Y_t^N(i)\right)^p>l\right\}\cup\infty\right).
\end{equation}
We now get for all $N,n,l\in\N$ and all $t\in[0,\infty)$ that
\begin{equation} \begin{split}
  &\E\left[\int_0^{t\land\tau_l^{N,n}}\sum_{i\in\D_n}\sigma_i\left[\left(2\eta p\left(Y_u^N(i)\right)^{p-1}\sqrt{\beta_H^N H_u^N(i)}\right)^2
  +\left(\delta p\left(Y_u^N(i)\right)^{p-1}\sqrt{\beta_P^N P_u^N(i)}\right)^2
  \right]\,du\right]
  \\
  &\quad
  =
  \E\left[\int_0^{t\land\tau_l^{N,n}}\sum_{i\in\D_n}\sigma_i
  p^2\left[2\eta\beta_H^N\left(\left(Y_u^N(i)\right)^{p-1}\sqrt{2\eta H_u^N(i)}\right)^2
  +\delta\beta_P^N\left(\left(Y_u^N(i)\right)^{p-1}\sqrt{\delta P_u^N(i)}\right)^2
  \right]\,du\right]
  \\
  &\quad
  \leq
  \left(2\eta\beta_H^N+\delta\beta_P^N\right)
  \E\left[\int_0^{t\land\tau_l^{N,n}}\sum_{i\in\D_n}\sigma_i
  p^2\left[\left(2\left(Y_u^N(i)\right)^{\frac{2p-1}{2}}\right)^2
  \right]\,du\right]
\end{split} \end{equation}
Using Young's inequality, we obtain for all $N,n,l\in\N$ and all $t\in[0,\infty)$ that
\begin{equation} \begin{split}
  &\E\left[\int_0^{t\land\tau_l^{N,n}}\sum_{i\in\D_n}\sigma_i\left[\left(2\eta p\left(Y_u^N(i)\right)^{p-1}\sqrt{\beta_H^N H_u^N(i)}\right)^2
  +\left(\delta p\left(Y_u^N(i)\right)^{p-1}\sqrt{\beta_P^N P_u^N(i)}\right)^2
  \right]\,du\right]
  \\
  &\quad
  \leq
  \left(2\eta\beta_H^N+\delta\beta_P^N\right)
  \E\left[\int_0^{t\land\tau_l^{N,n}}\sum_{i\in\D_n}\sigma_i
  \left[4p^2\left(\tfrac{2p-1}{2p}\left(Y_u^N(i)\right)^p+\tfrac{1}{2p}\right)^2
  \right]\,du\right]
  \\
  &\quad
  \leq
  \left(2\eta\beta_H^N+\delta\beta_P^N\right)
  \E\left[\int_0^{t\land\tau_l^{N,n}}\sum_{i\in\D_n}\tfrac{\sigma_i^2}{\min\limits_{k\in\D_n}\sigma_k}
  \left[\left((2p-1)\left(Y_u^N(i)\right)^p+1\right)^2
  \right]\,du\right]
  \\
  &\quad
  \leq
  \tfrac{2\eta\beta_H^N+\delta\beta_P^N}{\min\limits_{k\in\D_n}\sigma_k}
  \E\left[\int_0^{t\land\tau_l^{N,n}}
  \left[\left(\sum_{i\in\D_n}\sigma_i\left((2p-1)\left(Y_u^N(i)\right)^p+1\right)\right)^2
  \right]\,du\right]
  \\
  &\quad
  \leq
  \tfrac{2\eta\beta_H^N+\delta\beta_P^N}{\min\limits_{k\in\D_n}\sigma_k}
  \E\left[\int_0^t
  \left[\left((2p-1)\sum_{i\in\D_n}\sigma_i\left(Y_{u\land\tau_l^{N,n}}^N(i)\right)^p+\|\underline{1}\|_\sigma\right)^2
  \right]\,du\right]
  \\
  &\quad
  \leq
  \tfrac{2\eta\beta_H^N+\delta\beta_P^N}{\min\limits_{k\in\D_n}\sigma_k}
  t\left[\left((2p-1)l+\|\underline{1}\|_\sigma\right)^2
  \right]
  <\infty.
\end{split} \end{equation}
Hence, we get for all $N,n,l\in\N$ and all $t\in[0,\infty)$ that
$\E\left[M_{t\land\tau_l^{N,n}}^{N,n}\right]=0$.
From this and \eqref{eq:bound.(H+P)^p.with.stoch.integral} and
using Tonelli's theorem we see 
for all $N,n,l\in\N$ and all $t\in[0,\infty)$ that
\begin{equation} \begin{split} \label{eq:bound.terms.Y}
  &\E\left[\sum_{i\in\D_n}\sigma_i\left(Y_{t\land\tau_l^{N,n}}^N(i)\right)^p
  +
  \int_0^{t\land\tau_l^{N,n}}
  c_0p\sum_{i\in\D_n}\sigma_i\left(Y_u^N(i)\right)^{p+1}
  \,du\right]
  \\
  &\quad
  \leq
  \E\left[\left\|\left(Y_0^N\right)^p\right\|_\sigma\right]
  +
  \E\left[\int_0^{t\land\tau_l^{N,n}}
  c_2\left\|\left(Y_u^N\right)^p\right\|_\sigma
  +c_1\left\|\left(Y_u^N\right)^{p-1}\right\|_\sigma
  \,du\right]
  \\
  &\quad
  \leq
  \E\left[\left\|\left(Y_0^N\right)^p\right\|_\sigma\right]
  +
  \int_0^t
  c_2\E\left[\left\|\left(Y_u^N\right)^p\right\|_\sigma\right]
  +c_1\E\left[\left\|\left(Y_u^N\right)^{p-1}\right\|_\sigma\right]
  \,du.
\end{split} \end{equation}
For every $N,n\in\N$ the map $[0,\infty)\ni t \mapsto \sum_{i\in\D_n}\sigma_i\left(Y_t^N(i)\right)^p\in\R$
is $\P$-a.s.~continuous which implies for all $N,n\in\N$ and all $t\in[0,\infty)$ that
$\P\left[\lim_{l\to\infty}\tau_l^{N,n}<t\right]=0$.
From Tonelli's theorem and monotone convergence, then using
Fatou's lemma, and finally applying \eqref{eq:bound.terms.Y}
we see for all $N\in\N$ and all $t\in[0,\infty)$ that
\begin{equation} \begin{split}
  &\E\left[\sum_{i\in\D}\sigma_i\left(Y_t^N(i)\right)^p\right]
  +
  \int_0^t
  c_0p\E\left[\sum_{i\in\D}\sigma_i\left(Y_u^N(i)\right)^{p+1}\right]
  \,du
  \\
  &\quad
  =
  \lim_{n\to\infty}
  \E\left[\sum_{i\in\D_n}\sigma_i\left(Y_t^N(i)\right)^p
  +
  \int_0^t
  c_0p\sum_{i\in\D_n}\sigma_i\left(Y_u^N(i)\right)^{p+1}
  \,du\right]
  \\
  &\quad
  =
  \lim_{n\to\infty}
  \E\left[\lim_{l\to\infty}\left(\sum_{i\in\D_n}\sigma_i\left(Y_{t\land\tau_l^{N,n}}^N(i)\right)^p
  +
  \int_0^{t\land\tau_l^{N,n}}
  c_0p\sum_{i\in\D_n}\sigma_i\left(Y_u^N(i)\right)^{p+1}
  \,du\right)\right]
  \\
  &\quad
  \leq
  \lim_{n\to\infty}
  \liminf\limits_{l\to\infty}
  \E\left[\sum_{i\in\D_n}\sigma_i\left(Y_{t\land\tau_l^{N,n}}^N(i)\right)^p
  +
  \int_0^{t\land\tau_l^{N,n}}
  c_0p\sum_{i\in\D_n}\sigma_i\left(Y_u^N(i)\right)^{p+1}
  \,du\right]
  \\
  &\quad
  \leq
  \E\left[\left\|\left(Y_0^N\right)^p\right\|_\sigma\right]
  +
  \int_0^t
  c_2\E\left[\left\|\left(Y_u^N\right)^p\right\|_\sigma\right]
  +c_1\E\left[\left\|\left(Y_u^N\right)^{p-1}\right\|_\sigma\right]
  \,du.
\end{split} \end{equation}
This implies using Jensen's inequality for all $N\in\N$ and all $t\in[0,\infty)$ that we get
\begin{equation} \begin{split}
  &\E\left[\left\|\left(Y_t^N\right)^p\right\|_\sigma\right]
  -
  \E\left[\left\|\left(Y_0^N\right)^p\right\|_\sigma\right]
  \\
  &\quad
  \leq
  \int_0^t
  c_2\E\left[\sum_{i\in\D}\sigma_i\left(Y_u^N(i)\right)^p\right]
  +c_1\E\left[\sum_{i\in\D}\sigma_i\left(Y_u^N(i)\right)^{p-1}\right]
  -c_0p\E\left[\sum_{i\in\D}\sigma_i\left(Y_u^N(i)\right)^{p+1}\right]
  \,du
  \\
  &\quad
  =
  \int_0^t
  c_2
  \E\left[\sum_{i\in\D}\sigma_i\left(Y_u^N(i)\right)^p\right]
  +
  \tfrac{\sum_{k\in\D}\sigma_k}{\sum_{l\in\D}\sigma_l}
  \left(
  c_1
  \E\left[\sum_{i\in\D}\sigma_i\left(Y_u^N(i)\right)^{p-1}\right]
  -c_0p
  \E\left[\sum_{i\in\D}\sigma_i\left(Y_u^N(i)\right)^{p+1}\right]
  \right)
  \,du
  \\
  &\quad
  \leq
  \int_0^t
  c_2
  \E\left[\sum_{i\in\D}\sigma_i\left(Y_u^N(i)\right)^p\right]
  +c_4
  \left(\E\left[\sum_{i\in\D}\sigma_i\left(Y_u^N(i)\right)^p\right]\right)^{\frac{p-1}{p}}
  -c_3
  \left(\E\left[\sum_{i\in\D}\sigma_i\left(Y_u^N(i)\right)^p\right]\right)^{\frac{p+1}{p}}
  \,du
  \\
  &\quad
  =
  \int_0^t
  \left(\E\left[\left\|\left(Y_u^N\right)^p\right\|_{\sigma}\right]\right)^{\frac{p-1}{p}}
  \left\{
  c_4
  +c_2
  \left(\E\left[\left\|\left(Y_u^N\right)^p\right\|_{\sigma}\right]\right)^{\frac{1}{p}}
  -c_3
  \left(\E\left[\left\|\left(Y_u^N\right)^p\right\|_{\sigma}\right]\right)^{\frac{2}{p}}
  \right\}
  \,du.
\end{split} \end{equation}
For every $N\in\N$ let $z^N\colon [0,\infty)\to\R$ be a process that for all $t\in[0,\infty)$ satisfies
\begin{equation} \begin{split}
  &z_t^N=
  z_0^N
  +
  \int_0^t\left(z_s^N\right)^{\frac{p-1}{p}}
  \left\{
  c_4
  +c_2 \left(z_s^N\right)^{\frac{1}{p}}
  -c_3\left(z_s^N\right)^{\frac{2}{p}}
  \right\}\,ds
\end{split} \end{equation}
with
$z_0^N=\E\left[\left\|\left(Y_0^N\right)^p\right\|_{\sigma}\right]$,
where uniqueness follows from local Lipschitz continuity.
Using classical comparison results from the theory of ODEs,
the above computation shows that for all $N\in\N$ and all $t\in[0,\infty)$ we have
$\E\left[\left\|\left(Y_t^N\right)^p\right\|_{\sigma}\right]\leq z_t^N$
and for all $N\in\N$ we have
$\sup_{t\in[0,\infty)}z_t^N=\max\Big\{\E\left[\left\|\left(Y_0^N\right)^p\right\|_{\sigma}\right]
,\Big(\tfrac{c_2}{2c_3}+\sqrt{\tfrac{(c_2)^2}{4c_3^2}+\tfrac{c_4}{c_3}}\Big)^p\Big\}$.
We thereby conclude that
\begin{equation} \begin{split}
  &\sup\limits_{N\in\N}\sup\limits_{t\in[0,\infty)}
  E\left[\left\|\left(2\eta H_t^N+\delta P_t^N\right)^p\right\|_{\sigma}\right]
  \leq
  \sup\limits_{N\in\N}\sup\limits_{t\in[0,\infty)}z_t^N
  \leq
  \sup\limits_{N\in\N}\E\left[\left\|\left(Y_0^N\right)^p\right\|_{\sigma}\right]
  +\left(\tfrac{c_2}{2c_3}+\sqrt{\tfrac{c_2^2}{4c_3^2}+\tfrac{c_4}{c_3}}\right)^p
  \\
  &\quad
  =
  \sup\limits_{N\in\N}\E\left[\left\|\left(Y_0^N\right)^p\right\|_{\sigma}\right]
  +\tfrac{c_2}{2c_3}\left(1+\sqrt{1+\tfrac{c_4c_3}{c_2^2}}\right)^p
  =
  \sup\limits_{N\in\N}\E\left[\left\|\left(2\eta H_0^N+\delta P_0^N\right)^p\right\|_{\sigma}\right]
  \\
  &\qquad\quad
  +\|\underline{1}\|_\sigma
  \left(\tfrac{\lambda+\left(1-\frac{1}{p}+\frac{c}{p}\right)\left(\bar{\kappa}_H+\bar{\kappa}_P\right)}
  {2\min\left\{\tfrac{1}{2\eta}\tfrac{\lambda}{K},\tfrac{1}{4},\tfrac{1}{\delta}\gamma\right\}}
  \right)^p
  \left(1+\sqrt{1+
  \tfrac{4\min\left\{\tfrac{1}{2\eta}\tfrac{\lambda}{K},\tfrac{1}{4},\tfrac{1}{\delta}\gamma\right\}\left[2\eta\bar{\iota}_H+
    \delta\bar{\iota}_P+(p-1)\left(2\eta\betab_H+\tfrac{1}{2}\delta\betab_P\right)\right]}
  {\left(\lambda+\left(1-\frac{1}{p}+\frac{c}{p}\right)\left(\bar{\kappa}_H+\bar{\kappa}_P\right)\right)^2}}\right)^p.
\end{split} \end{equation}
This completes the proof of Lemma \ref{lem:bound.norm.(H+P)^p}.
\end{proof}
\begin{proof}[Proof of Lemma \ref{lem:bound.norm.1/H^2}]
If the right-hand side of \eqref{eq:claim.lemma.1/H^2} is infinite, then the claim trivially holds.
For the remainder of the proof assume the right-hand side of \eqref{eq:claim.lemma.1/H^2} to be finite.
Define $\D_0:=\emptyset$ and 
for every $n\in\N$ let $\D_n\subseteq\hat{\D}$ be a set with $|\D_n|=\min\big\{n,|\hat{\D}|\big\}$ and $\D_n\supseteq\D_{n-1}$.
Define $c_1:=\tfrac{1}{\lambda+\nu}$ and for all $n\in\N$ let
\begin{equation}
  c_0^n:=
  \tfrac{2c_1\bar{\kappa}_{P}c}{3}
  \sup\limits_{N\in\N}\sup\limits_{t\in[0,\infty)}\E\left[\sum_{i\in\D_n}\sigma_i\left(P_t^N(i)\right)^3\right]
  +2c_1\Bigg[\tfrac{\eta^2}{\lambda}
  +\tfrac{4\lambda}{K^2}
  \Bigg]
  \sup\limits_{N\in\N}\sup\limits_{t\in[0,\infty)}\E\left[\sum_{i\in\D_n}\sigma_iP_t^N(i)\right]
  +\tfrac{\lambda}{K^2\delta}.
\end{equation}
For $N,n,l\in\N$ define $[0,\infty]$-valued stopping times
\begin{equation}
  \tau_l^{N,n}:=\inf\left(\left\{t\in[0,\infty):\sum_{i\in\D_n}\sigma_i\left(P_t^N(i)+\left(H_t^N(i)\right)^{-1}\right)>l\right\}\cup\infty\right).
\end{equation}
We infer from Lemma \ref{lem:H.positive}
that for all $N,n\in\N$ the map $[0,\infty)\ni t\mapsto \sum_{i\in\D_n}\sigma_i\left(P_t^N(i)+\left(H_t^N(i)\right)^{-1}\right)\in\R$ is $\P$-a.s.~continuous.
Thereby, we have for all $t\in[0,\infty)$ and all $N,n\in\N$ that
\begin{equation} \label{eq:lim.tau>t}
  \P\left[\lim_{l\to\infty}\tau_l^{N,n}<t\right]=0.
\end{equation}
For all $t\in[0,\infty)$, $N,n,l\in\N$ applying Young's inequality we get
\begin{equation} \begin{split}
  &\E\left[\sum_{i\in\D_n}\sigma_i
  \int_0^{t\land\tau_l^{N,n}}  
  \left(2c_1\tfrac{\sqrt{\beta_P^N P_u^N(i)}}{\left(H_u^N(i)\right)^2}\right)^2\,du
  \right]
  \\
  &\quad
  \leq
  \E\left[
  \sum_{i\in\D_n}\tfrac{\sigma_i^5}{\min\limits_{k\in\D_n}\{\sigma_k^4\}}
  t\sup_{u\in[0,t]}
  4c_1^2\betab_P
  \left(\tfrac{1}{5}\left(P_{u\land\tau_l^{N,n}}^N(i)\right)^5
  +\tfrac{4}{5}\left(H_{u\land\tau_l^{N,n}}^N(i)\right)^{-5}\right)
  \right]
  \\
  &\quad
  \leq
  \tfrac{t4c_1^2\betab_P}{\min\limits_{k\in\D_n}\{\sigma_k^4\}}
  \E\left[
  \sup_{u\in[0,t]}
  \left(\sum_{i\in\D_n}\sigma_i\left(P_{u\land\tau_l^{N,n}}^N(i)+\left(H_{u\land\tau_l^{N,n}}^N(i)\right)^{-1}\right)\right)^5
  \right]
  \leq
  \tfrac{t4c_1^2\betab_P}{\min\limits_{k\in\D_n}\{\sigma_k^4\}}
  l^5
  <\infty
\end{split} \end{equation}
and
\begin{equation} \begin{split}
  &\E\left[\sum_{i\in\D_n}\sigma_i
  \int_0^{t\land\tau_l^{N,n}}
  \left(\left(4c_1P_u^N(i)+\tfrac{1}{\delta}\right)
  \tfrac{\sqrt{\beta_H^N H_u^N(i)}}{\left(H_u^N(i)\right)^3}\right)^2\,du
  \right]
  \\
  &\quad
  \leq
  \E\left[\sum_{i\in\D_n}\tfrac{\sigma_i^7}{\min\limits_{k\in\D_n}\{\sigma_k^6\}}
  t\sup_{u\in[0,t]}
  (4c_1+\tfrac{1}{\delta})^2
  \betab_H\left(
  \tfrac{2}{7}\left(P_{u\land\tau_l^{N,n}}^N(i)+1\right)^7
  +\tfrac{5}{7}\left(H_{u\land\tau_l^{N,n}}^N(i)\right)^{-7}\right)
  \right]
  \\
  &\quad
  \leq
  \tfrac{t(4c_1+\frac{1}{\delta})^2\betab_H}{\min\limits_{k\in\D_n}\{\sigma_k^6\}}
  \E\left[
  \sup_{u\in[0,t]}
  \left(
  \sum_{i\in\D_n}\sigma_i
  \left(
  P_{u\land\tau_l^{N,n}}^N(i)+1
  +\left(H_{u\land\tau_l^{N,n}}^N(i)\right)^{-1}
  \right)\right)^7
  \right]
  \\
  &\quad
  \leq
  \tfrac{t(4c_1+\frac{1}{\delta})^2\betab_H}{\min\limits_{k\in\D_n}\{\sigma_k^6\}}
  \left(
  l+
  \left\|\underline{1}\right\|_\sigma
  \right)^7
  <\infty.
\end{split} \end{equation}
Hence, we obtain for all $t\in[0,\infty)$ and all $N,n,l\in\N$ that
\begin{equation} \begin{split} \label{eq:lem.1/H^2+P/H^2.E.stoch.int.0}
  &\E\left[\sum_{i\in\D_n}\sigma_i
  \int_0^{t\land\tau_l^{N,n}}
  2c_1\tfrac{\sqrt{\beta_P^N P_u^N(i)}}{\left(H_u^N(i)\right)^2}
  \,dW_u^{P,N}(i)\right]
  =0,
  \\
  &\E\left[\sum_{i\in\D_n}\sigma_i
  \int_0^{t\land\tau_l^{N,n}}
  \left(4c_1P_u^N(i)+\tfrac{1}{\delta}\right)
  \tfrac{\sqrt{\beta_H^N H_u^N(i)}}{\left(H_u^N(i)\right)^3}
  \,dW_u^{H,N}(i)\right]
  =0.
\end{split} \end{equation}
Define the function $y\colon\N\times\N\times[0,\infty)\to[0,\infty]$ by
\begin{equation}
  \N\times\N\times[0,\infty)\ni (N,n,t)\mapsto y_t^{N,n}:=\E\left[\sum_{i\in\D_n}\sigma_i\left(2c_1\tfrac{P_t^N(i)}{\left(H_t^N(i)\right)^2}
  +\tfrac{1}{2\delta}\tfrac{1}{\left(H_t^N(i)\right)^2}\right)\right].
\end{equation}
Recall from the beginning of the proof that we assume for all $N,n\in\N$ that $y_0^{N,n}<\infty$.
Now, applying It\^o's lemma and using \eqref{eq:lem.1/H^2+P/H^2.E.stoch.int.0}, we obtain
for all $t\in[0,\infty)$ and all $N,n,l\in\N$ that
\begin{equation} \begin{split}
  &\E\left[\sum_{i\in\D_n}\sigma_i\left(2c_1\tfrac{P_{t\land\tau_l^{N,n}}^N(i)}{(H_{t\land\tau_l^{N,n}}^N(i))^2}
  +\tfrac{1}{2\delta}\tfrac{1}{(H_{t\land\tau_l^{N,n}}^N(i))^2}\right)\right]
  -y_0^{N,n}
  \\
  &\quad
  =
  \E\Bigg[\sum_{i\in\D_n}\sigma_i\int_0^{t\land\tau_l^{N,n}}
  2c_1\tfrac{1}{\left(H_u^N(i)\right)^2}
  \Bigg(\kappa_P^N\sum_{j\in\D} m(i,j)P_u^N(j)
  -(\kappa_P^N+\nu)P_u^N(i)
  -\gamma \left(P_u^N(i)\right)^2
  \\
  &\qquad
  +\left(\eta-\rho F_u^N(i)\right)P_u^N(i)H_u^N(i)
  +\iota_P^N\Bigg)
  -\left(2c_1\tfrac{2P_u^N(i)}{\left(H_u^N(i)\right)^3}+\tfrac{1}{2\delta}\tfrac{2}{\left(H_u^N\right)^3}\right)
  \Bigg(\kappa_H^N\sum_{j\in\D} m(i,j)H_u^N(j)
  \\
  &\qquad
  +(-\kappa_H^N+\lambda-\alpha^NF_u^N(i))H_u^N(i)
  -\tfrac{\lambda}{K}\left(H_u^N(i)\right)^2
  -\delta P_u^N(i)H_u^N(i)
  +\iota_H^N\Bigg)
  \\
  &\qquad
  +\tfrac{1}{2}2c_1\tfrac{6P_u^N(i)}{\left(H_u^N(i)\right)^4}\beta_H^NH_u^N(i)
  +\tfrac{1}{2}\tfrac{1}{2\delta}\tfrac{6}{\left(H_u^N(i)\right)^4}\beta_H^NH_u^N(i)
  \,du\Bigg].
\end{split} \end{equation}
Dropping some negative terms, we now get
for all $t\in[0,\infty)$ and all $N,n,l\in\N$ that
\begin{equation} \begin{split}
    &\E\left[\sum_{i\in\D_n}\sigma_i\left(2c_1\tfrac{P_{t\land\tau_l^{N,n}}^N(i)}{(H_{t\land\tau_l^{N,n}}^N(i))^2}
    +\tfrac{1}{2\delta}\tfrac{1}{(H_{t\land\tau_l^{N,n}}^N(i))^2}\right)\right]
    -y_0^{N,n}
    \\
    &\quad
    \leq
    \E\Bigg[\sum_{i\in\D_n}\sigma_i\int_0^{t\land\tau_l^{N,n}}
    \tfrac{2c_1}{\left(H_u^N(i)\right)^2}
    \left(\kappa_P^N\sum_{j\in\D} m(i,j)P_u^N(j)
    -\nu P_u^N(i)
    -\gamma \left(P_u^N(i)\right)^2
    +\eta P_u^N(i)H_u^N(i)
    +\iota_P^N\right)
    \\
    &\qquad
    -\left(4c_1\tfrac{P_u^N(i)}{\left(H_u^N(i)\right)^3}+\tfrac{1}{\delta}\tfrac{1}{\left(H_u^N\right)^3}\right)
    \left(
    (-\kappa_H^N+\lambda-\alpha^N)H_u^N(i)
    -\tfrac{\lambda}{K}\left(H_u^N(i)\right)^2
    -\delta P_u^N(i)H_u^N(i)
    +\iota_H^N\right)
    \\
    &\qquad
    +6c_1\tfrac{P_u^N(i)}{\left(H_u^N(i)\right)^3}\beta_H^N+\tfrac{3}{2\delta}\tfrac{1}{\left(H_u^N(i)\right)^3}\beta_H^N
    \,du\Bigg]
    \\
    &\quad
    =
    \E\Bigg[\sum_{i\in\D_n}\sigma_i\int_0^{t\land\tau_l^{N,n}}
    2c_1\Bigg(\kappa_P^N\tfrac{1}{\left(H_u^N(i)\right)^2}\sum_{j\in\D} m(i,j)P_u^N(j)
    -\nu\tfrac{P_u^N(i)}{\left(H_u^N(i)\right)^2}
    -\gamma\tfrac{\left(P_u^N(i)\right)^2}{\left(H_u^N(i)\right)^2}
    +\eta\tfrac{P_u^N(i)}{H_u^N(i)}
    \\
    &\qquad
    +\iota_P^N\tfrac{1}{\left(H_u^N(i)\right)^2}
    -2\left(-\kappa_H^N+\lambda-\alpha^N\right)\tfrac{P_u^N(i)}{\left(H_u^N(i)\right)^2}
    +2\tfrac{\lambda}{K}\tfrac{P_u^N(i)}{H_u^N(i)}
    +2\delta\tfrac{\left(P_u^N(i)\right)^2}{\left(H_u^N(i)\right)^2}
    -2\iota_H^N\tfrac{P_u^N(i)}{\left(H_u^N(i)\right)^3}
    \\
    &\qquad
    +3\tfrac{P_u^N(i)}{\left(H_u^N(i)\right)^3}\beta_H^N
    \Bigg)
    +\tfrac{\kappa_H^N-\lambda+\alpha^N}{\delta}\tfrac{1}{\left(H_u^N\right)^2}
    +\tfrac{\lambda}{K\delta}\tfrac{1}{H_u^N}
    +\tfrac{P_u^N(i)}{\left(H_u^N\right)^2}
    -\tfrac{\iota_H^N}{\delta}\tfrac{1}{\left(H_u^N\right)^3}
    +\tfrac{3 \beta_H^N}{2\delta}\tfrac{1}{\left(H_u^N(i)\right)^3}
    \,du\Bigg].
  \end{split} \end{equation}
Using Young's inequality as well as Lemma \ref{lem:estimate.term.sum^p}
we get for all $t\in[0,\infty)$ and all $N,n,l\in\N$ that
\begin{equation} \begin{split}
  &\E\left[\sum_{i\in\D_n}\sigma_i\left(2c_1\tfrac{P_{t\land\tau_l^{N,n}}^N(i)}{(H_{t\land\tau_l^{N,n}}^N(i))^2}
  +\tfrac{1}{2\delta}\tfrac{1}{(H_{t\land\tau_l^{N,n}}^N(i))^2}\right)\right]-y_0^{N,n}
  \\
  &\quad
  \leq
  \E\Bigg[\sum_{i\in\D_n}\sigma_i\int_0^{t\land\tau_l^{N,n}}
  2c_1
  \Bigg(
  \tfrac{2}{3}\kappa_P^N\tfrac{1}{\left(H_u^N(i)\right)^3}+\tfrac{1}{3}\kappa_P^Nc\left(P_u^N(i)\right)^3
  -\nu\tfrac{P_u^N(i)}{\left(H_u^N(i)\right)^2}
  -\gamma\tfrac{\left(P_u^N(i)\right)^2}{\left(H_u^N(i)\right)^2}
  \\
  &\qquad
  +\tfrac{1}{2}\tfrac{\lambda}{2\eta}\eta\tfrac{P_u^N(i)}{\left(H_u^N(i)\right)^2}
  +\tfrac{1}{2}\tfrac{2\eta}{\lambda}\eta P_u^N(i)
  +\iota_P^N\tfrac{1}{\left(H_u^N(i)\right)^2}
  -2(-\kappa_H^N+\lambda-\alpha^N)\tfrac{P_u^N(i)}{\left(H_u^N(i)\right)^2}
  +\tfrac{1}{2}\tfrac{K}{4}2\tfrac{\lambda}{K}\tfrac{P_u^N(i)}{\left(H_u^N(i)\right)^2}
  \\
  &\qquad
  +\tfrac{1}{2}\tfrac{4}{K}2\tfrac{\lambda}{K}P_u^N(i)
  +2\delta\tfrac{\left(P_u^N(i)\right)^2}{\left(H_u^N(i)\right)^2}
  -2\iota_H^N\tfrac{P_u^N(i)}{\left(H_u^N(i)\right)^3}
  +3\tfrac{P_u^N(i)}{\left(H_u^N(i)\right)^3}\beta_H^N
  \Bigg)
  +\tfrac{\kappa_H^N-\lambda+\alpha^N}{\delta}\tfrac{1}{\left(H_u^N\right)^2}
  \\
  &\qquad
  +\tfrac{1}{2}\tfrac{K}{2}\tfrac{\lambda}{K\delta}\tfrac{1}{\left(H_u^N\right)^2}+\tfrac{1}{2}\tfrac{2}{K}\tfrac{\lambda}{K\delta}
  +\tfrac{P_u^N(i)}{\left(H_u^N\right)^2}
  -\tfrac{\iota_H^N}{\delta}\tfrac{1}{\left(H_u^N\right)^3}
  +\tfrac{3\beta_H^N}{2\delta}\tfrac{1}{\left((H_u^N(i)\right)^3}
  \,du\Bigg]
  \\
  &\quad
  =
  \E\Bigg[\sum_{i\in\D_n}\sigma_i\int_0^{t\land\tau_l^{N,n}}
  \left[\tfrac{4c_1}{3}\kappa_P^N
  -\tfrac{1}{\delta}\iota_H^N
  +\tfrac{3}{2\delta}\beta_H^N
  \right]
  \tfrac{1}{\left(H_u^N(i)\right)^3}
  +\tfrac{2c_1\kappa_P^Nc}{3}\left(P_u^N(i)\right)^3
  +\Big[c_1\Big(
  -2\nu
  +\tfrac{\lambda}{2}
  \\
  &\qquad
  +4(\kappa_H^N-\lambda+\alpha^N)
  +\tfrac{\lambda}{2}
  \Big)
  +1
  \Big]
  \tfrac{P_u^N(i)}{\left(H_u^N(i)\right)^2}
  +2c_1
  \left[-\gamma+2\delta\right]
  \tfrac{\left(P_u^N(i)\right)^2}{\left(H_u^N(i)\right)^2}
  +2c_1\left[\tfrac{\eta^2}{\lambda}
  +\tfrac{4\lambda}{K^2}
  \right]
  P_u^N(i)
  \\
  &\qquad
  +\left[2c_1\iota_P^N
  +\tfrac{\kappa_H^N-\lambda+\alpha^N}{\delta}
  +\tfrac{\lambda}{4\delta}
  \right]
  \tfrac{1}{\left(H_u^N(i)\right)^2}
  +2c_1\left[
  -2\iota_H^N
  +3\beta_H^N
  \right]
  \tfrac{P_u^N(i)}{\left(H_u^N(i)\right)^3}
  +\tfrac{\lambda}{K^2\delta}
  \,du\Bigg].
\end{split} \end{equation}
Recall $\bar{\kappa}_P=\sup_{N\in\N}\kappa_P^N$
and that for all $N\in\N$ we have
$\alpha^N+\kappa_H^N\leq\tfrac{\lambda}{4}$,
$\iota_P^N\leq\tfrac{\lambda(\nu+\lambda)}{8\delta}$,
and $\iota_H^N\geq \tfrac{4\delta\kappa_P^N}{3(\nu+\lambda)}+\tfrac{3\beta_H^N}{2}$.
Furthermore, note that $\tfrac{\lambda}{2}\leq\tfrac{1}{2c_1}$.
Together with the assumption that $\gamma\geq 2\delta$
we see for all $t\in[0,\infty)$ and all $N,n,l\in\N$ that
\begin{equation} \begin{split}
  &\E\left[\sum_{i\in\D_n}\sigma_i\left(2c_1\tfrac{P_{t\land\tau_l^{N,n}}^N(i)}{(H_{t\land\tau_l^{N,n}}^N(i))^2}
  +\tfrac{1}{2\delta}\tfrac{1}{(H_{t\land\tau_l^{N,n}}^N(i))^2}\right)\right]
  -y_0^{N,n}
  \\
  &\quad
  \leq
  \E\Bigg[\sum_{i\in\D_n}\sigma_i\int_0^{t\land\tau_l^{N,n}}
  \tfrac{2c_1}{3}\bar{\kappa}_{P}c\left(P_u^N(i)\right)^3
  -\tfrac{P_u^N(i)}{\left(H_u^N(i)\right)^2}
  +2c_1\left[\tfrac{\eta^2}{\lambda}
  +\tfrac{4\lambda}{K^2}
  \right]
  P_u^N(i)
  -\tfrac{\lambda}{4\delta}\tfrac{1}{\left(H_u^N(i)\right)^2}
  +\tfrac{\lambda}{K^2\delta}
  du\Bigg]
  \\
  &\quad
  \leq
  \int_0^t
  c_0^n
  \,du
  -
  \E\left[
  \sum_{i\in\D_n}\sigma_i
  \int_0^{t\land\tau_l^{N,n}}
  \tfrac{\lambda}{2}
  \left(2c_1\tfrac{P_u^N(i)}{\left(H_u^N(i)\right)^2}
  +\tfrac{1}{2\delta}\tfrac{1}{\left(H_u^N(i)\right)^2}\right)
  \,du\right].
\end{split} \end{equation}
Using Tonelli's theorem, Fatou's lemma, and \eqref{eq:lim.tau>t} this implies for all $t\in[0,\infty)$ and all $N,n\in\N$ that
\begin{equation} \begin{split}
  &y_t^{N,n}
  +
  \int_0^t
  \tfrac{\lambda}{2}
  y_u^{N,n}
  \,du
  =
  y_t^{N,n}
  +\E\left[
  \sum_{i\in\D_n}\sigma_i
  \int_0^t
  \tfrac{\lambda}{2}
  \left(2c_1\tfrac{P_u^N(i)}{\left(H_u^N(i)\right)^2}
  +\tfrac{1}{2\delta}\tfrac{1}{\left(H_u^N(i)\right)^2}\right)
  \,du\right]
  \\
  &\quad
  \leq
  \liminf\limits_{l\to\infty}
  \left(
  \E\left[\sum_{i\in\D_n}\sigma_i\left(2c_1\tfrac{P_{t\land\tau_l^{N,n}}^N(i)}{(H_{t\land\tau_l^{N,n}}^N(i))^2}
  +\tfrac{1}{2\delta}\tfrac{1}{(H_{t\land\tau_l^{N,n}}^N(i))^2}\right)\right]
  \right.
  \\
  &\qquad\qquad\qquad\qquad
  \left.
  +\E\left[
  \sum_{i\in\D_n}\sigma_i
  \int_0^{t\land\tau_l^{N,n}}
  \tfrac{\lambda}{2}
  \left(2c_1\tfrac{P_u^N(i)}{\left(H_u^N(i)\right)^2}
  +\tfrac{1}{2\delta}\tfrac{1}{\left(H_u^N(i)\right)^2}\right)
  \,du\right]
  \right)
  \leq
  y_0^{N,n}
  +\int_0^tc_0^n\,du.
\end{split} \end{equation}
For every $N,n\in\N$ let $z^{N,n}\colon[0,\infty)\to\R$ be a process
that for all $t\in[0,\infty)$ satisfies
$z_t^{N,n}=z_0^{N,n}+\int_0^t\big(c_0^n-\tfrac{\lambda}{2}z_s^{N,n}\big)\,ds$
with $z_0^{N,n}=y_0^{N,n}$, where uniqueness follows from local Lipschitz continuity.
Due to classical comparison results of the theory of ODEs,
the above computation yields for all $t\in[0,\infty)$ and all $N,n\in\N$ that
$y_t^{N,n}\leq z_t^{N,n}$
and for all $N,n\in\N$ that
$\sup_{t\in[0,\infty)}z_t^{N,n}=\max\left\{z_0^{N,n},\tfrac{2c_0^n}{\lambda}\right\}$.
We obtain for all $n\in\N$ that
\begin{equation} \begin{split}
  &
  \sup\limits_{N\in\N}\sup\limits_{t\in[0,\infty)}y_t^{N,n}
  \leq
  \sup\limits_{N\in\N}\sup\limits_{t\in[0,\infty)}z_t^{N,n}
  =
  \max\left\{\sup\limits_{N\in\N}z_0^{N,n},\tfrac{2c_0^n}{\lambda}\right\}
  \\
  &\qquad
  \leq
  \sup\limits_{N\in\N}\E\left[\sum_{i\in\D_n}\sigma_i\left(2c_1\tfrac{P_0^N}{\left(H_0^N\right)^2}
  +\tfrac{1}{2\delta}\tfrac{1}{\left(H_0^N\right)^2}\right)\right]
  +\tfrac{2c_0^n}{\lambda}.
\end{split} \end{equation}
Using monotone convergence we thereby conclude
\begin{equation} \begin{split}
  &\sup\limits_{N\in\N}\sup\limits_{t\in[0,\infty)}
  \E\left[\sum_{i\in\hat{\D}}\sigma_i\left(
  \tfrac{2}{\lambda+\nu}\tfrac{P_t^N(i)}{\left(H_t^N(i)\right)^2}+\tfrac{1}{2\delta}\tfrac{1}{\left(H_t^N(i)\right)^2}\right)\right]
  \leq
  \lim_{n\to\infty}
  \sup\limits_{N\in\N}\sup\limits_{t\in[0,\infty)}y_t^{N,n}
  \\
  &\quad
  \leq
  \lim_{n\to\infty}
  \left(
  \sup\limits_{N\in\N}\E\left[\sum_{i\in\D_n}\sigma_i\left(2c_1\tfrac{P_0^N}{\left(H_0^N\right)^2}
  +\tfrac{1}{2\delta}\tfrac{1}{\left(H_0^N\right)^2}\right)\right]
  +\tfrac{2c_0^n}{\lambda}
  \right)
  \\
  &\quad
  \leq
  \sup\limits_{N\in\N}\E\left[\sum_{i\in\hat{\D}}\sigma_i\left(\tfrac{2}{\lambda+\nu}\tfrac{P_0^N}{\left(H_0^N\right)^2}
  +\tfrac{1}{2\delta}\tfrac{1}{\left(H_0^N\right)^2}\right)\right]
  +\tfrac{4\bar{\kappa}_{P}c}{3\lambda(\lambda+\nu)}
  \sup\limits_{N\in\N}\sup\limits_{t\in[0,\infty)}\E\left[\sum_{i\in\hat{\D}}\sigma_i\left(P_t^N(i)\right)^3\right]
  \\
  &\qquad\quad
  +\tfrac{4}{\lambda(\lambda+\nu)}
  \left(\tfrac{\eta^2}{\lambda}
  +\tfrac{4\lambda}{K^2}
  \right)
  \sup\limits_{N\in\N}\sup\limits_{t\in[0,\infty)}\E\left[\sum_{i\in\hat{\D}}\sigma_iP_t^N(i)\right]
  +\tfrac{2}{K^2\delta},
\end{split} \end{equation}
completes the proof of Lemma \ref{lem:bound.norm.1/H^2}.
\end{proof}
\begin{proof}[Proof of Lemma \ref{lem:bound.norm.1/P}]
If the right-hand side of \eqref{eq:claim.lemma.1/P} is infinite, then the claim trivially holds.
For the remainder of the proof assume the right-hand side of \eqref{eq:claim.lemma.1/P} to be finite.
Define $\D_0:=\emptyset$ and 
for every $n\in\N$ let $\D_n\subseteq\hat{\D}$ be a set with $|\D_n|=\min\left\{n,|\hat{\D}|\right\}$ and $\D_n\supseteq\D_{n-1}$.
Define
$c_0:=\tfrac{1}{2(\bar{\kappa}_P+\nu)}\left[(\eta-\rho )-\tfrac{\lambda}{K}\right]$
and for every $n\in\N$ let
\begin{equation}
  C^n:=
  \gamma c_0
  +\Bigg[\gamma+\delta\Bigg]
  \sup\limits_{N\in\N}\sup\limits_{t\in[0,\infty)}
  \E\left[\sum_{i\in\D_n}\sigma_i\tfrac{1}{H_t^N(i)}\right].
\end{equation}
Note that due to the assumption $\eta-\rho>\tfrac{\lambda}{K}$ we have $c_0\in(0,\infty)$.
For all $N,n,l\in\N$ define $[0,\infty]$-valued stopping times
\begin{equation}
  \tau_l^{N,n}:=\inf\left(\left\{t\in[0,\infty):\sum_{i\in\D_n}\sigma_i
  \left(\left(P_t^N(i)\right)^{-1}+\left(H_t^N(i)\right)^{-1}\right)>l\right\}\cup\infty\right).
\end{equation}
We infer from Lemmas \ref{lem:H.positive} and \ref{lem:P.positive}
that for all $N,n\in\N$ the map
$[0,\infty)\ni t\mapsto \sum_{i\in\D_n}\sigma_i\left(\left(P_t^N(i)\right)^{-1}+\left(H_t^N(i)\right)^{-1}\right)\in\R$
is $\P$-a.s.~continuous
which implies that we have for all $t\in[0,\infty)$ and all $N,n\in\N$ that
\begin{equation} \label{eq:lem.1/P.lim.tau>t}
  \P\left[\lim_{l\to\infty}\tau_l^{N,n}<t\right]=0.
\end{equation}
For all $t\in[0,\infty)$, $N,n,l\in\N$ applying Young's inequality
we see that
\begin{equation} \begin{split}
  &\E\left[\sum_{i\in\D_n}\sigma_i
  \int_0^{t\land\tau_l^{N,n}}
  \left(\tfrac{\sqrt{\beta_P^N P_u^N(i)}}{\left(P_u^N(i)\right)^2}
  \left(c_0+\tfrac{1}{H_u^N(i)}\right)\right)^2
  \,du
  \right]
  \\
  &\quad
  \leq
  \betab_P
  \E\left[
  t\sup_{u\in[0,t]}
  \sum_{i\in\D_n}\tfrac{\sigma_i^5}{\min\limits_{k\in\D_n}\left\{\sigma_k^4\right\}}
  \left(
  \tfrac{3}{5}\left(P_{u\land\tau_l^{N,n}}^N(i)\right)^{-5}
  +\tfrac{2}{5}\left(c_0+\left(H_{u\land\tau_l^{N,n}}^N(i)\right)^{-1}\right)^5
  \right)
  \,du
  \right]
  \\
  &\quad
  \leq
  \tfrac{\betab_P}{\min\limits_{k\in\D_n}\left\{\sigma_k^4\right\}}
  \E\left[t\sup_{u\in[0,t]}
  \left(
  \sum_{i\in\D_n}\sigma_i
  \left(
  \left(P_{u\land\tau_l^{N,n}}^N(i)\right)^{-1}
  +\left(H_{u\land\tau_l^{N,n}}^N(i)\right)^{-1}
  +c_0
  \right)
  \right)^5
  \right]
  \leq
  \tfrac{\betab_Pt\left(
  l+c_0\|\underline{1}\|_\sigma
  \right)^5}{\min\limits_{k\in\D_n}\left\{\sigma_k^4\right\}}
  <\infty
\end{split} \end{equation}
and
\begin{equation} \begin{split}
  &
  \E\left[\sum_{i\in\D_n}\sigma_i
  \int_0^{t\land\tau_l^{N,n}}
  \left(\tfrac{\sqrt{\beta_H^N H_u^N(i)}}{\left(H_u^N(i)\right)^2P_u^N(i)}\right)^2
  \,du\right]
  \\
  &\quad
  \leq
  \betab_H
  \E\left[
  t\sup_{u\in[0,t]}
  \sum_{i\in\D_n}\tfrac{\sigma_i^5}{\min\limits_{k\in\D_n}\left\{\sigma_k^4\right\}}
  \left(\tfrac{3}{5}\left(H_{u\land\tau_l^{N,n}}^N(i)\right)^{-5}+\tfrac{2}{5}\left(P_{u\land\tau_l^{N,n}}^N(i)\right)^{-5}\right)\,du\right]
  \\
  &\quad
  \leq
  t\sup_{u\in[0,t]}
  \tfrac{\betab_H}{\min\limits_{k\in\D_n}\left\{\sigma_k^4\right\}}
  \E\left[
  \left(
  \sum_{i\in\D_n}\sigma_i
  \left(\left(H_{u\land\tau_l^{N,n}}^N(i)\right)^{-1}+\left(P_{u\land\tau_l^{N,n}}^N(i)\right)^{-1}\right)\right)^5\right]
  \leq
  \tfrac{t\betab_Hl^5}{\min\limits_{k\in\D_n}\left\{\sigma_k^4\right\}}
  <\infty.
\end{split} \end{equation}
Hence, we obtain for all $t\in[0,\infty)$ and all $N,n,l\in\N$ that
\begin{equation}  \begin{split} \label{eq:lem.1/P+1/PH.E.stoch.int.0}
  &\E\left[\int_0^{t\land\tau_l^{N,n}}
  \sum_{i\in\D_n}\sigma_i\sqrt{\beta_P^N P_t^N(i)}\tfrac{1}{\left(P_t^N(i)\right)^2}
  \left(c_0+\tfrac{1}{H_t^N(i)}\right)
  \,dW_u^{P,N}(i)\right]
  =0,
  \\
  &\E\left[\int_0^{t\land\tau_l^{N,n}}
  \sum_{i\in\D_n}\sigma_i\tfrac{\sqrt{\beta_H^N H_t^N(i)}}{\left(H_t^N(i)\right)^2P_t^N(i)}
  \,dW_u^{H,N}(i)\right]
  =0.
\end{split} \end{equation}
Define the function $y\colon\N\times\N\times[0,\infty)\to[0,\infty]$ by
\begin{equation}
  \N\times\N\times[0,\infty)\ni (N,n,t)\mapsto y_t^{N,n}:=\E\left[\sum_{i\in\D_n}\sigma_i\left(c_0\tfrac{1}{P_t^N(i)}+\tfrac{1}{P_t^N(i)H_t^N(i)}\right)\right].
\end{equation}
Recall from the beginning of the proof that we assume for all $N,n\in\N$ that $y_0^{N,n}<\infty$.
Applying It\^o's lemma and using \eqref{eq:lem.1/P+1/PH.E.stoch.int.0},
we get for all $t\in[0,\infty)$ and all $N,n,l\in\N$ that
\begin{equation} \begin{split}
  &\E\left[\sum_{i\in\D_n}\sigma_i\left(c_0\tfrac{1}{P_{t\land\tau_l^{N,n}}^N(i)}+\tfrac{1}{P_{t\land\tau_l^{N,n}}^N(i)H_{t\land\tau_l^{N,n}}^N(i)}\right)\right]
  -y_0^{N,n}
  \\
  &\quad
  =
  \E\Bigg[\sum_{i\in\D_n}\sigma_i\int_0^{t\land\tau_l^{N,n}}
  -\left(c_0\tfrac{1}{\left(P_u^N(i)\right)^2}+\tfrac{1}{\left(P_u^N(i)\right)^2H_u^N(i)}
  \right)
  \Bigg(\kappa_P^N\sum_{j\in\D} m(i,j)P_u^N(j)
  \\
  &\qquad
  -(\kappa_P^N+\nu)P_u^N(i)
  -\gamma \left(P_u^N(i)\right)^2
  +\left(\eta-\rho F_u^N(i)\right)P_u^N(i)H_u^N(i)
  +\iota_P^N\Bigg)
  +\tfrac{1}{2}c_0\tfrac{2}{\left(P_u^N(i)\right)^3}\beta_P^NP_u^N(i)
  \\
  &\qquad
  +\tfrac{1}{2}\tfrac{2}{\left(P_u^N(i)\right)^3H_u^N(i)}\beta_P^NP_u^N(i)
  -\tfrac{1}{P_u^N(i)\left(H_u^N(i)\right)^2}
  \Bigg(\kappa_H^N\sum_{j\in\D} m(i,j)H_u^N(j)
  +(-\kappa_H^N+\lambda-\alpha^NF_u^N(i))H_u^N(i)
  \\
  &\qquad
  -\tfrac{\lambda}{K}\left(H_u^N(i)\right)^2
  -\delta P_u^N(i)H_u^N(i)
  +\iota_H^N\Bigg)
  +\tfrac{1}{2}\tfrac{2}{P_u^N(i)\left(H_u^N(i)\right)^3}\beta_H^NH_u^N(i)
  \,du\Bigg].
\end{split} \end{equation}
Dropping some negative terms, we now get for all $t\in[0,\infty)$ and all $N,n,l\in\N$ that
\begin{equation} \begin{split}
    &\E\left[\sum_{i\in\D_n}\sigma_i\left(c_0\tfrac{1}{P_{t\land\tau_l^{N,n}}^N(i)}+\tfrac{1}{P_{t\land\tau_l^{N,n}}^N(i)H_{t\land\tau_l^{N,n}}^N(i)}\right)\right]
    -y_0^{N,n}
    \\
    &\leq
    \E\Bigg[\sum_{i\in\D_n}\sigma_i\int_0^{t\land\tau_l^{N,n}}
    -\left(c_0\tfrac{1}{\left(P_u^N(i)\right)^2}+\tfrac{1}{\left(P_u^N(i)\right)^2H_u^N(i)}
    \right)
    \Bigg(
    -(\kappa_P^N+\nu)P_u^N(i)
    -\gamma \left(P_u^N(i)\right)^2
    \\
    &\qquad
    +\left(\eta-\rho\right)P_u^N(i)H_u^N(i)
    +\iota_P^N\Bigg)
    +c_0\beta_P^N\tfrac{1}{\left(P_u^N(i)\right)^2}
    +\beta_P^N\tfrac{1}{\left(P_u^N(i)\right)^2H_u^N(i)}
    +\beta_H^N\tfrac{1}{P_u^N(i)\left(H_u^N(i)\right)^2}
    \\
    &\qquad
    -\tfrac{1}{P_u^N(i)\left(H_u^N(i)\right)^2}
    \left(
    (-\kappa_H^N+\lambda-\alpha^N)H_u^N(i)
    -\tfrac{\lambda}{K}\left(H_u^N(i)\right)^2
    -\delta P_u^N(i)H_u^N(i)
    +\iota_H^N\right)
    \,du\Bigg]
    \\
    &=
    \E\Bigg[\sum_{i\in\D_n}\sigma_i\int_0^{t\land\tau_l^{N,n}}
    \left(\kappa_P^N+\nu\right)c_0\tfrac{1}{P_u^N(i)}
    +\gamma c_0
    -(\eta-\rho)c_0\tfrac{H_u^N(i)}{P_u^N(i)}
    -\iota_P^Nc_0\tfrac{1}{\left(P_u^N(i)\right)^2}
    +(\kappa_P^N+\nu)\tfrac{1}{P_u^N(i)H_u^N(i)}
    \\
    &\quad
    +\gamma \tfrac{1}{H_u^N(i)}
    -\left(\eta-\rho\right)\tfrac{1}{P_u^N(i)}
    -\iota_P^N\tfrac{1}{\left(P_u^N(i)\right)^2H_u^N(i)}
    +c_0\beta_P^N\tfrac{1}{\left(P_u^N(i)\right)^2}
    +\beta_P^N\tfrac{1}{\left(P_u^N(i)\right)^2H_u^N(i)}
    +\beta_H^N\tfrac{1}{P_u^N(i)\left(H_u^N(i)\right)^2}
    \\
    &\quad
    -(-\kappa_H^N+\lambda-\alpha^N)\tfrac{1}{P_u^N(i)H_u^N(i)}
    +\tfrac{\lambda}{K}\tfrac{1}{P_u^N(i)}
    +\delta \tfrac{1}{H_u^N(i)}
    -\iota_H^N\tfrac{1}{P_u^N(i)\left(H_u^N(i)\right)^2}
    \,du\Bigg]
    \\
    &=
    \E\Bigg[\sum_{i\in\D_n}\sigma_i\int_0^{t\land\tau_l^{N,n}}
    \left[\left(\kappa_P^N+\nu\right)c_0
    -\left(\eta-\rho\right)
    +\tfrac{\lambda}{K}
    \right]
    \tfrac{1}{P_u^N(i)}
    +\gamma c_0
    -(\eta-\rho)c_0
    \tfrac{H_u^N(i)}{P_u^N(i)}
    \\
    &\quad
    +\left[-\iota_P^Nc_0
    +c_0\beta_P^N
    \right]
    \tfrac{1}{\left(P_u^N(i)\right)^2}
    +\left[(\kappa_P^N+\nu)
    -(-\kappa_H^N+\lambda-\alpha^N)
    \right]
    \tfrac{1}{P_u^N(i)H_u^N(i)}
    +\left[\gamma
    +\delta
    \right]
    \tfrac{1}{H_u^N(i)}
    \\
    &\quad
    +\left[-\iota_P^N
    +\beta_P^N
    \right]
    \tfrac{1}{\left(P_u^N(i)\right)^2H_u^N(i)}
    +\left[\beta_H^N
    -\iota_H^N
    \right]
    \tfrac{1}{P_u^N(i)\left(H_u^N(i)\right)^2}
    \,du\Bigg].
  \end{split} \end{equation}
Recall from Section \ref{sec:convF_setting} that $\bar{\kappa}_P=\sup_{N\in\N}\kappa_P^N$,
and from Assumption \ref{ass:conv_LotkaVolterra} that $\lambda>\nu$, $\eta-\rho >\tfrac{\lambda}{K}$
and that for all $N\in\N$ we have
$\kappa_P^N+\kappa_H^N+\alpha^N\leq\tfrac{\lambda-\nu}{2}$,
$\iota_P^N\geq\beta_P^N$, and
$\iota_H^N\geq\beta_H^N$.
Hence, we get for all $t\in[0,\infty)$ and all $N,n,l\in\N$ that
\begin{equation} \begin{split}
  &\E\left[\sum_{i\in\D_n}\sigma_i\left(c_0\tfrac{1}{P_{t\land\tau_l^{N,n}}^N(i)}+\tfrac{1}{P_{t\land\tau_l^{N,n}}^N(i)H_{t\land\tau_l^{N,n}}^N(i)}\right)\right]
  -y_0^{N,n}
  \\
  &\quad
  \leq
  \E\Bigg[\sum_{i\in\D_n}\sigma_i\int_0^{t\land\tau_l^{N,n}}
  -\left(\bar{\kappa}_P+\nu\right)c_0
  \tfrac{1}{P_u^N(i)}
  +\gamma c_0
  -\tfrac{\lambda-\nu}{2}
  \tfrac{1}{P_u^N(i)H_u^N(i)}
  +\left[\gamma
  +\delta
  \right]
  \tfrac{1}{H_u^N(i)}
  \,du\Bigg]
  \\
  &\quad
  \leq
  \int_0^t
  C^n
  \,du
  -\E\left[
  \int_0^{t\land\tau_l^{N,n}}
  \min\left\{\bar{\kappa}_P+\nu,\tfrac{\lambda-\nu}{2}\right\}
  \sum_{i\in\D_n}\sigma_i\left(c_0\tfrac{1}{P_u^N(i)}+\tfrac{1}{P_u^N(i)H_u^N(i)}\right)\,du\right].
\end{split} \end{equation}
Applying Tonelli's theorem, Fatou's lemma, and \eqref{eq:lem.1/P.lim.tau>t} we obtain for all
$t\in[0,\infty)$ and all $N,n\in\N$ that
\begin{equation} \begin{split}
  &y_t^{N,n}+
  \int_0^t
  \min\left\{\bar{\kappa}_P+\nu,\tfrac{\lambda-\nu}{2}\right\}
  y_u^{N,n}\,du
  \\
  &\quad
  =
  y_t^{N,n}+\E\left[
  \int_0^t
  \min\left\{\bar{\kappa}_P+\nu,\tfrac{\lambda-\nu}{2}\right\}
  \sum_{i\in\D_n}\sigma_i\left(c_0\tfrac{1}{P_u^N(i)}+\tfrac{1}{P_u^N(i)H_u^N(i)}\right)\,du\right]
  \\
  &\quad
  \leq
  \liminf\limits_{l\to\infty}
  \Bigg(
  \E\left[\sum_{i\in\D_n}\sigma_i\left(c_0\tfrac{1}{P_{t\land\tau_l^{N,n}}^N(i)}+\tfrac{1}{P_{t\land\tau_l^{N,n}}^N(i)H_{t\land\tau_l^{N,n}}^N(i)}\right)\right]
  \\
  &\qquad
  +\E\left[
  \int_0^{t\land\tau_l^{N,n}}
  \min\left\{\bar{\kappa}_P+\nu,\tfrac{\lambda-\nu}{2}\right\}
  \sum_{i\in\D_n}\sigma_i\left(c_0\tfrac{1}{P_u^N(i)}+\tfrac{1}{P_u^N(i)H_u^N(i)}\right)\,du\right]
  \Bigg)
  \leq
  y_0^{N,n}
  +\int_0^tC^n\,du.
\end{split} \end{equation}
For every $N,n\in\N$,
let $z^{N,n}\colon[0,\infty)\to\R$ be a process that for all $t\in[0,\infty)$ satisfies
$z_t^{N,n}=z_0^{N,n}+\int_0^t\big(C^n-\min\left\{\bar{\kappa}_P+\nu,\tfrac{\lambda-\nu}{2}\right\}z_s^{N,n}\big)\,ds$,
with $z_0^{N,n}=y_0^{N,n}$,
where uniqueness follows from local Lipschitz continuity.
Using classical comparison results from the theory of ODEs,
the above computation yields for all $t\in[0,\infty)$ and all $N,n\in\N$
that $y_t^{N,n}\leq z_t^{N,n}$ and for all $N,n\in\N$ that
$\sup_{t\in[0,\infty)}z_t^{N,n}=\max\Big\{z_0^{N,n},\tfrac{C^n}{\min\left\{\bar{\kappa}_P+\nu,\frac{\lambda-\nu}{2}\right\}}\Big\}$.
Hence, we obtain for every $n\in\N$ that
\begin{equation} \begin{split}
  &\sup\limits_{N\in\N}\sup\limits_{t\in[0,\infty)}
  \E\left[\sum_{i\in\D_n}\sigma_i\left(c_0\tfrac{1}{P_t^N(i)}+\tfrac{1}{P_t^N(i)H_t^N(i)}\right)\right]
  =
  \sup\limits_{N\in\N}\sup\limits_{t\in[0,\infty)}
  y_t^{N,n}
  \leq
  \sup\limits_{N\in\N}\sup\limits_{t\in[0,\infty)}z_t^{N,n}
  \\
  &\qquad
  =
  \max\left\{\sup\limits_{N\in\N}z_0^{N,n},\tfrac{C^n}{\min\left\{\bar{\kappa}_P+\nu,\frac{\lambda-\nu}{2}\right\}}\right\}
  \\
  &\qquad
  \leq
  \sup\limits_{N\in\N}
  \E\left[\sum_{i\in\D_n}\sigma_i\left(c_0\tfrac{1}{P_0^N(i)}+\tfrac{1}{P_0^N(i)H_0^N(i)}\right)\right]
  +\tfrac{C^n}{\min\left\{\bar{\kappa}_P+\nu,\frac{\lambda-\nu}{2}\right\}}.
\end{split} \end{equation}
Using monotone convergence, we thereby conclude that
\begin{equation} \begin{split}
  &\sup\limits_{N\in\N}\sup\limits_{t\in[0,\infty)}
  \E\left[\sum_{i\in\hat{\D}}\sigma_i\left(c_0\tfrac{1}{P_t^N(i)}+\tfrac{1}{P_t^N(i)H_t^N(i)}\right)\right]
  =
  \lim_{n\to\infty}
  \sup\limits_{N\in\N}\sup\limits_{t\in[0,\infty)}
  \E\left[\sum_{i\in\D_n}\sigma_i\left(c_0\tfrac{1}{P_t^N(i)}+\tfrac{1}{P_t^N(i)H_t^N(i)}\right)\right]
  \\
  &\quad
  \leq
  \lim_{n\to\infty}
  \left(
  \sup\limits_{N\in\N}
  \E\left[\sum_{i\in\D_n}\sigma_i\left(c_0\tfrac{1}{P_0^N(i)}+\tfrac{1}{P_0^N(i)H_0^N(i)}\right)\right]
  +\tfrac{C^n}{\min\left\{\bar{\kappa}_P+\nu,\frac{\lambda-\nu}{2}\right\}}
  \right)
  \\
  &\quad
  =
  \sup\limits_{N\in\N}
  \E\left[\sum_{i\in\hat{D}}\sigma_i\left(c_0\tfrac{1}{P_0^N(i)}+\tfrac{1}{P_0^N(i)H_0^N(i)}\right)\right]
  +\tfrac{\gamma c_0
  +(\gamma+\delta)
  \sup\limits_{N\in\N}\sup\limits_{t\in[0,\infty)}
  \E\left[\sum_{i\in\hat{D}}\sigma_i\tfrac{1}{H_t^N(i)}\right]}{\min\left\{\bar{\kappa}_P+\nu,\frac{\lambda-\nu}{2}\right\}},
\end{split} \end{equation}
finishing the proof of Lemma \ref{lem:bound.norm.1/P}.
\end{proof}

\section*{Acknowledgement}
This paper has been partially supported by the DFG Priority Program ``Probabilistic Structures in Evolution'' (SPP 1590), grants HU 1889/3-2 and ME 3134/6-2.
\def\cprime{$'$}
  \hyphenation{Sprin-ger}\def\polhk$1{\setbox0=\hbox{$1}{\ooalign{\hidewidth
  \lower1.5ex\hbox{`}\hidewidth\crcr\unhbox0}}} \def\cprime{$'$}

\end{document}